\documentclass[11pt,reqno]{amsart}

\usepackage[T1]{fontenc}							
\usepackage[english]{babel}		
\usepackage{csquotes}

\usepackage{cite}					

\usepackage{amssymb, amsmath, amsthm}
\usepackage{mathtools}				      
\usepackage{enumitem}			
\usepackage{bm}							  

\newcommand{\N}{\ensuremath{\mathbb{N}}}   
\newcommand{\Z}{\ensuremath{\mathbb{Z}}}   
   
\newcommand{\R}{\ensuremath{\mathbb{R}}}

\newcommand{\bbV}{\mathbb{V}}
\newcommand{\bbH}{\mathbb{H}}

\newcommand{\rmL}{\mathrm L}
\newcommand{\rmH}{\mathrm H}
\newcommand{\rmW}{\mathrm W}
\newcommand{\rmC}{\mathrm C}
\newcommand{\rmCb}[1]{\mathrm{C}^{#1}_\mathrm{b}}

\newcommand{\BUC}{\mathrm{BUC}}

\newcommand{\calD}{\mathcal{D}}				
\newcommand{\bbT}{{\mathbb{T}^d}}			
\newcommand{\ofbbT}{}                 
\newcommand{\ofbbTd}{}                 

\newcommand{\Grad}{\nabla}
\newcommand{\Lap}{\Delta}
\DeclareMathOperator{\Div}{div}

\newcommand{\ddt}{\frac{\mathrm d}{\mathrm dt}}

\newcommand{\delt}{\partial_t}

\newcommand{\delxk}{\partial_{x_k}}
\newcommand{\D}{\mathrm D}

\newcommand{\dtau}{\;\mathrm d\tau}
\newcommand{\dx}{\;\mathrm dx}

\newcommand{\Id}{\mathbf{I}}									

\newcommand{\tot}{\textnormal{tot}}
\newcommand{\free}{\textnormal{free}}
\newcommand{\kin}{\textnormal{kin}}

\newcommand{\vel}{\mathbf{v}}									
\newcommand{\uel}{\mathbf{u}}									
\newcommand{\w}{\mathbf{w}}										

\newcommand{\Str}{\mathbf{S}}									

\newcommand{\z}{\mathbf{z}}	
\newcommand{\f}{\mathbf{f}}		
		
\newcommand{\F}{\mathbf{F}}		
\newcommand{\G}{\mathbf{G}}		
\newcommand{\Rcal}{\mathcal{R}}		

\newcommand{\tphi}[1]{\tilde\varphi_{#1,0}}		
\newcommand{\tnu}[1]{\tilde\nu_{#1,0}}		
\newcommand{\tlambda}[1]{\tilde\lambda_{#1,0}}		
\newcommand{\trho}[1]{\tilde\rho_{#1,0}}		
\newcommand{\barrho}[1]{\bar\rho_{#1}}		
\newcommand{\ta}[1]{\tilde{a}_{#1,0}}
\newcommand{\tant}{\tilde{a}_0}		
\newcommand{\tb}[1]{\tilde{b}_{#1,0}}
\newcommand{\tbnt}{\tilde{b}_0}		
\newcommand{\tdnt}{\tilde{d}_0}		
\newcommand{\barphi}[1]{\bar\varphi_{#1,0}}		

\newcommand{\prt}[1]{( #1 )}										
\newcommand{\bigprt}[1]{\big( #1 \big)}
\newcommand{\Bigprt}[1]{\Big( #1 \Big)}
\newcommand{\biggprt}[1]{\bigg( #1 \bigg)}
\newcommand{\abs}[1]{| #1 |}										
\newcommand{\bigabs}[1]{\big| #1 \big|}
\newcommand{\Bigabs}[1]{\Big| #1 \Big|}

\newcommand{\norm}[1]{\| #1 \|}										
\newcommand{\bignorm}[1]{\big\| #1 \big\|}
\newcommand{\Bignorm}[1]{\Big\| #1 \Big\|}

\newcommand{\ang}[2]{ \langle #1 , #2  \rangle}			            
\newcommand{\bigang}[2]{ \big< #1 , #2  \big>}

\newcommand{\bigscp}[2]{\big( #1 , #2 \big)}
\newcommand{\mean}[1]{ \langle #1 \rangle}				            

\newcommand{\ssubset}{\subset\joinrel\subset}		                
\newcommand{\trans}[1]{ #1^{\mathrm T} }

\DeclareMathOperator{\supp}{supp}

\newtheorem{theorem}{Theorem}[section]
\newtheorem{lemma}[theorem]{Lemma}
\newtheorem{proposition}[theorem]{Proposition}

\theoremstyle{definition}
\newtheorem{definition}[theorem]{Definition}

\newtheorem{remark}[theorem]{Remark}

\numberwithin{equation}{section}

\usepackage[colorlinks = true, linkcolor = blue, citecolor=green]{hyperref}

\begin{document}

	\title[Local Well-Posedness of a Diffuse Interface Model for Two-Phase Flows]{Local Well-Posedness for a Diffuse Interface Model for Two-Phase Flows from Mixture Theory}
	\author{Helmut Abels}
	\author{Harald Garcke}	
	\author{Julia Wittmann}
	\address{Fakultät für Mathematik, Universität Regensburg \\ D--93040 Regensburg, Deutschland}
	\email{helmut.abels@ur.de}
	\email{harald.garcke@ur.de}
	\email{julia4.wittmann@ur.de}
	\dedicatory{In memory of Hermann Sohr}
	
	\begin{abstract}
        Local-in-time well-posedness is established for a recently proposed diffuse interface model describing incompressible two-phase flows. The result constitutes the first analytical study of a model introduced by ten Eikelder et al.~for the motion of a binary mixture of macroscopically immiscible, viscous, incompressible fluids with unmatched densities. In contrast to classic diffuse interface models based on a single mean velocity, this model is derived within the framework of mixture theory, assigning each phase its own momentum and mass balance, which results in a system of two coupled Navier--Stokes equations and two mass transport equations. The proof of the well-posedness result uses a fixed-point strategy, where the main difficulty lies in the analysis of the principal part of the associated linearized system.


	\end{abstract}

	\maketitle

	\section{Introduction} \label{sec:intro}

    The mathematical modeling of two-phase flows is a fundamental topic in fluid dynamics, motivated by numerous applications in physics, engineering, chemistry, and biology. A key challenge in the derivation of such models is the accurate description of the kinetic energy of the two fluid components.

    Over the years, several diffuse interface models have been proposed to describe the dynamics of two macroscopically immiscible, incompressible fluids, where partial miscibility is assumed within a thin interfacial layer. In the classic setting, these models are based on the kinetic energy associated with a single mean velocity that is obtained by averaging the velocities of the individual fluid phases. Consequently, the momentum evolution is governed by a single linear momentum balance formulated with respect to this mean velocity.
    
    However, this reduction of complexity comes at the cost of neglecting the kinetic energy arising from relative motions of the two fluids. This phenomenon naturally occurs on the diffuse interface, where both phases coexist and their velocity fields can point in different directions.

    To account for this effect, ten Eikelder, van der Zee, and Schillinger~\cite{tenEikelder2024} recently proposed a diffuse interface model that incorporates the full kinetic energy of each fluid constituent. This leads to a system including separate linear momentum balances, one for each component; making it the first diffuse interface model for two-phase flows that features distinct velocity fields. The authors completed their system by adding mass transport equations for each constituent. The resulting model is thermodynamically consistent and---the first of its kind---fully compatible with the continuum theory of mixtures. In this paper, we establish local-in-time well-posedness for this model.

    \medskip
    \noindent\textbf{The model.}
	The model suggested by ten Eikelder et al.~\cite{tenEikelder2024} describes the dynamics of an isothermal mixture composed of viscous, incompressible Newtonian fluids with different densities. It is formulated in a general multi-phase setting, allowing for phase transitions and the presence of external forces. 
    In the present work, we focus on the case of two fluid phases without mass sources. Moreover, for simplicity, we neglect external forces. However, since their inclusion would require only minimal modifications in our analysis, they could readily be taken into account.
    
    Under these assumptions, the original model reduces to a formulation similar to that presented in the review article~\cite{Abels2024}. The resulting system of two coupled Navier--Stokes equations and two mass transport equations reads
    \begin{equation}
		\label{int:eq:system_conservative}
		\arraycolsep=2pt
		\begin{array}{rcll}
			\delt\bigprt{\hat\rho_1(\varphi_1)\vel_1} +\Div\bigprt{\hat\rho_1(\varphi_1)\vel_1 \otimes \vel_1} \qquad\quad\\[0.2ex] 
			- \Div\hat\Str_1(\varphi_1,\D\vel_1)+ \varphi_1\Grad\bigprt{\hat p+\hat\mu_1(\varphi_1,\varphi_2) } &=& \hat\Rcal(\varphi_1,\varphi_2)(\vel_2-\vel_1)
			&\text{ in $(0,T)\times\bbT$}, \\[1ex]
			\delt\bigprt{\hat\rho_2(\varphi_2)\vel_2} +\Div\bigprt{\hat\rho_2(\varphi_2)\vel_2 \otimes \vel_2} \qquad\quad\\[0.2ex] 
			- \Div\hat\Str_2(\varphi_2,\D\vel_2) + \varphi_2\Grad\bigprt{\hat p+\hat\mu_2(\varphi_1,\varphi_2)} &=& \hat\Rcal(\varphi_1,\varphi_2)(\vel_1-\vel_2)
			&\text{ in $(0,T)\times\bbT$}, \\[1ex]
			\delt\hat\rho_1(\varphi_1) + \Div\bigprt{\hat\rho_1(\varphi_1)\vel_1} &=& 0
			&\text{ in $(0,T)\times\bbT$}, \\[1ex]
			\delt\hat\rho_2(\varphi_2) + \Div\bigprt{\hat\rho_2(\varphi_2)\vel_2} &=& 0
			&\text{ in $(0,T)\times\bbT$}, \\[1ex]
			(\vel_1,\vel_2,\varphi_1,\varphi_2) \vert_{t=0} &=& (\vel_{1,0},\vel_{2,0},\varphi_{1,0},\varphi_{2,0})
			&\text{ on $\bbT$},
		\end{array}
	\end{equation}
    where
    \begin{align*}
        \hat\mu_j(\varphi_1,\varphi_2) &= \partial_{\varphi_j}\hat F(\varphi_1,\varphi_2) - \sigma_j\Lap\varphi_j,
        \qquad j=1,2.
	\end{align*}
    We impose periodic boundary conditions throughout. Accordingly, all spatially dependent functions are defined on the $d$-dimensional torus $\bbT = \R^d/2\pi\Z^d$, where $d\in\{2,3\}$. Moreover,~$T>0$ denotes a given final time.
    
    Here and throughout the paper, for $j=1,2$, the index $j$ distinguishes the following quantities associated with the respective component of the mixture. Each constituent is characterized by a velocity field $\vel_j$ and a phase field variable $\varphi_j$, which represents the concentration of the respective fluid and physically takes values between $0$ and $1$. In contrast, the model contains only a single mechanical pressure $\hat p$, which enforces the constraint of no excess volume due to mixing, that is,
	\begin{align}
		\label{int:eq:phi1+phi2=1}
		\varphi_1+\varphi_2=1 
		\qquad\text{ in $(0,T)\times\bbT$}.
	\end{align}
    
    Depending on these unknowns, we define, for $j=1,2$, the determined quantities
	\begin{align*}
		\hat\rho_j(\varphi_j) &= \bar\rho_j\varphi_j, \\[1ex]
        \hat\Str_j(\varphi_j,\D\vel_j) &= 2\hat\nu_j(\varphi_j) \D\vel_j + \hat\lambda_j(\varphi_j) \Div\vel_j \Id.
	\end{align*}
    Here, $\hat\rho_j$ denotes the mass density, which depends linearly on $\varphi_j$ via a constant specific mass density $\bar\rho_j>0$. Furthermore, $\hat\Str_j$ is the viscous stress tensor with viscosity coefficients $\hat\nu_j$ and $\hat\lambda_j$, where $\D\vel_j$ denotes the symmetrized gradient of $\vel_j$ and $\Id$ the identity matrix in~$\R^d$. Moreover, $\hat\mu_j$ describes the chemical potential with a homogeneous free energy density $\hat F$ and a surface tension coefficient $\sigma_j$. Being more general than \cite{tenEikelder2024} and \cite{Abels2024}, we choose a formulation of the model that involves a homogeneous free energy density $\hat F(\varphi_1,\varphi_2)$ instead of an additive form $\hat F_1(\varphi_1) + \hat F_2(\varphi_2)$. Finally, $\hat{\Rcal}>0$ represents the rate of linear momentum exchange induced by the relative motion of the two fluids.
    
	The total energy of the system is given by the sum of the kinetic energy and a free energy of Ginzburg--Landau type
	\begin{align*}
		\hat E_\tot\prt{\vel_1,\vel_2,\varphi_1,\varphi_2} 
		&= \hat E_\kin(\vel_1,\vel_2,\varphi_1,\varphi_2) + \hat E_\free(\varphi_1,\varphi_2) \\
		&= \sum_{j=1,2} \int_\bbT \hat\rho_j(\varphi_j) \frac{\abs{\vel_j}^2}{2} \dx + \int_\bbT \hat F(\varphi_1,\varphi_2) + \sum_{j=1,2} \sigma_j \frac{\abs{\Grad\varphi_j}^2}{2} \dx.
	\end{align*}

    \medskip
	\noindent\textbf{Main result.} 
	 For our analysis, we reformulate the system~\eqref{int:eq:system_conservative} in such a way that $\varphi_2$ no longer appears explicitly, which is feasible in view of the constraint \eqref{int:eq:phi1+phi2=1}. Moreover, for analytical reasons, we replace the transport equation for the mass density $\hat\rho_1(\varphi_1)$ by the equivalent transport equation for the phase field $\varphi_1$. Using~\eqref{int:eq:phi1+phi2=1}, one derives the divergence equation listed as the fourth equation of the system. Finally, rewriting the momentum equations in non-conservative form leads to the equivalent formulation
	\begin{equation}
		\label{int:eq:system}
		\arraycolsep=2pt
		\begin{array}{rcll}
			\rho_1(\varphi_1)\delt\vel_1 +\bigprt{\rho_1(\varphi_1)\vel_1 \cdot \nabla} \vel_1 \qquad\qquad\qquad\\[0.2ex] 
			- \Div\Str_1(\varphi_1,\D\vel_1)+ \varphi_1\Grad\bigprt{p+\mu_1(\varphi_1) } &=& \Rcal(\varphi_1)(\vel_2-\vel_1)
			&\text{ in $(0,T)\times\bbT$}, \\[1ex]
			\rho_2(\varphi_1)\delt\vel_2 +\bigprt{\rho_2(\varphi_1)\vel_2 \cdot \nabla} \vel_2 \qquad\qquad\qquad\\[0.2ex] 
			- \Div\Str_2(\varphi_1,\D\vel_2) + (1-\varphi_1)\Grad\bigprt{p+\mu_2(\varphi_1)} &=& \Rcal(\varphi_1)(\vel_1-\vel_2)
			&\text{ in $(0,T)\times\bbT$}, \\[1ex]
			\delt\varphi_1 + \Div\prt{\varphi_1\vel_1} &=& 0
			&\text{ in $(0,T)\times\bbT$}, \\[1ex]
			\Div\prt{\varphi_1\vel_1 + (1-\varphi_1)\vel_2} &=& 0
			&\text{ in $(0,T)\times\bbT$}, \\[1ex]
			(\vel_1,\vel_2,\varphi_1) \vert_{t=0} &=& (\vel_{1,0},\vel_{2,0},\varphi_{1,0})
			&\text{ on $\bbT$},
		\end{array}
	\end{equation}
	where 
	\begin{align*}
		\rho_1(\varphi_1) &= \bar\rho_1\varphi_1, &\quad
		\rho_2(\varphi_1) &= \bar\rho_2(1-\varphi_1), \\
		\mu_1(\varphi_1) &= F'(\varphi_1) - \sigma_1\Lap\varphi_1, &\quad
		\mu_2(\varphi_1) &= -F'(\varphi_1) + \sigma_2\Lap\varphi_1, 
    \end{align*}
    and
    \begin{align}
        \label{int:eq:S_j}
		\Str_j(\varphi_1,\D\vel_j) &= 2\nu_j(\varphi_1) \D\vel_j + \lambda_j(\varphi_1) \Div\vel_j \Id, \qquad j=1,2.
	\end{align}
    Note carefully that, compared to model~\eqref{int:eq:system_conservative}, we have redefined the homogeneous free energy density and the pressure such that
    \begin{align*}
        F(\varphi_1) &= \frac12 \hat F(\varphi_1,1-\varphi_1), \\
        p &= \hat p + \frac12 \bigprt{ \partial_{\varphi_1}\hat F(\varphi_1,1-\varphi_1) + \partial_{\varphi_2}\hat F(\varphi_1,1-\varphi_1) }.
    \end{align*}
	Moreover, we have changed the notation by setting $\Rcal(\varphi_1)=\hat \Rcal(\varphi_1,1-\varphi_1)=\hat \Rcal(\varphi_1,\varphi_2)$ and
	\begin{alignat*}{2}
		f_1(\varphi_1)&=\hat f_1(\varphi_1) &\qquad\text{for } f_1\in\{\rho_1, \nu_1, \lambda_1 \}, \\
		f_2(\varphi_1)&=\hat f_2(1-\varphi_1)= \hat f_2(\varphi_2) &\qquad\text{for } f_2\in\{\rho_2, \nu_2, \lambda_2 \}.
	\end{alignat*}
	With this notation, the total energy of the system~\eqref{int:eq:system} takes the form
	\begin{align*}
		E_\tot\prt{\vel_1,\vel_2,\varphi_1} 
		&= E_\kin(\vel_1,\vel_2,\varphi_1) + E_\free(\varphi_1) \\
		&= \sum_{j=1,2} \int_\bbT \rho_j(\varphi_1) \frac{\abs{\vel_j}^2}{2} \dx + \int_\bbT 2 F(\varphi_1) + (\sigma_1+\sigma_2) \frac{\abs{\Grad\varphi_1}^2}{2} \dx.
	\end{align*}
	
	The main contribution of this paper is the following local-in-time strong well-posedness result for the system from \cite{tenEikelder2024} in the formulation~\eqref{int:eq:system}.
	 
	\begin{theorem}
	   \label{int:thm:main_res}
	   Let $\vel_{j,0}\in\rmH^1(\bbT)^d$, $j=1,2$, and $\varphi_{1,0}\in\rmH^2(\bbT)$ with $\varphi_{1,0}\in(0,1)$ in $\bbT$ such that $\Div(\varphi_{1,0}\vel_{1,0} + (1-\varphi_{1,0}) \vel_{2,0}) = 0$. Under the assumptions~\eqref{ass:torus}--\eqref{ass:R} below, there exists some $T>0$ such that the system~\eqref{int:eq:system} admits a unique strong solution~$(\vel_1,\vel_2,\varphi_1,p)$ with
	   \begin{align*}
	       \vel_j &\in \rmL^2 \prt{0,T;\rmH^2(\bbT)^d} \cap \rmH^1 \prt{ 0,T;\rmL^2(\bbT)^d}, \qquad j=1,2, \\
	       \varphi_1 &\in \rmL^2 \prt{0,T;\rmH^3(\bbT)} \cap \rmH^1 \prt{ 0,T;\rmH^1(\bbT)} \cap \rmH^2 \prt{ 0,T;\rmH^{-1}(\bbT)}, \\
	       p &\in \rmL^2 \prt{0,T;\rmH^1_{(0)}(\bbT)}.
	   \end{align*}
	\end{theorem}

    \medskip
	\noindent\textbf{Background.}
    The prototypical diffuse interface models for incompressible fluid mixtures are of Navier--Stokes/Cahn--Hilliard type. Their origins can be traced back to the seminal work of Hohenberg \& Halperin~\cite{Hohenberg1977}, who introduced the basic model for fluids with identical constant densities in 1977. This model, commonly referred to as \emph{Model H}, is formulated in terms of a \emph{mean velocity}. A rigorous derivation together with a proof of its thermodynamic consistency was later provided by Gurtin et al.~\cite{Gurtin1996}.
    
    The extension of \emph{Model H} to mixtures with unmatched densities has attracted considerable attention over the past decades. Among the most influential is the thermodynamically consistent model by Lowengrub \& Truskinovsky~\cite{LowengrubTruskinovsky1998}, whose approach is based on a quasi-incompressible formulation, where each pure fluid is incompressible, while the mass-averaged velocity is not divergence-free. A different thermodynamically consistent system was introduced by Abels, Garcke \& Grün~\cite{AbelsGarckeGruen2012}, who employed a volume-averaged velocity to obtain an incompressible model. We also refer to \cite{Ding2007, Boyer2002, Heida2011, ShokrpourRoudbari2018} for further diffuse interface models with unmatched densities. A comprehensive unifying framework and comparison of several Navier--Stokes/Cahn--Hilliard models was recently presented by ten Eikelder et al.~\cite{tenEikelder2023}.

    The common feature of all models discussed above is their formulation for a single mean velocity, such that they consist of \emph{one} Navier--Stokes-type equation coupled to a Cahn--Hilliard system.
    Their derivation is based on the conservation of mass for the individual constituents together with the conservation of linear momentum for the mixture as a whole, while neglecting the momentum generated by the relative motion between the constituents. In the terminology of Bothe \& Dreyer~\cite{BotheDreyer2014}, such models are referred to as \emph{reduced mixture models}.
    
    A conceptually different approach is provided by mixture theory, which postulates a complete system of balance laws for each constituent. For a long time, a rigorous diffuse interface formulation for two- or multi-phase flows within this framework was not available. This gap was recently closed by ten Eikelder, van der Zee \& Schillinger~\cite{tenEikelder2024}, who derived a novel thermodynamically consistent diffuse interface model from mixture theory. Their \emph{(full) mixture model} is a multi-phase system that describes each phase by its own Navier--Stokes equations coupled with a corresponding mass transport equation.

    \medskip
    Concerning the mathematical analysis of the aforementioned diffuse interface models, we refer to \cite{AbelsGarckeGiorgini2023} for an overview of the extensive analytical literature on the model by Abels, Garcke \& Grün~\cite{AbelsGarckeGruen2012}. In particular, the first local well-posedness theorem was established by Abels \& Weber~\cite{AbelsWeber2020}, followed by further well-posedness results by Giorgini \cite{Giorgini2021, Giorgini2022}. We also mention various extensions of this model to more general settings and the associated analytical results, including the recent work by Abels et al.~\cite{AbelsGarckePoiatti2024} on multi-phase flows, for instance. 

    In contrast, the analysis of the model by Lowengrub \& Truskinovsky~\cite{LowengrubTruskinovsky1998} is more involved. The main difficulty stems from the stronger coupling between the Navier--Stokes and Cahn--Hilliard equations compared to volume-averaged models, due to the presence of the pressure in the equation for the chemical potential.
    The first analytical result was obtained by Abels~\cite{Abels2009g}, who proved the existence of weak solutions under a $p$-Laplacian regularization and for regular free energy densities. Later, in \cite{Abels2012}, he established local-in-time strong well-posedness for the unregularized system. Concerning weak solutions, partial progress was made by the authors~\cite{AbelsGarckeWittmann2025}, by showing existence for the quasi-stationary version of the closely related model by Aki et al.~\cite{Aki2014} with singular free energies.
    In a recent breakthrough, Fei et al.~\cite{FeiFeiLiuWu2026} achieved the existence of global weak solutions to the original Lowengrub--Truskinovsky model for a class of singular free energy densities without any spatial regularization, by exploiting the Bresch--Desjardins entropy for a viscosity that depends linearly on the mass density.

    To the best of the authors' knowledge, the mixture model derived by ten Eikelder et al.~\cite{tenEikelder2024} has not yet been investigated from an analytical point of view. Thus, the present work provides the first analytical study, establishing local-in-time well-posedness.

    From a technical perspective, our analysis is related to the theory of the Korteweg equations. The latter were derived by Dunn \& Serrin~\cite{DunnSerrin1985} to describe compressible viscous fluids with capillarity effects. For this model, local existence of strong solutions can be found in a work by Kotschote~\cite{Kotschote2008}.

    \medskip
	\noindent\textbf{Proof strategy and analytical challenges.} Our proof of local-in-time well-posedness for system~\eqref{int:eq:system}, which is stated in Theorem~\ref{int:thm:main_res}, follows a fixed-point approach. More precisely, given $T>0$ and suitable Banach spaces $X_T$ and $Y_T$, we decompose the system into its principal linear part and the remaining nonlinear terms. This leads to a linear operator~$\mathcal{L}_T\colon X_T\to Y_T$ representing the principal part of the linearization of \eqref{int:eq:system}, and a nonlinear operator $\mathcal{F}_T\colon X_T\to Y_T$. The original problem is then reformulated as the search for a fixed point of the operator~$\mathcal{L}_T^{-1} \circ \mathcal{F}_T$ in $X_T$. The Banach fixed-point theorem yields existence and uniqueness once two key ingredients have been established: the invertibility of $\mathcal{L}_T$, and the local Lipschitz continuity of $\mathcal{F}_T$. 

    The main analytical challenge of the paper is to prove the invertibility of $\mathcal{L}_T$, that is, to solve the associated linear system
    \begin{equation}
		\label{int:eq:lin_system}
		\begin{aligned}
			\rho_1(\tphi1)\delt\vel_1 - \Div\Str_1(\tphi1,\D\vel_1) + \tphi1\Grad p - \tphi1 \Grad\Lap\varphi_1 &= \f_1, \\[0.2ex] 
			\rho_2(\tphi1)\delt\vel_2 - \Div\Str_2(\tphi1,\D\vel_2) + (1-\tphi1)\Grad p + (1-\tphi1) \Grad\Lap\varphi_1 &= \f_2, \\[0.2ex] 
			\delt\varphi_1 + \Div(\tphi1\vel_1) &= g_1, \\[0.2ex] 
			\Div( \tphi1\vel_1 + (1-\tphi1)\vel_2) &= g_2.
		\end{aligned}
    \end{equation}
    This system couples the incompressible Stokes equations with the linearization of the (compressible) Korteweg equations 
    \begin{align*}
        \delt(\rho\vel) + \Div(\rho\vel\otimes\vel) + \Grad p(\rho) &= \kappa \rho \Grad\Lap\rho, \\[0.2ex] 
        \delt\rho +\Div(\rho\vel) &= 0.
    \end{align*}
    The main novelty of the present work also lies in the analysis of the linear system~\eqref{int:eq:lin_system}. Our approach is based on a suitable regularization of the system, to which abstract parabolic theory can be applied. Thus, the resulting regularized system admits a unique weak solution. The key step is then to derive sufficiently strong \textit{a priori} estimates that are uniform with respect to the regularization parameter.

    We begin by freezing the coefficient $\tphi1$ and establishing the required estimates in the case of constant coefficients. To obtain an estimate for the most delicate term $\Grad\Lap\varphi_1$, we use a carefully chosen testing procedure that leads to a weakly formulated damped plate equation for $\varphi_1$. Once all uniform estimates are in place, passing to the limit in the regularized system gives a strong solution to~\eqref{int:eq:lin_system} in the constant-coefficient setting.

    We then turn to the case of variable coefficients. By adapting the previous testing procedure, we first derive uniform estimates of lower regularity, which provide the existence of a unique weak solution. Finally, the desired higher regularity of this solution is recovered by combining a perturbation argument with localization techniques.

    \medskip
	\noindent\textbf{Outline.}
    This paper is organized as follows. In the beginning, Section~\ref{sec:preliminaries} lays the groundwork for our analysis by introducing the notation, stating the general assumptions, and gathering the necessary preliminary results. In Section~\ref{sec:mr}, we establish our main result, namely the local well-posedness of system~\eqref{int:eq:system}, by means of a fixed-point argument. The main analytical part of this work is contained in Section~\ref{sec:lin}, where the existence of a unique solution to the principal part of the linearized system is shown. Finally, Section~\ref{sec:Lip} is dedicated to verifying the local Lipschitz continuity of the corresponding nonlinear operator.

	\section{Preliminaries} \label{sec:preliminaries}

    In this section, we introduce the notation, state the standing assumptions, and collect several preliminary results that will be needed in the subsequent analysis.
    
	\subsection{Notation} \label{subsec:notation} 
	
	We begin by fixing the notation used throughout this paper.
	
	\medskip
	\noindent\textbf{Notation for function spaces.}
     As spatial manifold, we always consider $\bbT$, $d\in\N$, where the torus is given by $\mathbb{T}\coloneqq\R/2\pi\Z$. We use the standard notations~$\rmL^p(\bbT)$, $p\in [1,\infty]$, for Lebesgue spaces and $\rmW^{k,p}(\bbT)$, $k\in\N_0$, $p\in [1,\infty]$ for Sobolev spaces. Moreover, the notation~$\rmW^{-k,p}(\bbT) = (\rmW^{k,p'}(\bbT))'$ represents the corresponding dual space, where the dual exponent $p'\in [1,\infty]$ of $p$ is such that~$\frac1p + \frac1{p'}=1$. In addition, we introduce the $\rmL^2$-Bessel potential spaces $\rmH^s(\bbT)$,~$s\in\R$, as well as, for~$\alpha\in (0,1)$, the well-known Hölder~spaces~$\rmC^\alpha(\bbT)$. 
    
    Now, let $X$ be some Banach space and let $0 < T < \infty$ be given. In this setting, we denote by~$\rmL^p(0,T;X)$, $p\in [1,\infty]$, the usual Bochner spaces, and we write~$f \in \rmW^{k,p}(0,T;X)$, $k\in\N_0$, $p\in[1,\infty)$, if and only if $f,\frac{\mathrm{d} f}{\mathrm{d} t}, \dots, \frac{\mathrm{d}^k f}{\mathrm{d} t^k} \in \rmL^p(0,T;X)$, where~$\frac{\mathrm{d}^l f}{\mathrm{d} t^l}$ refers to the $l$th $X$-valued distributional derivative of~$f$. Furthermore, we set $\rmH^k(0,T;X) = \rmW^{k,2}(0,T;X)$, where $k\in\N_0$. In addition, $\rmC^\alpha([0,T];X)$, $\alpha\in (0,1)$, represents the respective Hölder spaces with values in $X$, and $\BUC([0,T];X)$ denotes the space of all bounded and uniformly continuous functions~$f\colon [0,T]\to X$ equipped with the supremum norm. We finally use the notation~$\rmC_0^\infty(0,T;X)$ for the vector space of all smooth functions~$f\colon (0,T) \to X$ with~$\supp f \ssubset (0,T)$. 

    Eventually, we define $\rmCb{k}(\R)$ as the subspace of all bounded functions in $\rmC^k(\R)$.

    \medskip
	\noindent\textbf{Further notation.}
    Introducing further notation, we point out that vector-valued variables and functions are indicated in bold fonts. Moreover, the outer product of $\mathbf{a},\mathbf{b}\in\R^d$ is given by $\mathbf{a}\otimes\mathbf{b}=(a_ib_j)_{i,j=1}^d$. For a function $\uel\in \rmH^1(\bbT)^d$, the notation $\D\uel \coloneqq \frac12 \bigprt{ \Grad\uel + \trans{\Grad\uel} }$ refers to its symmetrized gradient.
	
	For any Banach space $X$ and its dual $X'$, the duality pairing between elements $x'\in X'$ and $x\in X$ is denoted by $\ang{x'}{x}_{X',X}$. If $X$ is a Hilbert space, we write $(\cdot,\cdot)_X$ for its inner product on $X$. 
	Additionally,
	\begin{align*}
		\mean{f} \coloneqq \frac{1}{|\bbT|}\ang{f}{1}_{\rmH^1(\bbT)', \rmH^1(\bbT)} \quad\text{for } f\in \rmH^1(\bbT)'
	\end{align*}
	represents the generalized spatial mean of $f$, where $|\bbT|=(2\pi)^d$ is the $d$-dimensional Lebesgue measure of $\bbT$. With the standard identification $\rmL^2(\bbT) \subset \rmH^1(\bbT)'$, we have the identity~$\mean{f} = \frac{1}{|\bbT|} \int_\bbT f \dx$ if~$f\in \rmL^2(\bbT)$.
	Finally, given $m\in\R$, we set 
	\begin{align*}
		\rmL^p_{(m)}(\bbT) \coloneqq \{u\in \rmL^p(\bbT) \colon \mean{u} = m \}, \quad p\in[1,\infty],
	\end{align*}
	and $\rmH^k_{(m)}(\bbT) \coloneqq \rmH^k(\bbT) \cap \rmL^2_{(m)}(\bbT)$, $k\in\N$.

    For the sake of brevity, from Section~\ref{sec:mr} onward, we omit the notation of the space $\bbT$ in the index of all norms and dual pairings.

	\subsection{General Assumptions.} 
	
	Throughout this paper, we impose the following main assumptions.
	
	\begin{enumerate}[itemsep=0.8ex, label=$(\mathbf{A \arabic*})$, ref = $\mathbf{A \arabic*}$]
		\item \label{ass:torus}
		Any spatial variable is defined on $\bbT$ in dimension $d\in\{2,3\}$, where the torus is given by $\mathbb{T} \coloneqq \R/2\pi\Z$.
		\item \label{ass:pot}
		The potential $F\colon\R\to\R$ is of regularity $F\in \rmC^3(\R)$.
		\item \label{ass:coeffs}
		The viscosity coefficients $\nu_j,\lambda_j\in \rmCb{2}(\R)$, $j=1,2$, are uniformly positive, i.e., there exists a constant~$K_*>0$ such that
		\begin{align*}
			\nu_j(s), \lambda_j(s) \geq K_* \quad\text{ for all } s\in\R. 
		\end{align*}
        \item \label{ass:surface_tension}
		The surface tension coefficients are set to $\sigma_1=\sigma_2=1$ for simplicity.
		\item \label{ass:R}
		The rate of exchange of linear momentum $\mathcal R\colon\R\to\R$ is locally Lipschitz continuous.
	\end{enumerate}

	\subsection{Preliminary Results} \label{subsec:prelim_results}
	
	We conclude this section by collecting several preliminary results that play an important role in our analysis.
	
	\medskip
	\noindent\textbf{Composition of functions.}
	Concerning compositions of functions, we state the following lemma, whose formulation is taken from \cite[Theorem 3]{AbelsWeber2020} or \cite[Theorem 2.1]{AbelsHaselboeck2026}.
	
	\begin{lemma}[Composition with Sobolev functions]
		\label{prelim:lemma:comp_Sobolev}
		Let $\Omega\subset\R^d$ be a bounded domain with boundary of class $\rmC^1$, and let $m,N\in\N$ and $p\in[1,\infty)$ such that $m-\frac{d}{p}>0$. Then, for every $f\in\rmC^m(\R^N)$ and every $R>0$, there exists a constant $C(R)>0$ such that for all~$u\in\rmW^{m,p}(\Omega)^N$ with $\norm{u}_{\rmW^{m,p}(\Omega)^N}\leq R$, it holds $f(u)\in \rmW^{m,p}(\Omega)$ and
		\begin{align*}
			\norm{f(u)}_{\rmW^{m,p}(\Omega)} \leq C(R). 
		\end{align*}
		Moreover, if $f\in\rmC^{m+1}(\R^N)$, then for all $R>0$, there exists a constant $L(R)>0$ such that
		\begin{align*}
			\norm{f(u)-f(v)}_{\rmW^{m,p}(\Omega)} 
			\leq L(R) \norm{u-v}_{\rmW^{m,p}(\Omega)^N} 
		\end{align*}
		for all $u,v\in \rmW^{m,p}(\Omega)^N$ with $\norm{u}_{\rmW^{m,p}(\Omega)^N}, \norm{v}_{\rmW^{m,p}(\Omega)^N} \leq R$.
	\end{lemma}
	
	\begin{proof}
		The first assertion follows from \cite[Section 5.2.4, Theorem 1 and Lemma]{RunstSickel1996}, while the second one can be directly deduced from the first.
	\end{proof}
	
	\begin{remark}
		Lemma~\ref{prelim:lemma:comp_Sobolev} remains true if the domain $\Omega$ is replaced by the torus $\bbT$.
	\end{remark}

	\medskip
	\noindent\textbf{Interpolation results.}
	By $(X_0,X_1)_{\theta,r}$ we denote the real interpolation space between two Banach spaces
	$X_0$ and $X_1$ with exponent $\theta$ and summation index $r$. For a densely injected Banach couple $X_1\hookrightarrow X_0$ and $T>0$, we have the continuous embedding, cf.~\cite[Chapter~III, Theorem~4.10.2]{Amann1995},
	\begin{align}
		\label{prelim:interpol:BUC}
		\rmL^2\prt{0,T;X_1} \cap \rmH^1\prt{ 0,T;X_0} 
		\hookrightarrow \BUC\bigprt{ [0,T]; \prt{ X_0, X_1}_{\frac12, 2} }.
	\end{align}
	Moreover, for any Banach spaces  $X_0 \subset Y \subset X_1$ satisfying for some $\theta\in(0,1)$ and some constant $C>0$ that $\norm{x}_Y \leq C\norm{x}_{X_0}^{1-\theta} \norm{x}_{X_1}^\theta$ for all $x\in X_0$, it holds
	\begin{align}
		\label{prelim:interpol:Hölder}
		\rmC^{\alpha}\bigprt{ [0,T]; X_1 } \cap \rmL^\infty\prt{0,T;X_0}
		\hookrightarrow \rmC^{\alpha\theta}\bigprt{ [0,T]; Y }.
	\end{align}
	This result is well-known, cf.~\cite[Lemma 1]{AbelsWeber2020}.

	\medskip
	\noindent\textbf{Abstract parabolic evolution equations.}
	We state an existence result for abstract parabolic evolution equations, whose formulation is taken from \cite{AbelsDolzmannLiu2014}.
	Let $\bbV$ and $\bbH$ be two separable Hilbert spaces with a continuous and dense embedding $\bbV\hookrightarrow\bbH$, and fix $T\in(0,\infty)$. For all $t\in[0,T]$, suppose that a bilinear form $a(t,\cdot,\cdot)\colon\bbV\times\bbV\to\R$ is given, satisfying the following assumptions for all $v,w\in\bbV$:
	\begin{enumerate}[label=\textnormal{(\alph*)}]
		\item \label{prelim:APEE:ass_a}
		$a(\cdot,v,w)$ is measurable on $[0,T]$,
		\item \label{prelim:APEE:ass_b}
		there exists a constant $C>0$ independent of $t$, $v$, and $w$ such that
		\begin{align*}
			\abs{a(t,v,w)} \leq C\norm{v}_\bbV \norm{w}_\bbV \quad\text{for all } t\in[0,T],
		\end{align*}
		\item \label{prelim:APEE:ass_c}
		there exist constants $C_0,C_1\geq0$ independent of $t$ and $v$ such that
		\begin{align*}
			a(t,v,v) \geq C_0\norm{v}_\bbV^2 - C_1\norm{v}_\bbH^2 \quad\text{for all } t\in[0,T],
		\end{align*}
		\item \label{prelim:APEE:ass_d} 
		$a(\cdot,v,w)$ is diﬀerentiable, $a(\cdot,v,w)$ is continuous in $[0,T]$ and $\delt a(t,v,w)$ is
		measurable on $[0,T]$ with $\abs{\delt^k a(t,v,w)} \leq c\norm{v}_\bbV \norm{w}_\bbV$ for $k\in\{0,1\}$ with $c>0$ independent of $t$.           
	\end{enumerate}

    Under these assumptions, the following result holds true.
	
	\begin{proposition}
		\label{prelim:prop:APEE}
		Let the assumptions \ref{prelim:APEE:ass_a}--\ref{prelim:APEE:ass_c} hold. Then there exists a representation operator $L(t)\colon\bbV\to\bbV'$ with $a(t,v,w) = \ang{L(t)v}{w}_{\bbV',\bbV}$ for all $v,w\in\bbV$, which is continuous and
		linear for fixed $t$. Moreover, for every $f\in\rmL^2(0,T;\bbV')$ and $y_0\in\bbH$, there exists a unique solution 
		\begin{align*}
			y\in \big\{z\colon[0,T]\to\bbH \colon z\in\rmL^2(0,T;\bbV), \delt z\in\rmL^2(0,T;\bbV') \big\}
		\end{align*}
		to the equation
		\begin{align*}
			\delt y + L(t)y = f \quad\text{in } \bbV' \text{ for a.e. } t\in(0,T)
		\end{align*}
		with initial condition $y(0)=y_0$. If additionally \ref{prelim:APEE:ass_d} holds and $y_0\in\bbV$, then the operator~$L\colon \rmH^1(0,T;\bbV)\to \rmH^1(0,T;\bbV')$ is~continuous, and for every $f\in\rmH^1(0,T;\bbV')$ satisfying the compatibility condition~$f(0)-L(0)y_0\in\bbH$, the solution $y$ has the regularity properties
		\begin{align*}
			y\in\rmH^1(0,T;\bbV) \quad\text{and}\quad \partial_{tt} y\in\rmL^2(0,T;\bbV').
		\end{align*}
	\end{proposition}
	
	For a proof, we refer to \cite[Lemma 26.1 and Theorem 27.2]{Wloka1987}.

	\medskip
	\noindent\textbf{Very weak solutions to elliptic equations.}
    As a final preliminary result, we present the following proposition, whose proof is based on a duality argument and which yields a notion of very weak solutions for certain elliptic PDEs on the torus.
	
	\begin{proposition}
		\label{prelim:prop:very_weak_sol}
		Let $p\in (1,\infty)$ and suppose that $a\in\rmW^{1,q}(\bbT)$, where $q>d$ and $q\geq p'$, is such that $a(x)>0$ for all~$x\in\bbT$. Then, for every~$f\in \bigprt{\rmW^{2,p'}_{(0)}(\bbT)}'$, there exists a unique~$u\in \rmL^p_{(0)}(\bbT)$ that satisfies
		\begin{align*}
			(u,-\Div(a\Grad\phi))_{\rmL^2(\bbT)}
			= \ang{f}{\phi}_{\bigprt{\rmW^{2,p'}_{(0)}(\bbT)}', \rmW^{2,p'}_{(0)}(\bbT)}
			\quad\text{for all } \phi\in \rmW^{2,p'}_{(0)}(\bbT),
		\end{align*}
		which we refer to as \textnormal{very weak solution} to the elliptic equation $-\Div(a\Grad u)=f$. Moreover, there exists a constant $C_p>0$ such that
		\begin{align}
			\label{prelim:eq:very_weak_sol_est}
			\norm{u}_{\rmL^p(\bbT)} \leq C_p \norm{f}_{\bigprt{\rmW^{2,p'}_{(0)}(\bbT)}'}.
		\end{align}
	\end{proposition}
	
	\begin{proof}
		We start the proof by recalling that, under the assumptions made in the Proposition, the operator~$-\Div(a\Grad\cdot)\colon \rmW^{2,p'}_{(0)}(\bbT) \to \rmL^{p'}_{(0)}(\bbT)$ is an isomorphism, and so is its adjoint~$(-\Div(a\Grad\cdot))' \colon \rmL^p_{(0)}(\bbT) \to \bigprt{\rmW^{2,p'}_{(0)}(\bbT)}'$. Thus, for every~$f\in \bigprt{\rmW^{2,p'}_{(0)}(\bbT)}'$, there exists a unique solution $u\in \rmL^p_{(0)}(\bbT)$ to the equation $(-\Div(a\Grad u))'=f$. By the definition of the adjoint, it follows
		\begin{align*}
			\bigprt{u,-\Div(a\Grad\phi)}_{\rmL^2(\bbT)}
			= \bigang{(-\Div(a\Grad u))'}{\phi}_{\bigprt{\rmW^{2,p'}_{(0)}(\bbT)}', \rmW^{2,p'}_{(0)}(\bbT)}
			= \ang{f}{\phi}_{\bigprt{\rmW^{2,p'}_{(0)}(\bbT)}', \rmW^{2,p'}_{(0)}(\bbT)}
		\end{align*}
		for all $\phi\in \rmW^{2,p'}_{(0)}(\bbT)$. Finally, the continuity of $\bigprt{(-\Div(a\Grad\cdot))'}^{-1}$ implies estimate~\eqref{prelim:eq:very_weak_sol_est}.
	\end{proof}

	\section{Proof of the Main Result} \label{sec:mr}

    This section is dedicated to the proof of Theorem~\ref{int:thm:main_res}. The strategy is based on reformulating the system~\eqref{int:eq:system} as a fixed-point problem, for which the existence of a unique solution is equivalent to the unique solvability of the original system.
	
	\medskip
	\noindent \textbf{The spaces.} For the fixed-point argument, it is crucial to determine the function spaces for the solution. First, for any $0<T<\infty$, we set 
    \begin{equation}
        \label{mr:eq:spaces_Z_T}
    	\begin{aligned}
    		Z_T^j &\coloneqq \rmL^2 \prt{0,T;\rmH^2(\bbT)^d} \cap \rmH^1 \prt{ 0,T;\rmL^2(\bbT)^d}, \qquad j=1,2, \\
    		Z_T^3 &\coloneqq \rmL^2 \prt{0,T;\rmH^3(\bbT)} \cap \rmH^1 \prt{ 0,T;\rmH^1(\bbT)} \cap \rmH^2 \prt{ 0,T;\rmH^{-1}(\bbT)}, \\
    		Z_T^4 &\coloneqq \rmL^2 \prt{0,T;\rmH^1_{(0)}(\bbT)},
    	\end{aligned}
    \end{equation}
	and we denote the product space by $Z_T\coloneqq Z_T^1 \times Z_T^2 \times Z_T^3 \times Z_T^4$. The individual spaces are equipped with the norms $\norm{\cdot}_{Z_T^j}$, $\norm{\cdot}_{Z_T^3}$, and $\norm{\cdot}_{Z_T^4}$, respectively, given by
	\begin{align*}
		\norm{\vel}_{Z_T^j} &\coloneqq \norm{\vel}_{\rmL^2 \prt{0,T;\rmH^2\ofbbTd}} 
		+ \norm{\delt\vel}_{\rmL^2 \prt{ 0,T;\rmL^2\ofbbTd}}
		+ \norm{\vel\vert_{t=0}}_{\rmH^1\ofbbTd}, 
		\qquad j=1,2, \\
		\norm{\varphi_1}_{Z_T^3} &\coloneqq \norm{\varphi_1}_{\rmL^2 \prt{0,T;\rmH^3\ofbbT}\!} 
		+ \norm{\delt\varphi_1}_{\rmL^2 \prt{ 0,T;\rmH^1\ofbbT}\!} 
		+ \norm{\delt^2\varphi_1}_{\rmL^2 \prt{ 0,T;\rmH^{-1}\ofbbT}\!} 
		+ \norm{\varphi_1\vert_{t=0}}_{\rmH^2\ofbbT\!}
        + \norm{\delt\varphi_1\vert_{t=0}}_{\rmL^2\ofbbT\!}, \\
		\norm{p}_{Z_T^4} &\coloneqq \norm{p}_{\rmL^2 \prt{0,T;\rmH^1\ofbbT}}.
	\end{align*}
    The choice of these norms is intended to allow for estimates that are uniform in $0<T<\frac{T_0}{2}$, for any fixed~$0<T_0<\infty$. This is ensured by the following result from \cite[Lemma 2]{AbelsWeber2020}.
	
	\begin{lemma}
		\label{mr:lemma:extension_op}
		Fix $0<T_0<\infty$, and let $\mathcal{X}_0$, $\mathcal{X}_1$ be Banach spaces such that $\mathcal{X}_0 \hookrightarrow \mathcal{X}_1$ densely. For every $0<T<\frac{T_0}{2}$, we define
		\begin{align*}
			\mathcal{X}_T \coloneqq \rmL^2 \prt{0,T;\mathcal{X}_1} \cap \rmH^1 \prt{ 0,T;\mathcal{X}_0}, 
		\end{align*}
		endowed with the norm
		\begin{align*}
			\norm{u}_{\mathcal{X}_T} &\coloneqq \norm{u}_{\rmL^2 \prt{0,T;\mathcal{X}_1}} 
			+ \norm{\delt u}_{\rmL^2 \prt{ 0,T;\mathcal{X}_0}}
			+ \norm{u\vert_{t=0}}_{(\mathcal{X}_0,\mathcal{X}_1)_{\frac12,2}}.
		\end{align*}
		 Then there exists an extension operator $E\colon \mathcal{X}_T \to \mathcal{X}_{T_0}$ and some constant $C>0$ independent of $T$ such that $Eu\vert_{(0,T)}=u$ in $\mathcal{X}_T$ and 
		 \begin{align*}
		 	\norm{Eu}_{\mathcal{X}_{T_0}} \leq C\norm{u}_{\mathcal{X}_T}
		 \end{align*}
	 	for all $u\in\mathcal{X}_T$ and every $0<T<\frac{T_0}{2}$. Furthermore, there exists a constant $\tilde C(T_0)>0$ independent of $T$ such that 
	 	\begin{align*}
	 		\norm{u}_{\BUC([0,T]; (\mathcal{X}_0,\mathcal{X}_1)_{\frac12,2})} \leq \tilde C(T_0)\norm{u}_{\mathcal{X}_T}
	 	\end{align*}
	 	for all $u\in\mathcal{X}_T$ and every $0<T<\frac{T_0}{2}$.
	\end{lemma}
	
	Finally, given initial data $\vel_{j,0}\in\rmH^1(\bbT)^d$, $j=1,2$, and $\varphi_{1,0}\in\rmH^2(\bbT)$, we define the affine subspaces
    \begin{equation}
        \label{mr:eq:spaces_X_T}
	    \begin{aligned}
		  X_T^j &\coloneqq \{\vel_j\in Z_T^j \colon \vel_j\vert_{t=0}=\vel_{j,0}\}, &\qquad j&=1,2, \\
		  X_T^3 &\coloneqq \{\varphi_1\in Z_T^3 \colon \varphi_1\vert_{t=0}=\varphi_{1,0} \}, 
		  &\qquad X_T^4 &\coloneqq Z_T^4,
	   \end{aligned}
    \end{equation}
	inheriting the norms on $Z_T^j$, and we set $X_T\coloneqq X_T^1 \times X_T^2 \times X_T^3 \times X_T^4$. 
    
    Moreover, for any~$0<T<\infty$, we introduce the spaces
	\begin{equation}
        \label{mr:eq:spaces_Y_T}
    	\begin{aligned}
		    Y_T^j &\coloneqq \rmL^2 \prt{0,T;\rmL^2(\bbT)^d}, \qquad j=1,2, \\
		    Y_T^3 &\coloneqq \rmL^2 \prt{0,T;\rmH^1(\bbT)} \cap \rmH^1 \prt{ 0,T;\rmH^{-1}(\bbT)}, \\
            Y_T^4 &\coloneqq \{ g_2\in Y_T^3 \colon g_2\vert_{t=0} = 0 \},
	   \end{aligned}
    \end{equation}
	endowed with the usual norms, as well as $Y_T\coloneqq Y_T^1 \times Y_T^2 \times Y_T^3 \times Y_T^4$. The fact that the initial data $\vel_{j,0}$, $j=1,2$, and $\varphi_{1,0}$ satisfy the condition~$\Div(\varphi_{1,0}\vel_{1,0} + (1-\varphi_{1,0}) \vel_{2,0}) = 0$ requires imposing the constraint $g_2\vert_{t=0} = 0$ in $Y_T^4$.

	\medskip
	\noindent \textbf{The fixed-point problem.}
    Now we determine the linear and the nonlinear operator required for the formulation of the fixed-point problem that is equivalent to the system~\eqref{int:eq:system}. To this end, let initial data $\vel_{j,0}\in\rmH^1(\bbT)^d$, $j=1,2$, and $\varphi_{1,0}\in\rmH^2(\bbT)$ with~$\varphi_{1,0}\in(0,1)$ in~$\bbT$ be given that satisfy~$\Div(\varphi_{1,0}\vel_{1,0} + (1-\varphi_{1,0}) \vel_{2,0}) = 0$. On the one hand, derived from the highest order terms of the linearization of \eqref{int:eq:system}, we define, for~$\z\coloneqq (\vel_1,\vel_2,\varphi_1,p)$ and for any $0<T<\infty$, the linear operator~$\mathcal{L}_T\colon X_T\to Y_T$ by
	\begin{align}
        \label{mr:eq:L_T}
		\mathcal{L}_T(\z) \coloneqq
		\begin{pmatrix}
			\rho_1(\varphi_{1,0})\delt\vel_1 - \Div\Str_1(\varphi_{1,0},\D\vel_1) + \varphi_{1,0}\Grad p - \varphi_{1,0} \Grad\Lap\varphi_1 \\[1ex]
			\rho_2(\varphi_{1,0})\delt\vel_2 - \Div\Str_2(\varphi_{1,0},\D\vel_2) + (1-\varphi_{1,0})\Grad p + (1-\varphi_{1,0})\Grad\Lap\varphi_1 \\[1ex]
			\delt\varphi_1 + \Div(\varphi_{1,0}\vel_1) \\[1ex]
			\Div( \varphi_{1,0}\vel_1 + (1-\varphi_{1,0})\vel_2)
		\end{pmatrix},
	\end{align}
    where $\Str_j$ is defined in \eqref{int:eq:S_j}.
	On the other hand, for any~$0<T<\infty$, the corresponding nonlinear operator $\mathcal{F}_T\colon X_T\to Y_T$ is given by
	\begin{align}
        \label{mr:eq:F_T}
		\mathcal{F}_T(\z) \coloneqq
		\begin{pmatrix}
			\mathcal{F}_T^1(\z) \\[1ex]
			\mathcal{F}_T^2(\z) \\[1ex]
			- \Div\bigprt{(\varphi_1-\varphi_{1,0})\vel_1} \\[1ex]
			- \Div\bigprt{(\varphi_1-\varphi_{1,0})\vel_1 - (\varphi_1-\varphi_{1,0})\vel_2}
		\end{pmatrix},
	\end{align}
	where
	\begin{align*}
		\mathcal{F}_T^1(\z) 
		&\coloneqq -\bigprt{\rho_1(\varphi_1)-\rho_1(\varphi_{1,0})}\delt\vel_1 
		- \bigprt{\rho_1(\varphi_1)\vel_1\cdot\nabla} \vel_1 \\
		&\quad+ \Div\Bigprt{ 2\bigprt{\nu_1(\varphi_1)-\nu_1(\varphi_{1,0})} \D\vel_1 + \bigprt{\lambda_1(\varphi_1)-\lambda_1(\varphi_{1,0})} \Div\vel_1 \Id } 
		- \bigprt{\varphi_1-\varphi_{1,0}}\Grad p \\
		&\quad- \varphi_1 \Grad F'(\varphi_1)
		+ \bigprt{\varphi_1-\varphi_{1,0}} \Grad\Lap\varphi_1
		+ \Rcal(\varphi_1)(\vel_2-\vel_1), \\
		\mathcal{F}_T^2(\z) 
		&\coloneqq -\bigprt{\rho_2(\varphi_1)-\rho_2(\varphi_{1,0})}\delt\vel_2 
		- \bigprt{\rho_2(\varphi_1)\vel_2\cdot\nabla} \vel_2 \\
		&\quad+ \Div\Bigprt{ 2\bigprt{\nu_2(\varphi_1)-\nu_2(\varphi_{1,0})} \D\vel_2 + \bigprt{\lambda_2(\varphi_1)-\lambda_2(\varphi_{1,0})} \Div\vel_2 \Id } 
		+ \bigprt{\varphi_1-\varphi_{1,0}}\Grad p \\
		&\quad+ (1-\varphi_1) \Grad F'(\varphi_1)
		+ \bigprt{\varphi_1-\varphi_{1,0}} \Grad\Lap\varphi_1
		+ \Rcal(\varphi_1)(\vel_1-\vel_2).
	\end{align*}
    The fixed-point problem will then be formulated for the operator $\mathcal{L}_T^{-1} \circ \mathcal{F}_T$ on the solution space~$X_T$.
	
	The following propositions provide the key components for the proof of our main result. First, we state the invertibility of the linear operator $\mathcal{L}_T$, together with the fact that its inverse is bounded independently of $T$, and second, we need local Lipschitz continuity of the nonlinear operator $\mathcal{F}_T$.

	\begin{proposition}[Invertibility of $\mathcal{L}_T$ and uniform bound for $\mathcal{L}_T^{-1}$]
		\label{mr:prop:L_invertible}
        Let $\vel_{j,0}\in\rmH^1(\bbT)^d$, where $j=1,2$, and $\varphi_{1,0}\in\rmH^2(\bbT)$ with~$\varphi_{1,0}\in(0,1)$ in $\bbT$ be given that satisfy the condition~$\Div(\varphi_{1,0}\vel_{1,0} + (1-\varphi_{1,0}) \vel_{2,0}) = 0$.
		Then, for every $0<T<\infty$, the linear operator~$\mathcal{L}_T \colon X_T\to Y_T$ is invertible. Moreover, for any fixed $0<T_0<\infty$, there exists a constant~$C_{\mathcal{L}_{T_0}^{-1}}>0$ independent of $T$ such that
		\begin{align*}
			\bignorm{ \mathcal{L}_T^{-1}(\F) }_{Z_T}
			\leq C_{\mathcal{L}_{T_0}^{-1}} \Bigprt{ \norm{\F}_{Y_T} + \sum_{j=1,2}\norm{\vel_{j,0}}_{\rmH^1\ofbbTd} + \norm{\varphi_{1,0}}_{\rmH^2\ofbbT} }
		\end{align*}
        and 
		\begin{align*}
			\bignorm{ \mathcal{L}_T^{-1}(\F) - \mathcal{L}_T^{-1}(\G) }_{Z_T}
			\leq C_{\mathcal{L}_{T_0}^{-1}} \norm{ \F-\G }_{Y_T}
		\end{align*}
		for all $0<T\leq T_0$ and all $\F,\G\in Y_T$.
	\end{proposition}
	
	For a proof, we refer to Section~\ref{sec:lin}, which is devoted entirely to verifying this proposition.
	
	\begin{proposition}[Local Lipschitz continuity of $\mathcal{F}_T$]
		\label{mr:prop:F_Lipschitz}
        Let $\vel_{j,0}\in\rmH^1(\bbT)^d$, where $j=1,2$, and~$\varphi_{1,0}\in\rmH^2(\bbT)$ with~$\varphi_{1,0}\in(0,1)$ in $\bbT$ be given such that~$\Div(\varphi_{1,0}\vel_{1,0} + (1-\varphi_{1,0}) \vel_{2,0}) = 0$.
		Then, for every $R>0$ and every $0<T<\infty$, there exists some constant $C(R,T)>0$ such that~all~$(\vel_1,\vel_2,\varphi_1,p),(\w_1,\w_2,\psi_1,q) \in X_T$ with~$\norm{ (\vel_1,\vel_2,\varphi_1,p)}_{X_T}, \norm{(\w_1,\w_2,\psi_1,q) }_{X_T} \leq R$ satisfy
		\begin{align}
			&\norm{ \mathcal{F}_T(\vel_1,\vel_2,\varphi_1,p) - \mathcal{F}_T(\w_1,\w_2,\psi_1,q) }_{Y_T} \nonumber\\
			&\leq C(R,T) \norm{ (\vel_1,\vel_2,\varphi_1,p) - (\w_1,\w_2,\psi_1,q) }_{Z_T}.
		\end{align}
		Moreover, for every $R>0$, it holds $C(R,T)\to0$ as $T\to0$.
	\end{proposition}
	
	The proof of this proposition is carried out in Section~\ref{sec:Lip}.

    \medskip
	With these propositions at hand, we are now in a position to prove our main result, following the line of reasoning in \cite{AbelsWeber2020} and \cite{AbelsHaselboeck2026}.
	
	\begin{proof}[Proof of Theorem~\ref{int:thm:main_res}]
		The operators $\mathcal{L}_T$, $\mathcal{F}_T$ and the spaces $X_T$, $Y_T$ have been tailored such that an equivalence is established between solving the system~\eqref{int:eq:system} and the equation
		\begin{align}
			\label{mr:eq:fixed_point_eq}
			\z
			= \mathcal{L}_T^{-1} \bigprt{ \mathcal{F} (\z) }
			\eqqcolon \mathcal{I}_T (\z)
			\qquad\text{in } X_T,
		\end{align}
		where $\z\coloneqq (\vel_1,\vel_2,\varphi_1,p)$, which is well-defined by the invertibility of $\mathcal{L}_T$ from Proposition~\ref{mr:prop:L_invertible}. Thus, we aim at finding a solution to this fixed-point problem, which will be realized by means of a contraction principle.
		
		First, for some arbitrarily fixed $\bar\z \coloneqq (\bar\vel_1,\bar\vel_2,\bar\varphi_1,\bar p) \in X_{\bar T}$ for some $\bar T>0$, we define
		\begin{align*}
			M
			\coloneqq C_{\mathcal{L}_{\bar T}^{-1}} \Bigprt{ \norm{\mathcal{F}_{\bar T} (\bar \z) }_{Y_{\bar T} }+ \sum_{j=1,2}\norm{\vel_{j,0}}_{\rmH^1\ofbbTd} + \norm{\varphi_{1,0}}_{\rmH^2\ofbbT} }
			< \infty.
		\end{align*}
		Then we choose $R>0$ satisfying $\bar\z \in \overline{B_R^{Z_{\bar T}}(0)} \cap X_{\bar T}$ as well as $R>2M$. For this $R$ and for each $0<T<\infty$, Proposition~\ref{mr:prop:F_Lipschitz} gives rise to a constant $C(R,T)>0$ such that
		\begin{align}
			\label{mr:eq:F_Lipsch}
			\bignorm{ \mathcal{F}_T(\z_1) - \mathcal{F}_T(\z_2) }_{Y_T} 
			\leq C(R,T) \norm{ \z_1-\z_2 }_{Z_T}
		\end{align}
		for all $\z_i\in X_T$ with $\norm{ \z_i }_{X_T} \leq R$, $i=1,2$, along with $C(R,T)\to0$ as $T\to0$. Since we know from Proposition~\ref{mr:prop:L_invertible} that the constant $C_{\mathcal{L}_{\bar T}^{-1}}$ is independent of $T\in(0,\bar T]$, it is possible to find some sufficiently small $0<T<\infty$ satisfying
		\begin{align}
			\label{mr:eq:constants_small}
			C_{\mathcal{L}_{\bar T}^{-1}} C(R,T) < \frac14.
		\end{align}
		
		In order to ultimately apply the Banach fixed-point theorem, we will, in a next step, identify a suitable space for this purpose and check the respective assumptions of the theorem. Taking careful account of the fact that \eqref{mr:eq:F_Lipsch} is a property valid only on $\overline{B_R^{Z_T}(0)} \cap X_T$, we choose $B_R^{X_T}(0) \coloneqq B_R^{Z_T}(0) \cap X_T$ as the space in question. Then the first claim to be verified is that $\mathcal{I}_T$ maps $\overline{B_R^{X_T}(0)}$ into the same space. From the characteristics of $R$, Proposition~\ref{mr:prop:L_invertible}, and the monotonicity of the relevant norms with respect to time, we deduce for the restriction~$\bar\z\vert_{[0,T]} \in \overline{B_R^{X_T}(0)}$, that
		\begin{align*}
			\bignorm{ \mathcal{I}_T(\bar\z) }_{Z_T}
			= \bignorm{ \mathcal{L}_T^{-1} \bigprt{ \mathcal{F} (\bar\z)} }_{Z_T}
			\leq M< \frac{R}{2}.
		\end{align*}
		Considering the Lipschitz property~\eqref{mr:eq:F_Lipsch}, the smallness~\eqref{mr:eq:constants_small}, and the above estimate, we directly obtain
		\begin{align*}
			\bignorm{ \mathcal{I}_T(\z) }_{Z_T}
			&\leq \bignorm{ \mathcal{L}_T^{-1} \bigprt{ \mathcal{F} (\z)} - \mathcal{L}_T^{-1} \bigprt{ \mathcal{F} (\bar\z)} }_{Z_T} 
			+ \bignorm{ \mathcal{I}_T(\bar\z) }_{Z_T} \\
			&\leq C_{\mathcal{L}_{\bar T}^{-1}} \bignorm{ \mathcal{F}_T(\z) - \mathcal{F}_T(\bar\z) }_{Y_T} + \frac{R}{2} 
			\leq C_{\mathcal{L}_{\bar T}^{-1}} C(R,T) \norm{ \z-\bar\z }_{Z_T} + \frac{R}{2}
			< R
		\end{align*}
		for every $\z\in \overline{B_R^{X_T}(0)}$. Observing that, due to the definition of $\mathcal{L}_T^{-1}$, the initial condition is satisfied as well, the self-mapping property of $\mathcal{I}_T$ is established.
		
		The second assumption to be checked is that $\mathcal{I}_T$ is a contraction on $\overline{B_R^{X_T}(0)}$, which follows immediately in light of
		\begin{align*}
			\bignorm{ \mathcal{I}_T(\z_1) - \mathcal{I}_T(\z_2) }_{Z_T} 
			&\leq C_{\mathcal{L}_{\bar T}^{-1}} C(R,T) \norm{ \z_1-\z_2 }_{Z_T} 
			\leq \frac14 \norm{ \z_1-\z_2 }_{Z_T}
		\end{align*}
		for every $\z_1,\z_2\in \overline{B_R^{X_T}(0)}$.
	
		With these verified properties in place, we apply the Banach fixed-point theorem, which guarantees the existence of a unique solution $\z^* \in \overline{B_R^{X_T}(0)}$ to the fixed-point equation~\eqref{mr:eq:fixed_point_eq}, and equivalently, to the system~\eqref{int:eq:system}.
		
		Finally, it remains to show that the solution is unique in the whole of $X_T$. Assuming that $\hat\z\in X_T$ is another solution, we find, by the above reasoning, some $\hat T\in(0,T]$ and some $\hat R\geq R$ such that \eqref{mr:eq:fixed_point_eq} is uniquely solvable in $\overline{B_{\hat R}^{X_{\hat T}}(0)}$. Consequently, $\z^*$ and $\hat\z$ already coincide on the interval $[0,\hat T]$ and by a standard continuation argument, we conclude that this is also true for the entire interval $[0,T]$.
	\end{proof}

	\section{Existence and Uniqueness for the Linear System} \label{sec:lin}

    As the proof of Theorem~\ref{int:thm:main_res} in the previous section illustrates, the reduction of the well-posedness problem to a fixed-point formulation depends essentially on the invertibility of the linear operator~$\mathcal L_T$ and the uniform boundedness of its inverse. Therefore, our objective is to uniquely solve the following system corresponding to $\mathcal L_T$, see~\eqref{mr:eq:L_T}, which is derived by linearizing the system~\eqref{int:eq:system} around $\tphi1$. The relevant principal part of this linearization is given by
	\begin{subequations}
		\label{Lin:eq:lin_system}
		\begin{alignat}{2}
			\trho1\delt\vel_1 - \Div\bigprt{2\tnu1 \D\vel_1 + \tlambda1 \Div\vel_1 \Id} + \tphi1\Grad p - \tphi1 \Grad\Lap\varphi_1 &= \f_1 
			&\;\;\text{in }(0,T)\times\bbT,
			\label{Lin:eq:lin_system_eq_deltv1}\\[0.2ex] 
			\trho2\delt\vel_2 - \Div\bigprt{2\tnu2 \D\vel_2 + \tlambda2 \Div\vel_2 \Id} + \tphi2\Grad p + \tphi2 \Grad\Lap\varphi_1 &= \f_2
			&\text{in }(0,T)\times\bbT,
			\label{Lin:eq:lin_system_eq_deltv2}\\[0.2ex] 
			\delt\varphi_1 + \Div(\tphi1\vel_1) &= g_1
			&\text{in }(0,T)\times\bbT,
			\label{Lin:eq:lin_system_eq_deltphi1}\\[0.2ex] 
			\Div( \tphi1\vel_1 + \tphi2\vel_2) &= g_2
			&\text{in }(0,T)\times\bbT,
			\label{Lin:eq:lin_system_eq_div}
		\end{alignat}
		along with the initial condition
		\begin{align}
            \label{Lin:eq:reg_system_eq_initial}
			(\vel_1,\vel_2,\varphi_1)\vert_{t=0} &= (\vel_{1,0},\vel_{2,0},\varphi_{1,0}) 
			\quad\text{on }\bbT.
		\end{align}
	\end{subequations}
    For brevity, we have denoted $\tphi2 \coloneqq 1-\tphi1$ here, and similarly, for $j=1,2$, we have written~$\tilde\rho_{j,0} \coloneqq \rho_j(\tphi1)$ as well as $\tnu{j} \coloneqq \nu_j(\tphi1)$ and $\tlambda{j} \coloneqq \lambda_j(\tphi1)$.

	\subsection{Weak Solvability of the Regularized System}
	To apply abstract parabolic theory, we regularize the linear system~\eqref{Lin:eq:lin_system} with a term $- \varepsilon\Lap\varphi_1$, $\varepsilon>0$, in the third equation. Moreover, we reduce to a vanishing right-hand side in the fourth equation. The regularized linear parabolic system then reads

	\begin{subequations}
		\label{Lin:eq:reg_system}
		\begin{alignat}{2}
			\trho1\delt\vel_1 - \Div\bigprt{2\tnu1 \D\vel_1 + \tlambda1 \Div\vel_1 \Id} + \tphi1\Grad p - \tphi1 \Grad\Lap\varphi_1 &= \f_1
			&\;\;\text{in }(0,T)\times\bbT, 
			\label{Lin:eq:reg_system_eq_deltv1}\\[0.2ex]
			\trho2\delt\vel_2 - \Div\bigprt{2\tnu2 \D\vel_2 + \tlambda2 \Div\vel_2 \Id} + \tphi2\Grad p + \tphi2 \Grad\Lap\varphi_1 &= \f_2
			&\;\;\text{in }(0,T)\times\bbT, 
			\label{Lin:eq:reg_system_eq_deltv2}\\[0.2ex]
			\delt\varphi_1 + \Div(\tphi1\vel_1) - \varepsilon\Lap\varphi_1 &= g_1
			&\;\;\text{in }(0,T)\times\bbT,
			\label{Lin:eq:reg_system_eq_deltphi1}\\[0.2ex]
			\Div( \tphi1\vel_1 + \tphi2\vel_2) &= 0
			&\;\;\text{in }(0,T)\times\bbT,
			\label{Lin:eq:reg_system_eq_div}
		\end{alignat}
		together with the initial condition
		\begin{align}
			(\vel_1,\vel_2,\varphi_1)\vert_{t=0} &= (\vel_{1,0},\vel_{2,0},\varphi_{1,0}) 
			\quad\text{on }\bbT.
		\end{align}
	\end{subequations}

	\medskip
	\noindent\textbf{Abstract function spaces.}
    We introduce an abstract framework that allows the application of abstract parabolic theory. To this end, let $\bbV$ and $\bbH$ denote the Hilbert spaces defined as follows. On the one hand, let $\bbH = \bbH_1 \times \bbH_2$, where
	\begin{align*}
		\bbH_1 &= \big\{ (\vel_1,\vel_2) \in \rmL^2(\bbT)^d \times \rmL^2(\bbT)^d \colon \Div( \tphi1\vel_1 + \tphi2\vel_2)=0 \text{ in } \calD'(\bbT) \big\}, \\
		\bbH_2 &= \rmH_{(0)}^1(\bbT) = \big\{ \varphi \in \rmH^1(\bbT) \colon \mean\varphi= 0 \big\},
	\end{align*}
	equipped with the inner products
	\begin{align*}
		\bigprt{(\vel_1,\vel_2),(\w_1,\w_2)}_{\bbH_1} &= \sum_{j=1,2} \int_\bbT \trho{j}\vel_j\cdot\w_j \dx, \\
		(\varphi,\psi)_{\bbH_2} &= 2\int_\bbT \Grad\varphi\cdot\Grad\psi \dx.
	\end{align*}
	On the other hand, we define $\bbV=\bbV_1\times\bbV_2$ by
	\begin{align*}
		\bbV_1 &= \big\{ (\vel_1,\vel_2) \in \rmH^1(\bbT)^d \times \rmH^1(\bbT)^d \colon \Div( \tphi1\vel_1 + \tphi2\vel_2)=0 \text{ in } \rmL^2(\bbT) \big\}, \\
		\bbV_2 &= \rmH_{(0)}^2(\bbT) = \big\{ \varphi \in \rmH^2(\bbT) \colon \mean\varphi=0 \big\},
	\end{align*}
	with the inner products
	\begin{align*}
		\bigprt{(\vel_1,\vel_2),(\w_1,\w_2)}_{\bbV_1} &= \sum_{j=1,2} \int_\bbT \trho{j}\vel_j\cdot\w_j + 2\tnu{j}\D\vel_j:\D\w_j \dx, \\
		(\varphi,\psi)_{\bbV_2} &= 2\int_\bbT \Grad\varphi\cdot\Grad\psi + \Lap\varphi\Lap\psi \dx.
	\end{align*}
	Since the embedding $\bbV\hookrightarrow\bbH$ is continuous and dense, $\bbV\hookrightarrow\bbH \cong \bbH'\hookrightarrow\bbV'$ forms a Gelfand triple. Note carefully that the Riesz isomorphisms on $\bbH_1$ and $\bbH_2$ are given by
	\begin{align*}
		(\vel_1,\vel_2) &\mapsto \bigprt{(\vel_1,\vel_2), \cdot}_{\bbH_1} = \bigprt{ (\trho1\vel_1, \trho2\vel_2), \cdot }_{\rmL^2(\bbT)^d \times \rmL^2(\bbT)^d}, \\ 
		\varphi &\mapsto (\varphi,\cdot)_{\bbH_2} = 2\bigprt{ \Grad\varphi, \Grad\cdot }_{\rmL^2(\bbT)},
	\end{align*}
	respectively. In particular, the identity $\ang{\varphi}{\psi}_{\bbV_2',\bbV_2} = -2\int_\bbT \varphi\Lap\psi \dx$ for all $\varphi,\psi\in \bbV_2$ shows that the Riesz isomorphism on $\bbH_2$ corresponds to the operator $-2\Lap$.

	\medskip
	\noindent\textbf{Abstract weak formulation.}
	In the following, we write the regularized linear system~\eqref{Lin:eq:reg_system} as an abstract evolution equation. To this end, let $0<T<\infty$, $\varepsilon>0$, and let~$\tphi1\in\rmH^2(\bbT)$ with $\tphi1\in(0,1)$ in $\bbT$. For all $(\f_1 / \trho1, \f_2 / \trho2, g_1) \in \rmL^2(0,T;\bbV')$ and $(\vel_{1,0},\vel_{2,0},\varphi_{1,0}) \in \bbH$, we consider the abstract problem
	\begin{equation}
		\label{Lin:eq:abstract_eq}
		\begin{aligned}
			\delt(\vel_1,\vel_2,\varphi_1) + A_\varepsilon(\vel_1,\vel_2,\varphi_1) &= \prt{ \f_1 / \trho1, \f_2 / \trho2, g_1 } \quad\text{in } \bbV' \text{ a.e. in } (0,T), \\
			(\vel_1,\vel_2,\varphi_1)\vert_{t=0} &= (\vel_{1,0},\vel_{2,0},\varphi_{1,0}),
		\end{aligned}
	\end{equation}
	for $(\vel_1,\vel_2,\varphi_1)\in \rmL^2(0,T;\bbV) \cap \rmH^1(0,T;\bbV')$, with an operator $A_\varepsilon\colon \bbV\to\bbV'$ given by
	\begin{align*}
		&\bigang{A_\varepsilon(\vel_1,\vel_2,\varphi_1)}{(\w_1,\w_2,\psi)}_{\bbV',\bbV}
		\coloneqq \sum_{j=1,2} \int_\bbT 2\tnu{j}\D\vel_j:\D\w_j + \tlambda{j}\Div\vel_j\Div\w_j \dx \\
		&+\! \int_\bbT \Lap\varphi_1 \Div(\tphi1\w_1) - \Lap\varphi_1 \Div(\tphi2\w_2) \dx 
		- 2\!\int_\bbT \Div(\tphi1\vel_1)\Lap\psi \dx + 2\varepsilon\! \int_\bbT \Lap\varphi_1\Lap\psi \dx
	\end{align*}
	for all $(\vel_1,\vel_2,\varphi_1),(\w_1,\w_2,\psi)\in\bbV$. 
	
	\begin{definition}
        \label{Lin:def:abstract_eq_weak}
		A triple $(\vel_1,\vel_2,\varphi_1) \in \rmL^2(0,T;\bbV) \cap \rmH^1(0,T;\bbV')$ is called a weak solution to \eqref{Lin:eq:abstract_eq} if it satisfies
		\begin{align}
			\label{Lin:eq:abstract_eq_weak}
			&\bigang{\delt(\vel_1,\vel_2,\varphi_1)}{(\w_1,\w_2,\psi)}_{\bbV',\bbV} 
			+ \bigang{A_\varepsilon(\vel_1,\vel_2,\varphi_1)}{(\w_1,\w_2,\psi)}_{\bbV',\bbV} \nonumber\\
			&= \bigang{(\f_1 / \trho1, \f_2 / \trho2, g_1)}{(\w_1,\w_2,\psi)}_{\bbV',\bbV}
            \qquad \text{ a.e. in } (0,T)
		\end{align}
		for all $(\w_1,\w_2,\psi)\in \rmL^2(0,T;\bbV)$, as well as the initial condition
        \begin{align*}
            (\vel_1,\vel_2,\varphi_1)\vert_{t=0} &= (\vel_{1,0},\vel_{2,0},\varphi_{1,0}) 
            \quad\text{in } \rmL^2(\bbT)^d \times \rmL^2(\bbT)^d \times \rmH^1(\bbT).
        \end{align*}
	\end{definition}
	
	\begin{remark}
		\label{Lin:rmk:projection}
		We point out the following connection of $\bbH_1$ with the well-known divergence-free space~$\rmL^2_\sigma(\bbT) \coloneqq \{\vel\in\rmL^2(\bbT)^d \colon \Div \vel=0\}$ used to analyze the standard Stokes equations.
		To this end, we define $\mathbb{W}\coloneqq \{ (\vel_1,\vel_2,\varphi_1)\in\bbV \colon \vel_1,\vel_2\in\rmH^2(\bbT)^d, \varphi_1\in\rmH^3(\bbT) \}$. If we have a higher regularity $(\vel_1,\vel_2,\varphi_1)\in\mathbb{W}$, we can write the first and second components of the operator~$A_\varepsilon$ as $A_1\colon \bbV\to\bbV_1'$ given by
		\begin{align*}
			&\bigang{A_1(\vel_1,\vel_2,\varphi_1)}{(\w_1,\w_2)}_{\bbV_1',\bbV_1} \\
			&= \int_\bbT 2\tnu{1}\D\vel_1:\D\w_1 + \tlambda{1}\Div\vel_1\Div\w_1 + \Lap\varphi_1 \Div(\tphi1\w_1)\dx \\
			&\quad+ \int_\bbT 2\tnu{2}\D\vel_2:\D\w_2 + \tlambda{2}\Div\vel_2\Div\w_2  - \Lap\varphi_1 \Div(\tphi2\w_2) \dx \\
			&= - \int_\bbT \trho1 \trho1^{-1} \bigprt{ \Div\bigprt{2\tnu{1}\D\vel_1 + \tlambda{1}\Div\vel_1\Id } + \tphi1\Grad\Lap\varphi_1 } \w_1 \dx \\
			&\quad- \int_\bbT \trho2 \trho2^{-1} \bigprt{ \Div\bigprt{2\tnu{2}\D\vel_2 + \tlambda{2}\Div\vel_2\Id } - \tphi2\Grad\Lap\varphi_1 } \w_2 \dx \\
			&= \biggprt{ \mathbb{P} 
			\begin{pmatrix}
				- \trho1^{-1} \bigprt{ \Div\bigprt{2\tnu{1}\D\vel_1 + \tlambda{1}\Div\vel_1\Id } + \tphi1\Grad\Lap\varphi_1 } \\
				- \trho2^{-1} \bigprt{ \Div\bigprt{2\tnu{2}\D\vel_2 + \tlambda{2}\Div\vel_2\Id } - \tphi2\Grad\Lap\varphi_1 } \\
			\end{pmatrix},
			\begin{pmatrix}
				\w_1 \\
				\w_2 \\
			\end{pmatrix}
			}_{\bbH_1}
		\end{align*}  
		for all $(\w_1,\w_2)\in\bbV_1$. Here, $\mathbb{P}\colon \rmL^2(\bbT)^d \times \rmL^2(\bbT)^d \to \bbH_1$ is the orthogonal projection onto~$\bbH_1$ with respect to the inner product $(\cdot,\cdot)_{\bbH_1}$, which corresponds to the Helmholtz projection onto~$\rmL^2_\sigma(\bbT)$ for the Stokes equations. This reasoning reveals that $A_1(\vel_1,\vel_2,\varphi_1) \in \bbH_1$ for~$(\vel_1,\vel_2,\varphi_1)\in\mathbb{W}$.
	\end{remark}

	\medskip
	\noindent\textbf{Abstract existence result.}
    In the above setting, we establish the following existence and regularity result for the abstract evolution equation~\eqref{Lin:eq:abstract_eq}.
	
	\begin{lemma}
		\label{Lin:lemma:existence_abstr}
        Let $0<T<\infty$ and $\varepsilon>0$, and let $\tphi1\in\rmH^2(\bbT)$ satisfy $\tphi1\in(0,1)$ in $\bbT$.
		\begin{enumerate}[label=(\roman*)]
			\item \label{Lin:item:existence_abstr_i}
			For every $(\f_1 / \trho1, \f_2 / \trho2, g_1) \in \rmL^2(0,T;\bbV')$ and $(\vel_{1,0},\vel_{2,0},\varphi_{1,0}) \in \bbH$, equation~\eqref{Lin:eq:abstract_eq} admits a unique solution 
			\begin{align*}
				(\vel_1,\vel_2,\varphi_1) \in \rmL^2(0,T;\bbV) \cap \rmH^1(0,T;\bbV').
			\end{align*}
			\item  \label{Lin:item:existence_abstr_ii}
			Let $(\f_1 / \trho1, \f_2 / \trho2, g_1) \in \rmH^1(0,T;\bbV')$ with $(\f_1 / \trho1, \f_2 / \trho2, g_1) \vert_{t=0}\in\bbH$ and suppose that the initial data~$(\vel_{1,0},\vel_{2,0},\varphi_{1,0}) \in \bbV$ are of regularity $\vel_{j,0}\in\rmH^2(\bbT)^d$, $j=1,2$, and $\varphi_{1,0}\in\rmH^3(\bbT)$.
			Then, with respect to time, the solution to \eqref{Lin:eq:abstract_eq} has the higher regularity 
			\begin{align*}
				(\vel_1,\vel_2,\varphi_1) \in \rmH^1(0,T;\bbV) \cap \rmH^2(0,T;\bbV').
			\end{align*}
			\item \label{Lin:item:existence_abstr_iii}
			If in addition to the assumptions of \ref{Lin:item:existence_abstr_ii}, the right-hand side $(\f_1 / \trho1, \f_2 / \trho2, g_1)$ is in~$\rmL^2(0,T;\rmL^2(\bbT)^d \times \rmL^2(\bbT)^d \times \rmH^1(\bbT))$, the solution has the additional spatial regularity 
			\begin{align*}
				(\vel_1,\vel_2,\varphi_1) \in  \rmL^2(0,T;\rmH^2(\bbT)^d \times \rmH^2(\bbT)^d \times \rmH^3(\bbT)).
			\end{align*}
		\end{enumerate}
	\end{lemma}
	
	\begin{remark}
		\label{Lin:rmk:ptwise}
		Under the assumptions of Lemma~\ref{Lin:lemma:existence_abstr}~\ref{Lin:item:existence_abstr_ii}, the weak solution $(\vel_1,\vel_2,\varphi_1)$ of the abstract equation~\eqref{Lin:eq:abstract_eq} is sufficiently regular to solve the equations~\eqref{Lin:eq:reg_system_eq_deltphi1} and \eqref{Lin:eq:reg_system_eq_div} even a.e.~in $(0,T)\times\bbT$. While this is immediate for \eqref{Lin:eq:reg_system_eq_div}, it follows for \eqref{Lin:eq:reg_system_eq_deltphi1} from testing the weak formulation~\eqref{Lin:eq:abstract_eq_weak} with some arbitrary $(0,0,\psi)\in\bbV$ and invoking the above regularity properties for~$(\vel_1,\vel_2,\varphi_1)$, the form of the Riesz isomorphism on $\bbH_2$, and the fact that~$-2\Lap \colon \rmH^2_{(0)}(\bbT) \to \rmL^2_{(0)}(\bbT)$ is an isomorphism. 
	\end{remark}

    Having established the required framework, the proof of Lemma~\ref{Lin:lemma:existence_abstr}~\ref{Lin:item:existence_abstr_i}--\ref{Lin:item:existence_abstr_ii} is now a straightforward application of Proposition~\ref{prelim:prop:APEE}.
	
	\begin{proof}[Proof of Lemma~\ref{Lin:lemma:existence_abstr}~\ref{Lin:item:existence_abstr_i}--\ref{Lin:item:existence_abstr_ii}]
		In order to prove \ref{Lin:item:existence_abstr_i} and \ref{Lin:item:existence_abstr_ii}, we aim at applying the abstract existence result Proposition~\ref{prelim:prop:APEE}. To this end, we denote the time-independent bilinear form associated to $A_\varepsilon$ by $a_\varepsilon(t,\cdot,\cdot)\colon \bbV\times\bbV\to\R$, where
		\begin{align*}
			a_\varepsilon(t,(\vel_1,\vel_2,\varphi_1),(\w_1,\w_2,\psi)) 
            \coloneqq \ang{A_\varepsilon(\vel_1,\vel_2,\varphi_1)}{(\w_1,\w_2,\psi)}_{\bbV',\bbV} \quad \text{for all } t\in[0,T],
		\end{align*}
		and verify the assumptions \ref{prelim:APEE:ass_a}--\ref{prelim:APEE:ass_d} on the bilinear form. With $a_\varepsilon$ not depending on time, \ref{prelim:APEE:ass_a} and \ref{prelim:APEE:ass_d} are immediate. Regarding \ref{prelim:APEE:ass_c}, $a_\varepsilon(t,\cdot,\cdot)$ is coercive on $\bbV$ due to the regularization term, since
		\begin{align*}
			\bigang{A_\varepsilon(\vel_1,\vel_2,\varphi_1)}{(\vel_1,\vel_2,\varphi_1)}_{\bbV',\bbV} 
			&= \sum_{j=1,2} \int_\bbT 2\tnu{j} \abs{\D\vel_j}^2 + \tlambda{j}\prt{\Div\vel_j}^2\dx
			+ 2\varepsilon\! \int_\bbT \! \prt{\Lap\varphi_1}^2 \dx \\
			&\geq \norm{(\vel_1,\vel_2)}_{\bbV_1}^2 - \norm{(\vel_1,\vel_2)}_{\bbH_1}^2 + \varepsilon\norm{\varphi_1}_{\bbV_2}^2 - \varepsilon\norm{\varphi}_{\bbH_2}^2 \\
			&\geq C_0\norm{(\vel_1,\vel_2,\varphi_1)}_\bbV^2 - C_1\norm{(\vel_1,\vel_2,\varphi_1)}_\bbH^2
		\end{align*}
		for all $(\vel_1,\vel_2,\varphi_1)\in\bbV$ with constants $C_0,C_1>0$ depending on $\varepsilon$. Moreover, Hölder's theorem yields the required boundedness \ref{prelim:APEE:ass_b} of $a_\varepsilon(t,\cdot,\cdot)$ in the sense that
		\begin{align*}
			\bigabs{ \bigang{A_\varepsilon(\vel_1,\vel_2,\varphi_1)}{(\w_1,\w_2,\psi)}_{\bbV',\bbV} }
			\leq C \norm{(\vel_1,\vel_2,\varphi_1)}_\bbV \norm{(\w_1,\w_2,\psi)}_\bbV
		\end{align*}
		for all $(\vel_1,\vel_2,\varphi_1),(\w_1,\w_2,\psi)\in\bbV$ with a constant $C>0$ depending on $\varepsilon$. 
		It remains to check the compatibility condition $(\f_1 / \trho1, \f_2 / \trho2, g_1) \vert_{t=0} - A_\varepsilon (\vel_{1,0},\vel_{2,0},\varphi_{1,0}) \in \bbH$, where the first part is satisfied directly via assumption. The assumed additional regularity of the initial data, combined with the projection arguments of Remark~\ref{Lin:rmk:projection}, shows that $A_\varepsilon (\vel_{1,0},\vel_{2,0},\varphi_{1,0})$ is sufficiently regular and satisfies the divergence condition, and thus belongs to $\bbH$.
		By Proposition~\ref{prelim:prop:APEE}, we therefore conclude \ref{Lin:item:existence_abstr_i} and \ref{Lin:item:existence_abstr_ii}.
	\end{proof}

	Before we can prove the higher regularity stated in Lemma~\ref{Lin:lemma:existence_abstr}~\ref{Lin:item:existence_abstr_iii}, we need to reconstruct the pressure for fixed $t$, which enables us to treat the first and second equation of the system~\eqref{Lin:eq:reg_system} separately.	 
	
	\begin{lemma}[Reconstruction of $p$ for fixed time] 
		\label{Lin:lemma:pressure_fixed_t}
		Let $0<T<\infty$ and $\varepsilon>0$, and let~$\tphi1\in\rmH^2(\bbT)$ satisfy $\tphi1\in(0,1)$ in $\bbT$. Suppose that the assumptions of Lemma~\ref{Lin:lemma:existence_abstr}~\ref{Lin:item:existence_abstr_ii} hold. Denoting by~$(\vel_1,\vel_2,\varphi_1)$ the weak solution to the regularized system~\eqref{Lin:eq:reg_system} in the sense of Definition~\ref{Lin:def:abstract_eq_weak}, additionally assume that $\Grad\Lap\varphi_1\in\rmL^2(0,T;\rmL^2(\bbT))$. Then, for almost every~$t\in(0,T)$, there exists a unique~$p(t)\in\rmH^1_{(0)}(\bbT)$ such that $(\vel_1,\vel_2,\varphi_1,p)$ solves each of the equations \eqref{Lin:eq:eq_without_div_cond_fixed_t_1} and \eqref{Lin:eq:eq_without_div_cond_fixed_t_2} below weakly a.e.~in~$(0,T)$.
	\end{lemma}
	
	\begin{proof}
		With the regularity properties provided by Lemma~\ref{Lin:lemma:existence_abstr}~\ref{Lin:item:existence_abstr_ii}, equation~\eqref{Lin:eq:reg_system_eq_deltphi1}, which is satisfied pointwise, cf.~Remark~\ref{Lin:rmk:ptwise}, shows
		\begin{align*}
			\Div(\tphi1\vel_1) = g_1 - \delt\varphi_1 + \varepsilon\Lap\varphi_1 \in \rmH^1(\bbT)
		\end{align*}
		a.e.~in $(0,T)$. Combined with the same observation for \eqref{Lin:eq:reg_system_eq_div}, it follows for $j=1,2$ that~$\Div(\tphi{j}\vel_j) \in \rmH^1(\bbT)$ and consequently $\Div\vel_j \in \rmH^1(\bbT)$.
		
		Moreover, in light of the assumed regularities, the weak formulation~\eqref{Lin:eq:abstract_eq_weak} tested with arbitrary $(\hat\w_1,\hat\w_2,0)\in\bbV$ may be written as
		\begin{align}
			\label{Lin:eq:eq_for_hatwj_fixed_t}
			&\sum_{j=1,2} \int_\bbT \trho{j}\delt\vel_j \cdot \hat\w_j \dx 
			+ \sum_{j=1,2} \int_\bbT 2\tnu{j}\D\vel_j:\D\hat\w_j + \tlambda{j}\Div\vel_j\Div\hat\w_j \dx \nonumber\\
			&\quad- \int_\bbT (  \tphi1 \hat\w_1 -  \tphi2 \hat\w_2) \cdot \Grad\Lap\varphi_1 \dx 
			= \sum_{j=1,2} \int_\bbT \f_j \cdot \hat\w_j \dx
		\end{align}
		a.e.~in $(0,T)$, for all $(\hat\w_1,\hat\w_2)\in\bbV_1$, i.e., for all $(\hat\w_1,\hat\w_2)\in \rmH^1(\bbT)^d \times \rmH^1(\bbT)^d$ satisfying the divergence condition $\Div(\tphi1\hat\w_1+\tphi2\hat\w_2)=0$ in $\rmL^2(\bbT)$. Now, for fixed $t\in(0,T)$, we aim to reconstruct the pressure by formally summing~\eqref{Lin:eq:reg_system_eq_deltv1} tested with $\barrho1^{-1}\Grad\xi$ and \eqref{Lin:eq:reg_system_eq_deltv2} tested with $\barrho2^{-1}\Grad\xi$ for arbitrary $\xi\in\rmH^1(\bbT)$. Abbreviating $\tant\coloneqq \barrho1^{-1}\tphi1 + \barrho2^{-1}\tphi2$ and $\tdnt\coloneqq \barrho1^{-1}\tphi1 - \barrho2^{-1}\tphi2$, this means that we determine $p(t)\in\rmH^1_{(0)}(\bbT)$ as the unique solution to the equation
		\begin{align}
			\label{Lin:eq:eq_for_pressure_fixed_t}
			&\int_\bbT \tant \Grad p\cdot\Grad\xi\dx
			= \sum_{j=1,2} \int_\bbT \barrho{j}^{-1} \bigprt{ 2\tnu{j}\D\vel_j:\D\Grad\xi + \tlambda{j}\Div\vel_j\Div\Grad\xi } \dx  \nonumber \\
			&\quad+ \int_\bbT \tdnt \Grad\Lap\varphi_1 \cdot \Grad\xi \dx
			+ \sum_{j=1,2} \int_\bbT  \barrho{j}^{-1} \f_j \cdot \Grad\xi\dx
			- \sum_{j=1,2} \int_\bbT \tphi{j}\delt\vel_j \cdot \Grad\xi \dx 
		\end{align}
		a.e.~in $(0,T)$, for all $\xi\in\rmH^2(\bbT)$ to begin with. To solve this equation, we study the regularity of the right-hand side. For the first term there, a suitable integration by parts along with~$\ang{\Lap\vel_j}{ \Grad\xi }_{\rmH^{-1}\ofbbT,\rmH^1\ofbbT} = -(\Grad\vel_j,\Grad^2\xi)_{\rmL^2\ofbbT} = (\Grad\Div\vel_j, \Grad\xi)_{\rmL^2\ofbbT} $ reveals
		\begin{align*}
			&\int_\bbT \barrho{j}^{-1} \bigprt{ 2\tnu{j}\trans{(\nabla\vel_j)}:\Grad^2\xi + \tlambda{j}\Div\vel_j\Div\Grad\xi } \dx \\
			&= -\bigang{ \barrho{j}^{-1}\Div\bigprt{ 2\tnu{j}\trans{(\nabla\vel_j)} + \tlambda{j}\Div\vel_j\Id } }{ \Grad\xi }_{\rmH^{-1}\ofbbT,\rmH^1\ofbbT}  \\
			&= -\int_\bbT \barrho{j}^{-1} ( 2\tnu{j}+\tlambda{j} ) \Grad\Div\vel_j \cdot \Grad\xi \dx
			- \int_\bbT \barrho{j}^{-1} \bigprt{ 2\trans{(\nabla\vel_j)}\Grad\tnu{j} + \Div\vel_j\Grad\tlambda{j} } \cdot \Grad\xi \dx.
		\end{align*}
		With this identity, \eqref{Lin:eq:eq_for_pressure_fixed_t} takes the form of an equation
		\begin{align*}
			\int_\bbT \tant \Grad p\cdot\Grad\xi \dx
			= \int_\bbT \mathbf F\cdot\Grad\xi \dx
		\end{align*}
		 a.e.~in $(0,T)$, for all $\xi\in\rmH^2(\bbT)$, which is also valid for all $\xi\in\rmH^1(\bbT)$. Since this is a weak elliptic equation on $\bbT$ with $\tant>0$ in $\bbT$ and right-hand side $\F(t)\in\rmL^2(\bbT)^d$ for almost every $t\in(0,T)$, the Lax--Milgram theorem and Poincaré's inequality yield the existence of a unique solution $p(t)\in\rmH^1_{(0)}(\bbT)$ for almost every $t\in(0,T)$.
		
		Now it remains to verify that, regarding this pressure $p$, the equations~\eqref{Lin:eq:reg_system_eq_deltv1} and~\eqref{Lin:eq:reg_system_eq_deltv2} are satisfied weakly with independent test functions.
		To this end, let $\w_j\in\rmH^1(\bbT)^d$, $j=1,2$, be arbitrary and decompose $\w_j$ into~$\w_j=\hat\w_j + \barrho{j}^{-1}\Grad\xi$ such that $\Div(\tphi1\hat\w_1+\tphi2\hat\w_2)=0$ (that is, $\xi$ is the solution to~$\Div\bigprt{\tant \Grad\xi } = \Div(\tphi1\w_1+\tphi2\w_2)$). Then, with $p$ from above, we have
		\begin{align}
			\label{Lin:eq:eq_without_div_cond_fixed_t}
			&\sum_{j=1,2} \int_\bbT \trho{j}\delt\vel_j \cdot \w_j \dx
			+ \sum_{j=1,2} \int_\bbT 2\tnu{j}\D\vel_j:\D\w_j + \tlambda{j}\Div\vel_j\Div\w_j \dx  \nonumber\\
			&+ \sum_{j=1,2} \int_\bbT \tphi{j}\Grad p\cdot\w_j \dx 
			- \int_\bbT (  \tphi1 \w_1 -  \tphi2 \w_2) \cdot \Grad\Lap\varphi_1 \dx 
			= \sum_{j=1,2} \int_\bbT \f_j \cdot \w_j \dx
		\end{align}
		a.e.~in $(0,T)$, which follows from \eqref{Lin:eq:eq_for_hatwj_fixed_t} for the terms including $\hat\w_j$, \eqref{Lin:eq:eq_for_pressure_fixed_t} for the terms with~$\Grad\xi$, and from
		\begin{align*}
			\sum_{j=1,2} \int_\bbT \tphi{j}\Grad p\cdot\hat\w_j \dx 
			= - \int_\bbT \Div(\tphi1\hat\w_1 + \tphi2\hat\w_2) p \dx
			= 0.
		\end{align*}
		On the one hand, setting $\w_2=0$ in \eqref{Lin:eq:eq_without_div_cond_fixed_t} shows that
		\begin{subequations}
			\begin{align}
				\label{Lin:eq:eq_without_div_cond_fixed_t_1}
				&\int_\bbT \trho1\delt\vel_1 \cdot \w_1 \dx
				+ \int_\bbT 2\tnu1\D\vel_1:\D\w_1 + \tlambda1\Div\vel_1\Div\w_1 \dx \nonumber \\
				&\quad+ \int_\bbT \tphi1\Grad p\cdot\w_1 \dx
				- \int_\bbT  \tphi1 \w_1 \cdot \Grad\Lap\varphi_1 \dx
				= \int_\bbT \f_1 \cdot \w_1 \dx
			\end{align}
			a.e.~in $(0,T)$ is satisfied for all $\w_1\in\rmH^1(\bbT)^d$, which corresponds to the weak formulation of equation~\eqref{Lin:eq:reg_system_eq_deltv1} a.e.~in $(0,T)$. On the other hand, choosing $\w_1=0$ provides the according weak equation for~\eqref{Lin:eq:reg_system_eq_deltv2}, given by
			\begin{align}
				\label{Lin:eq:eq_without_div_cond_fixed_t_2}
				&\int_\bbT \trho2\delt\vel_2 \cdot \w_2 \dx
				+ \int_\bbT 2\tnu2\D\vel_2:\D\w_2 + \tlambda2\Div\vel_2\Div\w_2 \dx \nonumber \\
				&\quad+ \int_\bbT \tphi2\Grad p\cdot\w_2 \dx
				+ \int_\bbT  \tphi2 \w_2 \cdot \Grad\Lap\varphi_1 \dx
				= \int_\bbT \f_2 \cdot \w_2 \dx
			\end{align}
			a.e.~in $(0,T)$ for all $\w_2\in\rmH^1(\bbT)^d$.
		\end{subequations}
		This finishes the claimed reconstruction of the pressure.
	\end{proof}

	Now we are finally in a position to prove the third assertion of Lemma~\ref{Lin:lemma:existence_abstr}.
	
	\begin{proof}[Proof of Lemma~\ref{Lin:lemma:existence_abstr}~\ref{Lin:item:existence_abstr_iii}]
	   Under the assumptions of \ref{Lin:item:existence_abstr_iii}, the weak solution satisfies the equation
	   \begin{align*}
	       A_\varepsilon(\vel_1,\vel_2,\varphi_1) = (\f_1 / \trho1, \f_2 / \trho2, g_1) - \delt(\vel_1,\vel_2,\varphi_1) \quad\text{in } \bbV' \text{ a.e. in } (0,T),
	   \end{align*}
	   where the right-hand side is in $\rmL^2(0,T;\rmL^2(\bbT)^d \times \rmL^2(\bbT)^d \times \rmH^1(\bbT))$. The idea is to use elliptic regularity theory for each component to conclude the claimed higher regularity in space.
		
	   Starting with the third component, we observe that it corresponds to the equation
	   \begin{align*}
		  - 2\int_\bbT \Div(\tphi1\vel_1)\Lap\psi \dx + 2\varepsilon \int_\bbT \Lap\varphi_1\Lap\psi \dx
		  &= - 2\int_\bbT g_1\Lap\psi \dx + 2\int_\bbT \delt\varphi_1\Lap\psi \dx
	   \end{align*}
	   a.e.~in $(0,T)$ for all $\psi\in\bbV_2$, which is equivalent to
	   \begin{align*}
		  \varepsilon \ang{\Grad\Lap\varphi_1}{\Grad\psi}_{\rmH^{-1}\ofbbT,\rmH^1\ofbbT}
		  = \bigang{\Grad\Div(\tphi1\vel_1) - \Grad g_1 + \Grad\delt\varphi_1}{\Grad\psi}_{\rmH^{-1}\ofbbT,\rmH^1\ofbbT}.
	   \end{align*}
	   By means of elliptic regularity theory and after an integration over $(0,T)$, we conclude that~$\Lap\varphi_1\in\rmL^2\prt{0,T;\rmH^1(\bbT)}$ and thus $\varphi_1\in\rmL^2\prt{0,T;\rmH^3(\bbT)}$.
		
	   To study the first and second component, we note that the assumptions of Lemma~\ref{Lin:lemma:pressure_fixed_t} are met. Thus, the reconstruction of a unique pressure~$p\in\rmH^1_{(0)}(\bbT)$ a.e.~in $(0,T)$ is possible such that $(\vel_1,\vel_2,\varphi_1,p)$ solves the equation~\eqref{Lin:eq:eq_without_div_cond_fixed_t_1} a.e.~in $(0,T)$ for all $\w_1\in\rmH^1(\bbT)^d$. Rearranging leads to
	   \begin{align*}
	       &\int_\bbT 2\tnu1\D\vel_1:\D\w_1 + \tlambda1\Div\vel_1\Div\w_1 \dx \\
           &= \int_\bbT \f_1 \cdot \w_1 \dx
			- \int_\bbT \trho1\delt\vel_1 \cdot \w_1 \dx
			- \int_\bbT \tphi1\Grad p\cdot\w_1 \dx
			+ \int_\bbT \tphi1 \w_1 \cdot \Grad\Lap\varphi_1 \dx
	   \end{align*}
	   a.e.~in $(0,T)$ for all $\w_1\in\rmH^1(\bbT)^d$. Since this is a weakly formulated elliptic equation with right-hand side in $\rmL^2(\bbT)^d$, it follows by elliptic regularity theory that $\vel_1\in\rmH^2(\bbT)^d$, and integrating with respect to time yields the claimed regularity. The corresponding equation for $\vel_2$ accordingly proves that~$\vel_2\in\rmL^2\prt{0,T;\rmH^2(\bbT)^d}$.
	\end{proof}

	\begin{remark}
		\label{Lin:rmk:weak_formulation_rewritten}
		Under the assumptions of Lemma~\ref{Lin:lemma:existence_abstr}~\ref{Lin:item:existence_abstr_iii}, the weak solution $(\vel_1,\vel_2,\varphi_1)$ to the abstract equation~\eqref{Lin:eq:abstract_eq} is sufficiently regular to rewrite the weak formulation~\eqref{Lin:eq:abstract_eq_weak} such that testing with less regular triples~$(\w_1,\w_2,\psi)\in\bbH$ is possible. This reformulated weak formulation then reads
		\begin{align}
			\label{Lin:eq:weak_rewritten}
			&\sum_{j=1,2} \int_\bbT \trho{j}\delt\vel_j \cdot \w_j \dx
			+ 2\int_\bbT  \Grad\delt\varphi_1 \cdot \Grad\psi \dx \nonumber\\
			&\quad - \sum_{j=1,2} \int_\bbT \Div\bigprt{ 2\tnu{j}\D\vel_j + \tlambda{j}\Div\vel_j\Id } \cdot \w_j \dx 
			- 2\varepsilon \int_\bbT  \Grad\Lap\varphi_1 \cdot\Grad\psi \dx \nonumber\\
			&\quad- \int_\bbT (  \tphi1 \w_1 -  \tphi2 \w_2) \cdot \Grad\Lap\varphi_1 \dx 
			+ 2\int_\bbT  \Lap(\tphi1\vel_1) \cdot \Grad\psi \dx \nonumber\\
			&= \sum_{j=1,2} \int_\bbT \f_j \cdot \w_j \dx + 2\int_\bbT  \Grad g_1\cdot\Grad\psi \dx.
		\end{align}
	\end{remark}

	\medskip
	\noindent\textbf{First uniform regularity estimates.}
    By means of a typical energy estimate, we obtain the first regularity estimates for the weak solution, which are uniform with respect to the regularization parameter $\varepsilon$.
	
	\begin{lemma}[Energy estimate]
		\label{Lin:lemma:energy_est}
		Let $0<T<\infty$ and $\varepsilon>0$, and let $\tphi1\in\rmH^2(\bbT)$ satisfy~$\tphi1\in(0,1)$ in $\bbT$. Suppose that
        \begin{equation*}
            \begin{alignedat}{4}
                \f_j &\in\rmL^2(0,T;\rmL^2(\bbT)^d), &\quad j&=1,2, &\qquad
                g_1&\in\rmL^2(0,T;\rmH^1(\bbT)) \cap \rmH^1(0,T;\rmH^{-1}(\bbT)), \\
                \vel_{j,0}&\in \rmH^2(\bbT)^d, &\quad j&=1,2, &\qquad
                \varphi_{1,0}&\in\rmH^3(\bbT), 
            \end{alignedat}
        \end{equation*}
        and that the assumptions of Lemma~\ref{Lin:lemma:existence_abstr}~\ref{Lin:item:existence_abstr_iii} hold. Then the weak solution~$(\vel_1,\vel_2,\varphi_1)$ to the regularized system~\eqref{Lin:eq:reg_system} in the sense of Definition~\ref{Lin:def:abstract_eq_weak} satisfies the uniform \textit{a priori} estimate
		\begin{align}
			\label{Lin:eq:a_priori_est}
			&\sum_{j=1,2} \! \bigprt{\norm{\vel_j}_{\rmL^\infty(0,T;\rmL^2\ofbbTd)\!}
				+ \norm{\vel_j}_{\rmL^2(0,T;\rmH^1\ofbbTd)} } 
			 + \norm{\varphi_1}_{\rmL^\infty(0,T;\rmH^1\ofbbT)\!} 
			+ \norm{\delt\varphi_1}_{\rmL^2(0,T;\rmL^2\ofbbT)\!} 
			+ \sqrt\varepsilon \norm{\varphi_1}_{\rmL^2(0,T;\rmH^2\ofbbT)\!} \nonumber\\
			&\leq \hat C \Bigprt{ \sum_{j=1,2} \norm{\f_j}_{\rmL^2(0,T;\rmL^2\ofbbTd)} + \norm{g_1}_{\rmL^2(0,T;\rmH^1\ofbbT)} 
			+ \sum_{j=1,2} \norm{\vel_{j,0}}_{\rmL^2\ofbbTd} + \norm{\varphi_{1,0}}_{\rmH^1\ofbbT} }
		\end{align}
		with a constant $\hat C>0$ depending on $\norm{\tphi1}_{\rmH^2\ofbbT}$ and $T$.
	\end{lemma}
	
	\begin{proof}
		Testing the weak formulation~\eqref{Lin:eq:abstract_eq_weak} with the weak solution $(\vel_1,\vel_2,\varphi_1)$ itself, we get a.e.~in~$(0,T)$ that
		\begin{align*}
			&\sum_{j=1,2} \int_\bbT \trho{j}\delt\vel_j \cdot \vel_j \dx
			- 2\int_\bbT   \delt\varphi_1\Lap\varphi_1 \dx \\
			&\quad + \sum_{j=1,2} \int_\bbT 2\tnu{j}\abs{\D\vel_j}^2+ \tlambda{j}\prt{\Div\vel_j}^2 \dx 
			+ 2\varepsilon \int_\bbT \prt{\Lap\varphi_1}^2 \dx \\
			&\quad+ \int_\bbT  \Div(\tphi1\vel_1) \Lap\varphi_1 \dx - \int_\bbT  \Div(\tphi2\vel_2) \Lap\varphi_1 \dx 
			- 2\int_\bbT  \Div(\tphi1\vel_1)\Lap\varphi_1 \dx \\
			&= \sum_{j=1,2} \int_\bbT \f_j \cdot \vel_j \dx - 2\int_\bbT  g_1\Lap\varphi_1 \dx.
		\end{align*}
		Equivalently, recalling $\Div( \tphi1\vel_1 + \tphi2\vel_2)=0$, we find
		\begin{align*}
			&\sum_{j=1,2} \int_\bbT \trho{j}\delt\vel_j \cdot \vel_j \dx
			+ 2\int_\bbT  \delt\Grad\varphi_1\cdot\Grad\varphi_1 \dx 
			 + \sum_{j=1,2} \int_\bbT 2\tnu{j}\abs{\D\vel_j}^2+ \tlambda{j}\prt{\Div\vel_j}^2 \dx \\
			&\quad+ 2\varepsilon \int_\bbT \prt{\Lap\varphi_1}^2 \dx 
			= \sum_{j=1,2} \int_\bbT \f_j \cdot \vel_j \dx + 2\int_\bbT  \Grad g_1\cdot\Grad\varphi_1 \dx.
		\end{align*}
		In view of Hölder's and Young's inequalities, it follows
		\begin{align*}
			&\sum_{j=1,2} \frac12 \ddt \int_\bbT \trho{j} \abs{\vel_j}^2 \dx
			+ \ddt\int_\bbT  \abs{\Grad\varphi_1}^2 \dx \\
			&\quad + \sum_{j=1,2} \int_\bbT 2\tnu{j}\abs{\D\vel_j}^2+ \tlambda{j}\prt{\Div\vel_j}^2 \dx 
			+ 2\varepsilon \int_\bbT \prt{\Lap\varphi_1}^2 \dx \\
			&\leq C \Bigprt{ \sum_{j=1,2} \int_\bbT \abs{\f_j}^2 \dx + \int_\bbT\abs{\vel_j}^2 \dx + \int_\bbT \abs{\Grad g_1}^2 \dx + \int_\bbT \abs{\Grad\varphi_1}^2 \dx}
		\end{align*}
		a.e.~in $(0,T)$. Eventually, Korn's and Poincaré's inequalities and Gronwall's lemma yield the \textit{a priori} estimates
		\begin{align*}
			&\sum_{j=1,2} \norm{\vel_j}_{\rmL^\infty(0,T;\rmL^2\ofbbTd)}
			+ \norm{\varphi_1}_{\rmL^\infty(0,T;\rmH^1\ofbbT)} 
			+ \sum_{j=1,2} \norm{\vel_j}_{\rmL^2(0,T;\rmH^1\ofbbTd)}
			+ \sqrt{\varepsilon} \norm{\varphi_1}_{\rmL^2(0,T;\rmH^2\ofbbT)} \\
			&\leq \hat C \Bigprt{ \sum_{j=1,2} \norm{\f_j}_{\rmL^2(0,T;\rmL^2\ofbbTd)} + \norm{g_1}_{\rmL^2(0,T;\rmH^1\ofbbT)} 
			+ \sum_{j=1,2} \norm{\vel_{j,0}}_{\rmL^2\ofbbTd} + \norm{\varphi_{1,0}}_{\rmH^1\ofbbT} }.
		\end{align*}
		with a constant $\hat C>0$ depending on $\norm{\tphi1}_{\rmH^2\ofbbT}$ and $T$. Moreover, with regard to equation~\eqref{Lin:eq:reg_system_eq_deltphi1}, these estimates together with $\tphi1\in\rmH^2(\bbT)$ allow controlling
		\begin{align*}
			\norm{\delt\varphi_1}_{\rmL^2(0,T;\rmL^2\ofbbT)}
			&\leq \norm{g_1}_{\rmL^2(0,T;\rmL^2\ofbbT)} + \norm{\Div(\tphi1\vel_1)}_{\rmL^2(0,T;\rmL^2\ofbbT)} + \sqrt{\varepsilon} \norm{\Lap\varphi_1}_{\rmL^2(0,T;\rmL^2\ofbbT)} \nonumber\\
			&\leq \hat C \Bigprt{ \sum_{j=1,2} \norm{\f_j}_{\rmL^2(0,T;\rmL^2\ofbbTd)} + \norm{g_1}_{\rmL^2(0,T;\rmH^1\ofbbT)} 
			+ \sum_{j=1,2} \norm{\vel_{j,0}}_{\rmL^2\ofbbTd} + \norm{\varphi_{1,0}}_{\rmH^1\ofbbT} }.
		\end{align*}
	\end{proof}

	\subsection{Uniform Regularity Estimates for Constant Coefficients} \label{subsec:lin:est_const_coeffs}

    In order to pass to the limit $\varepsilon\to 0$ in the regularized system~\eqref{Lin:eq:reg_system}, we need \textit{a priori} estimates for the weak solution that are uniform in $\varepsilon>0$. Since these are challenging to prove, we first freeze the linearization coefficient $\tphi1$ and deduce regularity estimates for the case of constant coefficients $\tphi{j}$, $\trho{j}$, $\tnu{j}$ and $\tlambda{j}$, $j=1,2$. To this end, we assume for this entire subsection that $\tphi1\in(0,1)$ is constant.

    We begin with formally differentiating the system with respect to the spatial variable to obtain a higher energy estimate.
	
	\begin{lemma}[Higher regularity estimates in space]
		\label{Lin:lemma:est_higher}
        Let $0<T<\infty$ and $\varepsilon>0$, and let~$\tphi1\in(0,1)$ be constant. Suppose that
        \begin{equation*}
            \begin{alignedat}{4}
                \f_j &\in\rmL^2(0,T;\rmL^2(\bbT)^d), &\quad j&=1,2, &\qquad
                g_1&\in\rmL^2(0,T;\rmH^1(\bbT)) \cap \rmH^1(0,T;\rmH^{-1}(\bbT)), \\
                \vel_{j,0}&\in \rmH^2(\bbT)^d, &\quad j&=1,2, &\qquad
                \varphi_{1,0}&\in\rmH^3(\bbT), 
            \end{alignedat}
        \end{equation*}
        and that the assumptions of Lemma~\ref{Lin:lemma:existence_abstr}~\ref{Lin:item:existence_abstr_iii} hold. Then the weak solution~$(\vel_1,\vel_2,\varphi_1)$ to the regularized system~\eqref{Lin:eq:reg_system} in the sense of Definition~\ref{Lin:def:abstract_eq_weak} satisfies the uniform \textit{a priori} estimate
		\begin{align}
			\label{Lin:eq:a_priori_est_higher}
			&\sum_{j=1,2} \! \bigprt{\norm{\vel_j}_{\rmL^\infty(0,T;\rmH^1\ofbbTd)\!}
			+ \norm{\vel_j}_{\rmL^2(0,T;\rmH^2\ofbbTd)\!} } 
			+ \norm{\varphi_1}_{\rmL^\infty(0,T;\rmH^2\ofbbT)\!} 
			+ \norm{\delt\varphi_1}_{\rmL^2(0,T;\rmH^1\ofbbT)\!}
			+ \sqrt\varepsilon \norm{\varphi_1}_{\rmL^2(0,T;\rmH^3\ofbbT)\!} \nonumber \\
			&\leq C \Bigprt{ \sum_{j=1,2} \norm{\f_j}_{\rmL^2(0,T;\rmL^2\ofbbTd)} + \norm{g_1}_{\rmL^2(0,T;\rmH^1\ofbbT)} + \norm{g_1}_{\rmL^2(0,T;\rmH^1\ofbbT)}^{\frac12} \norm{\varphi_1}_{\rmL^2(0,T;\rmH^3\ofbbT)}^{\frac12} \nonumber\\
			&\qquad+ \sum_{j=1,2} \norm{\vel_{j,0}}_{\rmH^1\ofbbTd} + \norm{\varphi_{1,0}}_{\rmH^2\ofbbT} }
		\end{align}
		with a constant $C>0$ depending on $T$.
	\end{lemma}
	
	\begin{proof}
		In the following, for each $k\in\{1, \dots, d\}$, we use $(-\delxk^2\vel_1, -\delxk^2\vel_2, -\delxk^2\varphi_1)\in\bbH$ as a test function in~\eqref{Lin:eq:weak_rewritten} to obtain a higher regularity with respect to the spatial variable. With~$\tphi{j}$ being constant, this is admissible, since
		\begin{align*}
			\Div\bigprt{\tphi1(-\delxk^2\vel_1) + \tphi2(-\delxk^2\vel_2)}
			= -\delxk^2 \Div(\tphi1\vel_1 + \tphi2\vel_2)
			= 0 
			\quad\text{in }\calD'(\bbT).
		\end{align*}
		Proceeding similarly to the energy estimate, see Lemma~\ref{Lin:lemma:energy_est}, inserting the test function described above gives
		\begin{align*}
			&-\sum_{j=1,2}  \int_\bbT \trho{j}\delt\vel_j \cdot \delxk^2\vel_j \dx
			- 2\int_\bbT  \Grad\delt\varphi_1 \cdot \Grad\delxk^2\varphi_1 \dx \\
			&\quad + \sum_{j=1,2} \int_\bbT \Div\bigprt{ 2\tnu{j}\D\vel_j + \tlambda{j}\Div\vel_j\Id } \cdot \delxk^2\vel_j \dx 
			+ 2\varepsilon \int_\bbT  \Grad\Lap\varphi_1 \cdot\Grad\delxk^2\varphi_1 \dx \\
			&\quad+ \int_\bbT \bigprt{  \tphi1 \delxk^2\vel_1 -  \tphi2 \delxk^2\vel_2} \cdot \Grad\Lap\varphi_1 \dx 
			- 2\int_\bbT  \Lap(\tphi1\vel_1) \cdot \Grad\delxk^2\varphi_1 \dx \\
			&= - \sum_{j=1,2} \int_\bbT \f_j \cdot \delxk^2\vel_j \dx 
			- 2\int_\bbT  \Grad g_1\cdot\Grad\delxk^2\varphi_1 \dx,
		\end{align*}
		which is equivalent to
		\begin{align*}
			&\sum_{j=1,2}  \int_\bbT \trho{j}\delt\delxk\vel_j \cdot \delxk\vel_j \dx
			+ 2\int_\bbT  \delt\Grad\delxk\varphi_1 \cdot \Grad\delxk\varphi_1 \dx \\
			&\quad + \sum_{j=1,2} \int_\bbT 2\tnu{j}\abs{\D\delxk\vel_j}^2 + \tlambda{j}\prt{\Div\delxk\vel_j}^2 \dx 
			+ 2\varepsilon \int_\bbT  \prt{\Lap\delxk\varphi_1}^2 \dx \\
			&\quad+ \int_\bbT \bigprt{  \Div(\tphi1 \delxk\vel_1) -  \Div(\tphi2 \delxk\vel_2)} \Lap\delxk\varphi_1 \dx 
			- 2\int_\bbT  \Div(\tphi1\delxk\vel_1) \Lap\delxk\varphi_1 \dx \\
			&= - \sum_{j=1,2} \int_\bbT \f_j \cdot \delxk^2\vel_j \dx 
			- 2\int_\bbT  \Grad g_1\cdot\Grad\delxk^2\varphi_1 \dx,
		\end{align*}
		where the third line vanishes due to $-\Div(\tphi2 \delxk\vel_2) = \Div(\tphi1 \delxk\vel_1)$. Hence, this leads to
		\begin{align*}
			&\sum_{j=1,2} \frac12 \ddt \int_\bbT \trho{j} \abs{\delxk\vel_j}^2 \dx
			+ \ddt\int_\bbT  \abs{\Grad\delxk\varphi_1}^2 \dx \nonumber\\
			&\quad + \sum_{j=1,2} \int_\bbT 2\tnu{j}\abs{\D\delxk\vel_j}^2+ \tlambda{j}\prt{\Div\delxk\vel_j}^2 \dx 
			+ 2\varepsilon \int_\bbT \prt{\Lap\delxk\varphi_1}^2 \dx \nonumber\\
			&\leq C \Bigprt{ \sum_{j=1,2} C_\delta\int_\bbT \abs{\f_j}^2 \dx + \delta\int_\bbT\abs{\delxk^2\vel_j}^2 \dx 
				+ \int_\bbT  \abs{\Grad g_1}\abs{\Grad\delxk^2\varphi_1} \dx },
		\end{align*}
		using Young's inequality, followed by an absorption after choosing $\delta$ sufficiently small. Subsequently, we apply Korn's and Poincaré's inequalities as well as Gronwall's lemma as in the energy estimate. Gathering the obtained estimates for all $k\in\{1,\dots,d\}$ and invoking the energy estimate from Lemma~\ref{Lin:lemma:energy_est}, we end up with
		\begin{align*}
			&\sum_{j=1,2} \norm{\vel_j}_{\rmL^\infty(0,T;\rmH^1\ofbbTd)}
			+ \norm{\varphi_1}_{\rmL^\infty(0,T;\rmH^2\ofbbT)} 
			+ \sum_{j=1,2} \norm{\vel_j}_{\rmL^2(0,T;\rmH^2\ofbbTd)}
			+ \sqrt\varepsilon \norm{\varphi_1}_{\rmL^2(0,T;\rmH^3\ofbbT)} \\
			&\leq C \Bigprt{ \sum_{j=1,2} \norm{\f_j}_{\rmL^2(0,T;\rmL^2\ofbbTd)} 
			+ \norm{g_1}_{\rmL^2(0,T;\rmH^1\ofbbT)}^{\frac12} \norm{\varphi_1}_{\rmL^2(0,T;\rmH^3\ofbbT)}^{\frac12} 
			+ \sum_{j=1,2} \norm{\vel_{j,0}}_{\rmH^1\ofbbTd} + \norm{\varphi_{1,0}}_{\rmH^2\ofbbT} }
		\end{align*}
		with a constant $C>0$ depending on $T$. In addition, combining these estimates with equation~\eqref{Lin:eq:reg_system_eq_deltphi1} reveals
		\begin{align*}
			\norm{\delt\varphi_1}_{\rmL^2(0,T;\rmH^1\ofbbT)} 
			&\leq \norm{g_1}_{\rmL^2(0,T;\rmH^1\ofbbT)} + \norm{\Div(\tphi1\vel_1)}_{\rmL^2(0,T;\rmH^1\ofbbT)} + \sqrt\varepsilon \norm{\Lap\varphi_1}_{\rmL^2(0,T;\rmH^1\ofbbT)} \nonumber\\
			&\leq C \Bigprt{ \sum_{j=1,2} \norm{\f_j}_{\rmL^2(0,T;\rmL^2\ofbbTd)} + \norm{g_1}_{\rmL^2(0,T;\rmH^1\ofbbT)} + \norm{g_1}_{\rmL^2(0,T;\rmH^1\ofbbT)}^{\frac12} \norm{\varphi_1}_{\rmL^2(0,T;\rmH^3\ofbbT)}^{\frac12} \\
			&\qquad+ \sum_{j=1,2} \norm{\vel_{j,0}}_{\rmH^1\ofbbTd} + \norm{\varphi_{1,0}}_{\rmH^2\ofbbT} }.
		\end{align*}
	\end{proof}

	The main difficulty of this subsection is to show a uniform \textit{a priori} estimate for $\Grad\Lap\varphi_1$. To achieve this, we use a specific testing procedure that leads to a weak damped plate equation for $\varphi_1$. The choice of a suitable test function is then crucial for establishing the desired estimate.
	
	\begin{lemma}[Estimate for $\Grad\Lap\varphi_1$]
		\label{Lin:lemma:est_grad_Lap_phi1}
		Let $0<T<\infty$ and $\varepsilon>0$, and let $\tphi1\in(0,1)$ be constant. Suppose that
        \begin{equation*}
            \begin{alignedat}{4}
                \f_j &\in\rmL^2(0,T;\rmL^2(\bbT)^d), &\quad j&=1,2, &\qquad
                g_1&\in\rmL^2(0,T;\rmH^1(\bbT)) \cap \rmH^1(0,T;\rmH^{-1}(\bbT)), \\
                \vel_{j,0}&\in \rmH^2(\bbT)^d, &\quad j&=1,2, &\qquad
                \varphi_{1,0}&\in\rmH^3(\bbT), 
            \end{alignedat}
        \end{equation*}
        and that the assumptions of Lemma~\ref{Lin:lemma:existence_abstr}~\ref{Lin:item:existence_abstr_iii} hold. Then the weak solution~$(\vel_1,\vel_2,\varphi_1)$ to the regularized system~\eqref{Lin:eq:reg_system} in the sense of Definition~\ref{Lin:def:abstract_eq_weak} satisfies the uniform \textit{a priori} estimate
		\begin{align}
			\label{Lin:eq:a_priori_est_phi1}
			&\norm{\delt\varphi_1}_{\rmL^\infty(0,T;\rmL^2\ofbbT)}
			+ \norm{\Lap\varphi_1}_{\rmL^\infty(0,T;\rmL^2\ofbbT)} 
			+ \norm{\Grad\Lap\varphi_1}_{\rmL^2(0,T;\rmL^2\ofbbTd)}
			\!+ \!\sqrt\varepsilon \norm{\Grad\delt\varphi_1}_{\rmL^2(0,T;\rmL^2\ofbbT)} \nonumber \\
			&\leq C \Bigprt{ \sum_{j=1,2} \norm{\f_j}_{\rmL^2(0,T;\rmL^2\ofbbTd)} + \norm{g_1}_{\rmL^2(0,T;\rmH^1\ofbbT)} + \norm{g_1}_{\rmH^1(0,T;\rmH^{-1}\ofbbT)} \nonumber \\
			&\qquad+ \sum_{j=1,2} \norm{\vel_{j,0}}_{\rmH^1\ofbbTd} + \norm{\varphi_{1,0}}_{\rmH^2\ofbbT} 
			+ \norm{\delt\varphi_1\vert_{t=0}}_{\rmL^2\ofbbT} } 
		\end{align}
		with a constant $C>0$ depending on $T$.
	\end{lemma}
	
	\begin{proof}
		The strategy of this proof is to write the divergence of a suitably weighted sum of~\eqref{Lin:eq:reg_system_eq_deltv1} and \eqref{Lin:eq:reg_system_eq_deltv2} as a damped plate equation
		\begin{align*}
			\alpha\delt^2\varphi_1 - (\beta+\varepsilon \alpha) \Lap\delt\varphi_1 + (\gamma+\varepsilon \beta) \Lap^2\varphi_1 = f,
		\end{align*}
		with coefficients $\alpha,\beta,\gamma>0$ and some right-hand side $f$. Therein, we add and subtract a term~$-\alpha \Lap\delt\varphi_1$ and test with $\delt\varphi_1 - \Lap\varphi_1$.
		
		In order to achieve this goal on a rigorous level, we test the weak formulation~\eqref{Lin:eq:weak_rewritten} with~$\bigprt{\tphi1^{-1} \Grad\Psi, -\tphi2^{-1} \Grad\Psi, 0}$, where $\Psi= \delt\varphi_1 - \Lap\varphi_1$, noting that this is an admissible triple of test functions in $\bbH$. The key identities, which we employ several times in this proof, are
		\begin{align*}
			\Div \vel_1 &= \tphi1^{-1} \prt{-\delt\varphi_1 + \varepsilon\Lap\varphi_1 + g_1}, \\
			\Div \vel_2 &= -\tphi2^{-1} \prt{-\delt\varphi_1 + \varepsilon\Lap\varphi_1 + g_1}
		\end{align*}
		a.e.~in~$(0,T)\times\bbT$, derived by the equations~\eqref{Lin:eq:reg_system_eq_deltphi1} and \eqref{Lin:eq:reg_system_eq_div}. Moreover, we recall the identity $\trho{j}=\barrho{j}\tphi{j}$ and abbreviate notation by 
		\begin{alignat*}{1}
			\tant = \sum_{j=1,2}\barrho{j}\tphi{j}^{-1}, \quad
			\tbnt = \sum_{j=1,2} (2\tnu{j}+\tlambda{j}) \tphi{j}^{-2}.
		\end{alignat*}
		
		First, we consider some arbitrary test function $\Psi\in\rmC^\infty(\bbT)$ in the resulting tested equation~\eqref{Lin:eq:weak_rewritten}. A.e.~in $(0,T)$, we study the individual parts separately, beginning with
		\begin{align*}
			&\int_\bbT \trho1 \delt \vel_1 \cdot \tphi1^{-1} \Grad\Psi \dx
			- \int_\bbT \trho2 \delt \vel_2 \cdot \tphi2^{-1} \Grad\Psi \dx \\
			&= - \int_\bbT \barrho1 \delt \Div\vel_1 \Psi \dx
			+ \int_\bbT \barrho2 \delt \Div\vel_2 \Psi \dx 
			= \bigang{ \tant \delt^2\varphi_1 
			- \varepsilon \tant \delt\Lap\varphi_1 
			- \tant \delt g_1 }{ \Psi }_{\rmH^{-1}\ofbbT,\rmH^1\ofbbT}.
		\end{align*}
		For the second part, we exploit, obtained by an integration by parts and using Schwarz's theorem, the two identities~$(\Lap\vel_j, \Grad\Psi)_{\rmL^2\ofbbT} = -(\Grad\vel_j, \Grad^2\Psi)_{\rmL^2\ofbbT} = (\Grad\Div\vel_j, \Grad\Psi)_{\rmL^2\ofbbT} $ to compute
		\begin{align*}
			&-\int_\bbT \Div\bigprt{ 2\tnu{j}\D\vel_j + \tlambda{j}\Div\vel_j\Id } \cdot \tphi{j}^{-1} \Grad\Psi \dx \\
			&= -\int_\bbT \bigprt{\tnu{j}\Lap\vel_j + (\tnu{j}+ \tlambda{j})\Grad\Div\vel_j} \cdot \tphi{j}^{-1} \Grad\Psi \dx \\
			&= -\int_\bbT (2\tnu{j}+\tlambda{j}) \tphi{j}^{-1} \Grad\Div\vel_j \cdot \Grad\Psi\dx,
		\end{align*}
		which leads to the expression
		\begin{align*}
			&-\int_\bbT (2\tnu1+\tlambda1) \tphi1^{-1} \Grad\Div\vel_1 \cdot \Grad\Psi\dx 
			+ \int_\bbT (2\tnu2+\tlambda2) \tphi2^{-1} \Grad\Div\vel_2 \cdot \Grad\Psi\dx \\
			&= \bigang{ - \tbnt \delt\Lap\varphi_1 
			+ \varepsilon \tbnt \Lap^2\varphi_1 
			+ \tbnt \Lap g_1 }{\Psi}_{\rmH^{-1}\ofbbT,\rmH^1\ofbbT}.
		\end{align*}
		Moreover, the third part is given by
		\begin{align*}
			- \int_\bbT \bigprt{  \tphi1 \tphi1^{-1} \Grad\Psi +  \tphi2 \tphi2^{-1} \Grad\Psi } \cdot \Grad\Lap\varphi_1 \dx 
			= \bigang{ 2 \Lap^2\varphi_1 }{\Psi}_{\rmH^{-1}\ofbbT,\rmH^1\ofbbT},
		\end{align*}
		while the right-hand side reads
		\begin{align*}
			 \int_\bbT \f_1 \cdot \tphi1^{-1} \Grad\Psi \dx
			- \int_\bbT \f_2 \cdot \tphi2^{-1} \Grad\Psi \dx 
			= \bigang{ -\tphi1^{-1} \Div\f_1 + \tphi2^{-1} \Div\f_2 }{\Psi }_{\rmH^{-1}\ofbbT,\rmH^1\ofbbT}.
		\end{align*}
		Altogether, these computations reveal the damped plate equation 
		\begin{align*}
			&\bigang{ \tant \delt^2\varphi_1 - (\tbnt + \varepsilon \tant) \delt\Lap\varphi_1 + (2+\varepsilon \tbnt)\Lap^2\varphi_1}{\Psi}_{\rmH^{-1}\ofbbT,\rmH^1\ofbbT} \\
			&= \bigang{ \tant\delt g_1 - \tbnt\Lap g_1 - \tphi1^{-1}\Div\f_1 + \tphi2^{-1}\Div\f_2 }{\Psi}_{\rmH^{-1}\ofbbT,\rmH^1\ofbbT}
		\end{align*}
		a.e.~in $(0,T)$. In a next step, moving the term $-\ang{ \tbnt\delt\Lap\varphi_1}{\Psi}_{\rmH^{-1}\ofbbT,\rmH^1\ofbbT}$ to the right-hand side and adding a term $-\ang{ \tant\delt\Lap\varphi_1}{\Psi}_{\rmH^{-1}\ofbbT,\rmH^1\ofbbT}$ on both sides, this is equivalent to 
		\begin{align*}
			&\bigang{ \tant \delt(\delt\varphi_1 - \Lap\varphi_1) - \varepsilon \tant \delt\Lap\varphi_1 + (2+\varepsilon \tbnt)\Lap^2\varphi_1}{\Psi}_{\rmH^{-1}\ofbbT,\rmH^1\ofbbT} \\
			&= \bigang{ (\tbnt - \tant) \delt\Lap\varphi_1 + \tant\delt g_1 - \tbnt\Lap g_1 - \tphi1^{-1}\Div\f_1 + \tphi2^{-1}\Div\f_2 }{\Psi}_{\rmH^{-1}\ofbbT,\rmH^1\ofbbT}.
		\end{align*}
		Since, by a density argument, this equation holds for all $\Psi\in\rmH^1$, we insert $\Psi= \delt\varphi_1 - \Lap\varphi_1$, which yields
		\begin{align*}
			&\tant \bigang{ \delt(\delt\varphi_1 - \Lap\varphi_1)}{\delt\varphi_1 - \Lap\varphi_1}_{\rmH^{-1}\ofbbT,\rmH^1\ofbbT} 
			+ 2 \bigang{ \Lap^2\varphi_1}{\delt\varphi_1}_{\rmH^{-1}\ofbbT,\rmH^1\ofbbT} \\
			&\quad- 2 \bigang{ \Lap^2\varphi_1}{\Lap\varphi_1}_{\rmH^{-1}\ofbbT,\rmH^1\ofbbT} 
			- \varepsilon \tant \bigang{ \delt\Lap\varphi_1 }{\delt\varphi_1}_{\rmH^{-1}\ofbbT,\rmH^1\ofbbT} \\
			&\quad+ \varepsilon ( \tant + \tbnt) \bigang{ \delt\Lap\varphi_1 }{\Lap\varphi_1}_{\rmH^{-1}\ofbbT,\rmH^1\ofbbT} 
			- \varepsilon \tbnt \bigang{ \Lap^2\varphi_1 }{\Lap\varphi_1}_{\rmH^{-1}\ofbbT,\rmH^1\ofbbT} \\
			&= \bigang{ (\tbnt - \tant) \delt\Lap\varphi_1 - \tbnt\Lap g_1 - \tphi1^{-1}\Div\f_1 + \tphi2^{-1}\Div\f_2 }{\delt\varphi_1 - \Lap\varphi_1}_{\rmH^{-1}\ofbbT,\rmH^1\ofbbT} \\
			&\quad+ \tant \bigang{ \delt g_1 }{\delt\varphi_1 - \Lap\varphi_1}_{\rmH^{-1}\ofbbT,\rmH^1\ofbbT}.
		\end{align*}
		Performing some suitable integrations by parts, combined with Hölder's and Young's inequalities, gives
		\begin{align*}
			&\frac{\tant}{2}\ddt \int_\bbT \abs{\delt\varphi_1 - \Lap\varphi_1}^2 \dx
			+ \ddt \int_\bbT \abs{\Lap\varphi_1 }^2 \dx
			+ 2\int_\bbT  \abs{\Grad\Lap\varphi_1}^2 \dx \\
			&\quad+ \varepsilon \tant \int_\bbT \abs{\Grad\delt\varphi_1}^2 \dx 
			+ \varepsilon \frac{\tant + \tbnt}{2} \ddt \int_\bbT \abs{\Lap\varphi_1}^2 \dx
			+ \varepsilon \tbnt \int_\bbT \abs{\Grad\Lap\varphi_1}^2 \dx \\
			&\leq C_\delta \int_\bbT \abs{\Grad\delt\varphi_1}^2 + \abs{\Grad g_1}^2 + \sum_{j=1,2} \abs{\f_j}^2  \dx 
			+ \delta \int_\bbT \abs{ \Grad(\delt\varphi_1 - \Lap\varphi_1)}^2 \dx \\
			&\quad+ \bigang{ \tant\delt g_1 }{\delt\varphi_1 - \Lap\varphi_1}_{\rmH^{-1}\ofbbT,\rmH^1\ofbbT}
		\end{align*}
		a.e.~in $(0,T)$. Estimating the last line against $C_\delta \bigprt{\norm{\delt g_1}_{\rmH^{-1}\ofbbT}^2 + \norm{\delt\varphi_1}_{\rmH^1\ofbbT}^2} +  \delta \norm{\Lap\varphi_1}_{\rmH^1\ofbbT}^2$, we absorb the latter term as well as the $\Grad\Lap\varphi_1$-term on the right-hand side of the equation. In view of Gronwall's lemma, we eventually find
		\begin{align*}
			&\norm{\delt\varphi_1-\Lap\varphi_1}_{\rmL^\infty(0,T;\rmL^2\ofbbT)}
			+ \norm{\Lap\varphi_1}_{\rmL^\infty(0,T;\rmL^2\ofbbT)} 
			+ \norm{\Grad\Lap\varphi_1}_{\rmL^2(0,T;\rmL^2\ofbbTd)}\\
			&\quad+ \sqrt\varepsilon \norm{\Grad\delt\varphi_1}_{\rmL^2(0,T;\rmL^2\ofbbT)}
			+ \sqrt\varepsilon \norm{\Lap\varphi_1}_{\rmL^\infty(0,T;\rmL^2\ofbbT)}
			+ \sqrt\varepsilon \norm{\Grad\Lap\varphi_1}_{\rmL^2(0,T;\rmL^2\ofbbTd)} \\
			&\leq C \Bigprt{ \sum_{j=1,2} \norm{\f_j}_{\rmL^2(0,T;\rmL^2\ofbbTd)} + \norm{g_1}_{\rmL^2(0,T;\rmH^1\ofbbT)} + \norm{\delt g_1}_{\rmL^2(0,T;\rmH^{-1}\ofbbT)}  \\
			&\qquad+ \norm{\delt\varphi_1}_{\rmL^2(0,T;\rmH^1\ofbbT)}
			+ \sum_{j=1,2} \norm{\vel_{j,0}}_{\rmH^1\ofbbTd} + \norm{\varphi_{1,0}}_{\rmH^2\ofbbT}  
			+ \norm{\delt\varphi_1\vert_{t=0}}_{\rmL^2\ofbbT} }.
		\end{align*}
		Finally, for the $\delt\varphi_1$ term on the right-hand side, the \textit{a priori} estimate \eqref{Lin:eq:a_priori_est_higher} from Lemma~\ref{Lin:lemma:est_higher} gives rise to the product~$\norm{g_1}_{\rmL^2(0,T;\rmH^1\ofbbT)}^{\frac12} \norm{\varphi_1}_{\rmL^2(0,T;\rmH^3\ofbbT)}^{\frac12}$, for which another application of Hölder and Young allows for a further absorption of the term $\norm{\varphi_1}_{\rmL^2(0,T;\rmH^3\ofbbT)}$. This concludes the desired estimate.
	\end{proof}

	\medskip
	\noindent\textbf{Reconstruction of the pressure.}	
    In the weak formulation of the regularized system, we aim to allow for arbitrary independent test functions that are not required to satisfy the divergence condition. To this end, we reconstruct the pressure as follows.
    
	\begin{lemma}[Existence of $p$ and estimate for $p$] 
		\label{Lin:lemma:pressure}
        Let $0<T<\infty$, $\varepsilon>0$, and let~$\tphi1\in(0,1)$ be constant. Suppose that
        \begin{equation*}
            \begin{alignedat}{4}
                \f_j &\in\rmL^2(0,T;\rmL^2(\bbT)^d), &\quad j&=1,2, &\qquad
                g_1&\in\rmL^2(0,T;\rmH^1(\bbT)) \cap \rmH^1(0,T;\rmH^{-1}(\bbT)), \\
                \vel_{j,0}&\in \rmH^2(\bbT)^d, &\quad j&=1,2, &\qquad
                \varphi_{1,0}&\in\rmH^3(\bbT), 
            \end{alignedat}
        \end{equation*}
        and that the assumptions of Lemma~\ref{Lin:lemma:existence_abstr}~\ref{Lin:item:existence_abstr_iii} hold. Then, denoting by $(\vel_1,\vel_2,\varphi_1)$ the weak solution to the regularized system~\eqref{Lin:eq:reg_system} in the sense of Definition~\ref{Lin:def:abstract_eq_weak}, there exists a unique~$p\in\rmL^2(0,T;\rmH^1_{(0)}(\bbT))$ such that $(\vel_1,\vel_2,\varphi_1,p)$ solves \eqref{Lin:eq:reg_system_eq_deltv1}--\eqref{Lin:eq:reg_system_eq_div} a.e.~in~$(0,T)\times\bbT$. Moreover, $p$ satisfies the uniform \textit{a priori} estimate
		\begin{align}
			\label{Lin:eq:a_priori_est_pressure}
			\norm{p}_{\rmL^2(0,T;\rmH^1\ofbbT)}
			&\leq C \Bigprt{ \sum_{j=1,2} \norm{\f_j}_{\rmL^2(0,T;\rmL^2\ofbbTd)} + \norm{g_1}_{\rmL^2(0,T;\rmH^1\ofbbT)} + \norm{g_1}_{\rmH^1(0,T;\rmH^{-1}\ofbbT)} \nonumber \\
			&\qquad+ \sum_{j=1,2} \norm{\vel_{j,0}}_{\rmH^1\ofbbTd} + \norm{\varphi_{1,0}}_{\rmH^2\ofbbT} 
			+ \norm{\delt\varphi_1\vert_{t=0}}_{\rmL^2\ofbbT} } 
		\end{align}
		with a constant $C>0$ depending on $T$.
	\end{lemma}
	
	\begin{proof}
		The idea is to redo the pressure reconstruction in the proof of Lemma~\ref{Lin:lemma:pressure_fixed_t} and to additionally establish a uniform estimate for $p$.
		
		We first note that testing the weak formulation~\eqref{Lin:eq:weak_rewritten} with arbitrary $(\hat\w_1,\hat\w_2,0)\in\bbV$ shows that $(\vel_1,\vel_2,\varphi_1)$ solves 
		\begin{align}
			\label{Lin:eq:eq_for_hatwj}
			&\sum_{j=1,2} \int_\bbT \trho{j}\delt\vel_j \cdot \hat\w_j \dx
			- \sum_{j=1,2} \int_\bbT \Div\bigprt{ 2\tnu{j}\D\vel_j + \tlambda{j}\Div\vel_j\Id } \cdot \hat\w_j \dx \nonumber \\
			&\quad- \int_\bbT (  \tphi1 \hat\w_1 -  \tphi2 \hat\w_2) \cdot \Grad\Lap\varphi_1 \dx
			= \sum_{j=1,2} \int_\bbT \f_j \cdot \hat\w_j \dx
		\end{align}
		a.e.~in $(0,T)$ for all $(\hat\w_1,\hat\w_2)\in\bbV_1$, i.e., for all $(\hat\w_1,\hat\w_2)\in \rmH^1(\bbT)^d \times \rmH^1(\bbT)^d$ satisfying the condition $\Div(\tphi1\hat\w_1+\tphi2\hat\w_2)=0$. Our goal now is to prove the existence of a unique solution~$p\in\rmH^1_{(0)}(\bbT)$ to the following equation, which is formally obtained by summing~\eqref{Lin:eq:reg_system_eq_deltv1} tested with $\barrho1^{-1}\Grad\xi$ and \eqref{Lin:eq:reg_system_eq_deltv2} tested with $\barrho2^{-1}\Grad\xi$ for arbitrary $\xi\in\rmH^1(\bbT)$, and using $\int_\bbT \delt\Div(\tphi1\vel_1+\tphi2\vel_2) \xi \dx=0$, that is
		\begin{align}
			\label{Lin:eq:eq_for_pressure}
			&\int_\bbT \tant \Grad p\cdot\Grad\xi\dx
			= \sum_{j=1,2} \int_\bbT \barrho{j}^{-1}\Div\bigprt{ 2\tnu{j}\D\vel_j + \tlambda{j}\Div\vel_j\Id } \cdot \Grad\xi \dx \nonumber \\
			&\quad+ \int_\bbT \tant \Grad\Lap\varphi_1 \cdot \Grad\xi \dx
			+ \sum_{j=1,2} \int_\bbT  \barrho{j}^{-1} \f_j \cdot \Grad\xi\dx
		\end{align}
		a.e.~in $(0,T)$ for all $\xi\in\rmH^1(\bbT)$, where $\tant\coloneqq \barrho1^{-1}\tphi1 + \barrho2^{-1}\tphi2 >0$. Writing the right-hand side as $\int_\bbT \mathbf F\cdot\Grad\xi \dx$, the uniform estimates from Lemma~\ref{Lin:lemma:est_higher} for $\norm{\vel_j}_{\rmL^2(0,T;\rmH^2\ofbbTd)}$, $j=1,2$, and from Lemma~\ref{Lin:lemma:est_grad_Lap_phi1} for~$\norm{\Grad\Lap\varphi_1}_{\rmL^2(0,T;\rmL^2\ofbbTd)}$ reveal
		\begin{align*}
			\norm{\mathbf F}_{\rmL^2(0,T;\rmL^2\ofbbTd)}
			&\leq C \Bigprt{ \sum_{j=1,2} \norm{\f_j}_{\rmL^2(0,T;\rmL^2\ofbbTd)} + \norm{g_1}_{\rmL^2(0,T;\rmH^1\ofbbT)} + \norm{g_1}_{\rmH^1(0,T;\rmH^{-1}\ofbbT)} \\
			&\qquad+ \sum_{j=1,2} \norm{\vel_{j,0}}_{\rmH^1\ofbbTd} + \norm{\varphi_{1,0}}_{\rmH^2\ofbbT} 
			+ \norm{\delt\varphi_1\vert_{t=0}}_{\rmL^2\ofbbT} }.
		\end{align*}
		For this weak elliptic equation on $\bbT$ with $\tant>0$, the Lax--Milgram theorem and Poincaré's inequality imply that \eqref{Lin:eq:eq_for_pressure} has a unique solution $p\in\rmH^1_{(0)}(\bbT)$, which satisfies
		\begin{align*}
			\norm{\Grad p}_{\rmL^2\ofbbT}
			\leq C \norm{\mathbf F}_{\rmL^2\ofbbT}.
		\end{align*}
		Thus, invoking Poincaré's inequality, we have $p\in\rmL^2(0,T;\rmH^1_{(0)}(\bbT))$ with
		\begin{align*}
			\norm{p}_{\rmL^2(0,T;\rmH^1\ofbbT)}
			&\leq C \Bigprt{ \sum_{j=1,2} \norm{\f_j}_{\rmL^2(0,T;\rmL^2\ofbbTd)} + \norm{g_1}_{\rmL^2(0,T;\rmH^1\ofbbT)} + \norm{g_1}_{\rmH^1(0,T;\rmH^{-1}\ofbbT)} \\
			&\qquad+ \sum_{j=1,2} \norm{\vel_{j,0}}_{\rmH^1\ofbbTd} + \norm{\varphi_{1,0}}_{\rmH^2\ofbbT} 
			+ \norm{\delt\varphi_1\vert_{t=0}}_{\rmL^2\ofbbT} }.
		\end{align*}
		
		Now let the couple $(\w_1,\w_2)\in\rmH^1(\bbT)^d \times \rmH^1(\bbT)^d$ be arbitrary and decompose $\w_j$ into~$\w_j=\hat\w_j + \barrho{j}^{-1}\Grad\xi$ such that $\Div(\tphi1\hat\w_1+\tphi2\hat\w_2)=0$ (that is, $\xi$ is the solution of~$\Div\bigprt{\prt{\barrho1^{-1}\tphi1 + \barrho2^{-1}\tphi2}  \Grad\xi } = \Div(\tphi1\w_1+\tphi2\w_2)$). Then, with $p$ from above, it holds
		\begin{align}
			\label{Lin:eq:eq_without_div_cond}
			&\sum_{j=1,2} \int_\bbT \trho{j}\delt\vel_j \cdot \w_j \dx
			- \sum_{j=1,2} \int_\bbT \Div\bigprt{ 2\tnu{j}\D\vel_j + \tlambda{j}\Div\vel_j\Id } \cdot \w_j \dx \nonumber\\
			&- \int_\bbT (  \tphi1 \w_1 -  \tphi2 \w_2) \cdot \Grad\Lap\varphi_1 \dx
			+ \sum_{j=1,2} \int_\bbT \tphi{j}\Grad p\cdot\w_j \dx  
			= \sum_{j=1,2} \int_\bbT \f_j \cdot \w_j \dx
		\end{align}
		a.e.~in $(0,T)$, due to \eqref{Lin:eq:eq_for_hatwj} for the terms including $\hat\w_j$, \eqref{Lin:eq:eq_for_pressure} for the terms with $\Grad\xi$, and since
		\begin{align*}
			\sum_{j=1,2} \int_\bbT \tphi{j}\Grad p\cdot\hat\w_j \dx 
			= - \int_\bbT \Div(\tphi1\hat\w_1 + \tphi2\hat\w_2) p \dx
			= 0.
		\end{align*}
		The choice of $(\w_1,0)$ in \eqref{Lin:eq:eq_without_div_cond} implies that
		\begin{align*}
			&\int_\bbT \trho1\delt\vel_1 \cdot \w_1 \dx
			- \int_\bbT \Div\bigprt{ 2\tnu1\D\vel_1 + \tlambda1\Div\vel_1\Id } \cdot \w_1 \dx \\
			&\quad- \int_\bbT  \tphi1 \w_1 \cdot \Grad\Lap\varphi_1 \dx
			+ \int_\bbT \tphi1\Grad p\cdot\w_1 \dx
			= \int_\bbT \f_1 \cdot \w_1 \dx
		\end{align*}
		a.e.~in $(0,T)$ is satisfied for all $\w_1\in\rmH^1(\bbT)^d$ and therefore, equation~\eqref{Lin:eq:reg_system_eq_deltv1} holds pointwise a.e.~in $(0,T)\times\bbT$. Choosing $(0,\w_2)$ with $\w_2\in\rmH^1(\bbT)^d$ proves the corresponding result for~\eqref{Lin:eq:reg_system_eq_deltv2}.
	\end{proof}

    The uniform estimates established previously provide the basis for the derivation of further estimates.
	
	\begin{lemma}[Estimates for $\delt\vel_j$ and $\delt^2\varphi_1$]
		\label{Lin:lemma:est_deltvj}
        Let $0<T<\infty$ and $\varepsilon>0$, and let $\tphi1\in(0,1)$ be constant. Suppose that
        \begin{equation*}
            \begin{alignedat}{4}
                \f_j &\in\rmL^2(0,T;\rmL^2(\bbT)^d), &\quad j&=1,2, &\qquad
                g_1&\in\rmL^2(0,T;\rmH^1(\bbT)) \cap \rmH^1(0,T;\rmH^{-1}(\bbT)), \\
                \vel_{j,0}&\in \rmH^2(\bbT)^d, &\quad j&=1,2, &\qquad
                \varphi_{1,0}&\in\rmH^3(\bbT), 
            \end{alignedat}
        \end{equation*}
        and that the assumptions of Lemma~\ref{Lin:lemma:existence_abstr}~\ref{Lin:item:existence_abstr_iii} hold. Then the strong solution~$(\vel_1,\vel_2,\varphi_1,p)$ to the regularized system~\eqref{Lin:eq:reg_system} from Lemma~\ref{Lin:lemma:pressure} satisfies the uniform \textit{a priori} estimate
		\begin{align}
			\label{Lin:eq:a_priori_est_deltv}
			&\sum_{j=1,2} \norm{\delt\vel_j}_{\rmL^2(0,T;\rmL^2\ofbbTd)} + \norm{\delt^2\varphi_1}_{\rmL^2(0,T;\rmH^{-1}\ofbbT)} 
			\leq C \Bigprt{ \sum_{j=1,2} \norm{\f_j}_{\rmL^2(0,T;\rmL^2\ofbbTd)} + \norm{g_1}_{\rmL^2(0,T;\rmH^1\ofbbT)}  \nonumber \\
			&\qquad\qquad+ \norm{g_1}_{\rmH^1(0,T;\rmH^{-1}\ofbbT)} 
            + \sum_{j=1,2} \norm{\vel_{j,0}}_{\rmH^1\ofbbTd} + \norm{\varphi_{1,0}}_{\rmH^2\ofbbT} 
			+ \norm{\delt\varphi_1\vert_{t=0}}_{\rmL^2\ofbbT} } 
		\end{align}
		with a constant $C>0$ depending on $T$.
	\end{lemma}
	
	\begin{proof}
		As established in the previous Lemma~\ref{Lin:lemma:pressure}, the quadruple $(\vel_1,\vel_2,\varphi_1,p)$ solves the equations~\eqref{Lin:eq:reg_system_eq_deltv1}--\eqref{Lin:eq:reg_system_eq_div} a.e.~in~$(0,T)\times\bbT$. Therefore, after writing \eqref{Lin:eq:reg_system_eq_deltv1} in the form
		\begin{align*}
			\trho1\delt\vel_1 
			= \Div\bigprt{2\tnu1 \D\vel_1 + \tlambda1 \Div\vel_1 \Id} - \tphi1\Grad p +  \tphi1 \Grad\Lap\varphi_1 + \f_1,
		\end{align*}
		we deduce from the Lemmas~\ref{Lin:lemma:est_higher}, \ref{Lin:lemma:est_grad_Lap_phi1}, and \ref{Lin:lemma:pressure} that
		\begin{align*}
			\norm{\delt\vel_1}_{\rmL^2(0,T;\rmL^2\ofbbTd)} 
			&\leq C \Bigprt{ \norm{\vel_1}_{\rmL^2(0,T;\rmH^2\ofbbTd)} + \norm{p}_{\rmL^2(0,T;\rmH^1\ofbbT)} + \norm{\Grad\Lap\varphi}_{\rmL^2(0,T;\rmL^2\ofbbTd)} + \norm{\f_1}_{\rmL^2(0,T;\rmL^2\ofbbTd)}} \\
			&\leq C \Bigprt{ \sum_{j=1,2} \norm{\f_j}_{\rmL^2(0,T;\rmL^2\ofbbTd)} + \norm{g_1}_{\rmL^2(0,T;\rmH^1\ofbbT)} + \norm{g_1}_{\rmH^1(0,T;\rmH^{-1}\ofbbT)} \\
			&\qquad+ \sum_{j=1,2} \norm{\vel_{j,0}}_{\rmH^1\ofbbTd} + \norm{\varphi_{1,0}}_{\rmH^2\ofbbT} 
			+ \norm{\delt\varphi_1\vert_{t=0}}_{\rmL^2\ofbbT} }.
		\end{align*}
		Analogously, the estimate for $\delt\vel_2$ is derived from \eqref{Lin:eq:reg_system_eq_deltv2}.
		
		Furthermore, differentiating equation~\eqref{Lin:eq:reg_system_eq_deltphi1} with respect to time reveals that for any test function~$\psi\in\rmH^1(\bbT)$, it holds
		\begin{align*}
			\langle \delt^2\varphi_1, \psi \rangle_{\rmH^{-1}\ofbbT, \rmH^1\ofbbT} 
			&= \langle \delt g_1, \psi \rangle_{\rmH^{-1}\ofbbT, \rmH^1\ofbbT} 
			+ \int_\bbT \tphi1 \delt \vel_1\cdot\Grad\psi \dx
			- \varepsilon \int_\bbT \Grad\delt\varphi_1 \cdot \Grad\psi \dx \\
			&\leq \Bigprt{ \norm{\delt g_1}_{\rmH^{-1}\ofbbT} 
			+ \norm{\delt\vel_1}_{\rmL^2\ofbbT} 
			+ \varepsilon \norm{\Grad\delt\varphi_1}_{\rmL^2\ofbbT} }
			\norm{\psi}_{\rmH^1\ofbbT}
		\end{align*}
		a.e.~in $(0,T)$, and hence
		\begin{align*}
			\norm{\delt^2\varphi_1}_{\rmL^2(0,T;\rmH^{-1}\ofbbT)} 
			\leq \norm{\delt g_1}_{\rmL^2(0,T;\rmH^{-1}\ofbbT)} 
			+ \norm{\delt\vel_1}_{\rmL^2(0,T;\rmL^2\ofbbTd)} 
			+ \sqrt\varepsilon \norm{\Grad\delt\varphi_1}_{\rmL^2(0,T;\rmL^2\ofbbTd)}.
		\end{align*}
		From this, together with Lemma~\ref{Lin:lemma:est_grad_Lap_phi1}, the claimed estimate finally follows.
	\end{proof}

	\subsection{Strong Solvability of the Linear System for Constant Coefficients}
    \label{subsec:lin:solution_const_coeffs}
    With the uniform regularity estimates proven in Subsection~\ref{subsec:lin:est_const_coeffs}, we are in a position to pass to the limit $\varepsilon\to0$ in the regularized system~\eqref{Lin:eq:reg_system}. This yields a unique strong solution to the linear system~\eqref{Lin:eq:lin_system} in the case of constant coefficients.
	
	\begin{proposition}
		\label{Lin:prop:ex_const_coeff}
        Let $0<T<\infty$ and let $\tphi1\in(0,1)$ be constant. Suppose that
        \begin{equation*}
            \begin{alignedat}{4}
                \f_j &\in\rmL^2(0,T;\rmL^2(\bbT)^d), &\quad j&=1,2, &\quad
                g_j &\in\rmL^2(0,T;\rmH^1(\bbT)) \cap \rmH^1(0,T;\rmH^{-1}(\bbT)), &\quad j&=1,2,\\
                \vel_{j,0}&\in \rmH^1(\bbT)^d, &\quad j&=1,2, &\quad
                \varphi_{1,0}&\in\rmH^2(\bbT), 
            \end{alignedat}
        \end{equation*}
        such that $\Div(\tphi1\vel_{1,0} + \tphi2\vel_{2,0})=g_2\vert_{t=0}$ holds. Then there exists a unique strong solution~$(\vel_1, \vel_2, \varphi_1, p)$ to the linear system~\eqref{Lin:eq:lin_system} with
		\begin{align*}
			\vel_j &\in \rmL^2 \prt{0,T;\rmH^2(\bbT)^d} \cap \rmH^1 \prt{ 0,T;\rmL^2(\bbT)^d}, \\
			\varphi_1 &\in \rmL^2 \prt{0,T;\rmH^3(\bbT)} \cap \rmH^1 \prt{ 0,T;\rmH^1(\bbT)} \cap \rmH^2 \prt{ 0,T;\rmH^{-1}(\bbT)}, \\
			p &\in \rmL^2 \prt{0,T;\rmH^1_{(0)}(\bbT)}.
		\end{align*}
	\end{proposition}
	
	\begin{proof}
		Without loss of generality, we restrict our attention to right-hand sides satisfying more specific assumptions.
		First, we reduce to the case $g_2=0$ by replacing $\vel_j$ with $\vel_j-\Grad q_1$, where $q_1\in \rmL^2(0,T;\rmH^3(\bbT)) \cap \rmH^1(0,T;\rmH^1(\bbT))$ is a solution to $\Lap q_1=g_2$. Correspondingly, we replace the initial velocity $\vel_{j,0}$ by $\vel_{j,0}-\Grad q_1$. This modification leads to additional terms on the right-hand side of \eqref{Lin:eq:lin_system_eq_deltv1}--\eqref{Lin:eq:lin_system_eq_div}, which then takes the form $(\tilde\f_1, \tilde\f_2, \tilde g_1,0)$, where
		\begin{align*}
			\tilde\f_j &\coloneqq \f_j - \trho{j} \delt\Grad q_1 + \Div\bigprt{2\tnu{j}\D\Grad q_1 + \tlambda{j}\Div\Grad q_1\Id} \in \rmL^2(0,T;\rmL^2(\bbT)^d), \quad j=1,2,\\
			\tilde g_1 &\coloneqq g_1 + \tphi1 g_2 \in \rmL^2(0,T;\rmH^1(\bbT)) \cap \rmH^1(0,T;\rmH^{-1}(\bbT)).
		\end{align*}
		For notational convenience, we again write $(\f_1, \f_2, g_1,0)$ in place of $(\tilde\f_1, \tilde\f_2, \tilde g_1,0)$. 
		
		Next, integrating \eqref{Lin:eq:lin_system_eq_deltphi1} with respect to the spatial variable, we observe
		\begin{align*}
			\ddt \int_\bbT \varphi_1 \dx
			= \int_\bbT \delt\varphi_1 \dx
			= \int_\bbT g_1 \dx - \int_\bbT \Div(\tphi1 \vel_1) \dx
			= \int_\bbT g_1 \dx,
		\end{align*}
		for almost every $t\in[0,T]$, which implies that $\mean{\varphi_1}$ is determined by $g_1$ via
		\begin{align*}
			\mean{\varphi_1(t)}
			= \int_0^t \mean{g_1(\tau)} \dtau + \mean{\varphi_{1,0}}.
		\end{align*}
		Therefore, by replacing $(\varphi_1, \varphi_{1,0}, g_1)$ by $(\varphi_1-\mean{\varphi_1}, \varphi_{1,0}-\mean{\varphi_{1,0}}, g_1-\mean{g_1})$, it is sufficient to study the case $(\mean{\varphi_1}, \mean{\varphi_{1,0}}, \mean{g_1})= (0,0,0)$.
		
		Finally, we may also impose the condition
		\begin{align*}
			\Div( \barrho1^{-1}\f_1+ \barrho2^{-1}\f_2) =0 \quad\text{in } \mathcal D'(\bbT) \quad\text{for a.e.~}t\in[0,T] \text{ including } t=0,
		\end{align*}
		by now replacing $\f_j$ by $\f_j-\Grad q_2$, $j=1,2$, with $q_2\in\rmL^2(0,T;\rmH^1(\bbT))$ solving the equation~$(\barrho1^{-1}+\barrho2^{-1})\Lap q_2 = \Div( \barrho1^{-1}\f_1 + \barrho2^{-1}\f_2)$.
		
		To apply the abstract existence result Lemma~\ref{Lin:lemma:existence_abstr}, we approximate the data suitably. First, we consider functions~$\f_{j,\varepsilon}\in\rmL^2(0,T;\rmL^2(\bbT)^d) \cap \rmH^1(0,T;(\rmH^1(\bbT)^d)')$ such that, for each $\varepsilon>0$, it holds $\Div( \barrho1^{-1}\f_{1,\varepsilon}+ \barrho2^{-1}\f_{2,\varepsilon}) =0$ in $\mathcal D'(\bbT)$ for a.e.~$t\in[0,T]$ including~$t=0$, and such that $\f_{j,\varepsilon}\to\f_j$ in $\rmL^2(0,T;\rmL^2(\bbT)^d)$ as $\varepsilon\to0$. Moreover, the initial data are approximated by $\vel_{j,0,\varepsilon}\in \rmH^2(\bbT)^d$, $j=1,2$, such that $\Div(\tphi1\vel_{1,0,\varepsilon} + \tphi2\vel_{2,0,\varepsilon})=0$ and~$\vel_{j,0,\varepsilon}\to\vel_{j,0}$ in~$\rmH^1(\bbT)^d$. Subtracting some suitable $\Grad\xi_\varepsilon$ from $\vel_{j,0,\varepsilon}$, if necessary, ensures that the divergence condition is satisfied. Finally, we choose $\varphi_{1,0,\varepsilon}\in\rmH^3(\bbT)$ with $\varphi_{1,0,\varepsilon}\to\varphi_{1,0}$ in~$\rmH^2(\bbT)$.
		
		With these approximations, the data $\f_{j,\varepsilon},g_1$ and $\vel_{j,0,\varepsilon}, \varphi_{1,0,\varepsilon}$ satisfy the assumptions of Lemma~\ref{Lin:lemma:existence_abstr}--Lemma~\ref{Lin:lemma:est_deltvj}. Hence, for each $\varepsilon>0$, the regularized linear system~\eqref{Lin:eq:reg_system} with right-hand side $(\f_{1,\varepsilon}, \f_{2,\varepsilon},g_1,0)$ admits a strong solution $(\vel_{1,\varepsilon}, \vel_{2,\varepsilon}, \varphi_{1,\varepsilon}, p_{\varepsilon})$ with the regularity properties from Lemma~\ref{Lin:lemma:existence_abstr}.
		
		The uniform \textit{a priori} estimates from the Lemmas~\ref{Lin:lemma:est_higher}, \ref{Lin:lemma:est_grad_Lap_phi1}, \ref{Lin:lemma:pressure}, and \ref{Lin:lemma:est_deltvj} allow us to pass to non-relabeled subsequences converging weakly (or weakly-*, respectively) as $\varepsilon\to0$ as follows:
		\begin{alignat*}{3}
			\vel_{j,\varepsilon} &\overset{(\ast)\;\,}{\rightharpoonup} \vel_j &&\quad\text{in } \rmL^\infty \prt{ 0,T;\rmH^1(\bbT)^d} \cap \rmL^2 \prt{0,T;\rmH^2(\bbT)^d}, &\qquad j=1,2,\\
			\delt\vel_{j,\varepsilon} &\rightharpoonup \delt\vel_j &&\quad\text{in } \rmL^2 \prt{0,T;\rmL^2(\bbT)^d}, &\qquad j=1,2,\\
			\varphi_{1,\varepsilon} &\overset{(\ast)\;\,}{\rightharpoonup} \varphi_1 &&\quad\text{in } \rmL^\infty \prt{ 0,T;\rmH^2(\bbT)} \cap \rmL^2 \prt{0,T;\rmH^3(\bbT)}, &\\
			\delt\varphi_{1,\varepsilon} &\overset{(\ast)\;\,}{\rightharpoonup} \delt\varphi_1 &&\quad\text{in } \rmL^\infty \prt{ 0,T;\rmL^2(\bbT)} \cap \rmL^2 \prt{0,T;\rmH^1(\bbT)}, &\\
			\delt^2\varphi_{1,\varepsilon} &\rightharpoonup \delt^2\varphi_1 &&\quad\text{in } \rmL^2 \prt{0,T;\rmH^{-1}(\bbT)}, &\\
			p_\varepsilon &\rightharpoonup p &&\quad\text{in } \rmL^2 \prt{0,T;\rmH^1_{(0)}(\bbT)}. &
		\end{alignat*}
		In particular, each term in \eqref{Lin:eq:reg_system_eq_deltv1}--\eqref{Lin:eq:reg_system_eq_div} converges weakly in $\rmL^2(0,T;\rmL^2(\bbT))$ in the respective dimension, which suffices to pass to the limit in the system. Thus, the limit quadruple~$(\vel_1,\vel_2,\varphi_1,p)$ is a strong solution to \eqref{Lin:eq:lin_system_eq_deltv1}--\eqref{Lin:eq:lin_system_eq_div}.
		
		Finally, to prove uniqueness, let $(\hat\vel_1,\hat\vel_2,\hat\varphi_1,\hat p)$ be another solution to system~\eqref{Lin:eq:lin_system}, and set~$(\w_1,\w_2,\psi,q)$ to be the difference of the two solutions. Then $(\w_1,\w_2,\psi,q)$ solves the corresponding homogeneous system. The associated weak formulation, which coincides with~\eqref{Lin:eq:weak_rewritten} without the regularization term, is then tested with $(\w_1,\w_2,\psi)$ as in the proof of the energy estimate in Lemma~\ref{Lin:lemma:energy_est}. Accordingly, we obtain 
		\begin{align*}
			\sum_{j=1,2} \bigprt{\norm{\w_j}_{\rmL^\infty(0,T;\rmL^2\ofbbTd)}
				+ \norm{\w_j}_{\rmL^2(0,T;\rmH^1\ofbbTd)} }
			+ \norm{\psi}_{\rmL^\infty(0,T;\rmH^1\ofbbT)} 
			= 0,
		\end{align*}
		and therefore $(\w_1,\w_2,\psi)=(0,0,0)$. Due to $q\in \rmL^2(0,T;\rmH^1_{(0)}(\bbT))$, equation~\eqref{Lin:eq:lin_system_eq_deltv1} eventually reveals that $q$ vanishes as well.
	\end{proof}

	\subsection{Uniform Regularity Estimates for General Coefficients}
    \label{subsec:lin:est_gen_coeffs}

    Returning to the regularized system with general (nonconstant) coefficients, our next goal is to derive further uniform estimates beyond the energy estimate. In contrast to the constant-coefficient case, these estimates are of one lower spatial order.

	\medskip
	\noindent\textbf{Reconstruction of the modified pressure.}	
    In a first step, we reconstruct a modified pressure $q$ depending on the original pressure $p$ and $\Lap\varphi_1$. For $q$, we prove a uniform estimate whose right-hand side depends on $\varphi_1$.
    
	\begin{lemma}[Existence of $q$ and estimate for $q$ depending on $\varphi_1$] 
		\label{Lin:lemma:pressure_q_gen}
        Let $0<T<\infty$ and~$\varepsilon>0$, and let~$\tphi1\in\rmH^2(\bbT)$ satisfy $\tphi1\in(0,1)$ in $\bbT$. Suppose that
        \begin{equation*}
            \begin{alignedat}{4}
                \f_j &\in\rmL^2(0,T;\rmL^2(\bbT)^d), &\quad j&=1,2, &\qquad
                g_1&\in\rmL^2(0,T;\rmH^1(\bbT)) \cap \rmH^1(0,T;\rmH^{-1}(\bbT)), \\
                \vel_{j,0}&\in \rmH^2(\bbT)^d, &\quad j&=1,2, &\qquad
                \varphi_{1,0}&\in\rmH^3(\bbT), 
            \end{alignedat}
        \end{equation*}
        and that the assumptions of Lemma~\ref{Lin:lemma:existence_abstr}~\ref{Lin:item:existence_abstr_iii} hold. Then, denoting by $(\vel_1,\vel_2,\varphi_1)$ the weak solution to the regularized system~\eqref{Lin:eq:reg_system} in the sense of Definition~\ref{Lin:def:abstract_eq_weak}, there exists a unique~$q\in\rmL^2(0,T;\rmL^2_{(0)}(\bbT))$ such that $(\vel_1,\vel_2,\varphi_1,q)$ solves \eqref{Lin:eq:eq_deltv1_with_q}--\eqref{Lin:eq:eq_deltv2_with_q} below as well as \eqref{Lin:eq:reg_system_eq_deltphi1}--\eqref{Lin:eq:reg_system_eq_div} a.e.~in $(0,T)\times\bbT$. Moreover, $q$ satisfies the \textit{a priori} estimate
		\begin{align}
			\label{Lin:eq:a_priori_est_pressure_q_gen}
			\norm{q}_{\rmL^2(0,T;\rmL^2\ofbbT)} 
			&\leq \hat C \Bigprt{ \norm{ \varphi_1 }_{\rmL^2(0,T;\rmH^2\ofbbT)} + \sum_{j=1,2} \norm{\f_j}_{\rmL^2(0,T;\rmL^2\ofbbTd)} + \norm{g_1}_{\rmL^2(0,T;\rmH^1\ofbbT)} \nonumber\\
			&\qquad+ \sum_{j=1,2} \norm{\vel_{j,0}}_{\rmL^2\ofbbTd} + \norm{\varphi_{1,0}}_{\rmH^1\ofbbT} }
		\end{align}
		with a constant $\hat C>0$ depending on $\norm{\tphi1}_{\rmH^2\ofbbT}$ and $T$.
	\end{lemma}
	
	\begin{proof}
		The strategy of this proof is similar to that in Lemma~\ref{Lin:lemma:pressure} and differs in the regularity of the (modified) pressure. First, we recall that the weak solution $(\vel_1,\vel_2,\varphi_1)$ has the regularity stated in the abstract existence result Lemma~\ref{Lin:lemma:existence_abstr}.
		For the sake of brevity, we write in the following
		\begin{align*}
			\tant \coloneqq \barrho1^{-1} \tphi1 + \barrho2^{-1} \tphi2,
			\qquad  \tdnt \coloneqq \barrho1^{-1} \tphi1 - \barrho2^{-1} \tphi2.
		\end{align*}		
		In the sense of distributions, multiplying \eqref{Lin:eq:reg_system_eq_deltv1} by $\barrho1^{-1}$ and \eqref{Lin:eq:reg_system_eq_deltv2} by $\barrho2^{-1}$ and considering the sum, followed by taking the negative divergence such that the $\delt\vel_j$-terms vanish, we observe that the weak solution solves the equation
		\begin{align*}
			-\Div\bigprt{\tant \Grad p - \tdnt \Grad\Lap\varphi_1 }
			= \Div \Bigprt{- \sum_{j=1,2} \Bigprt{ \barrho{j}^{-1} \Div \bigprt{2\tnu{j}\D\vel_j + \tlambda{j} \Div\vel_j\Id} + \barrho{j}^{-1}\f_j }}
		\end{align*}
		in $\calD'(\bbT)$. Defining the modified pressure
		\begin{align*}
			q\coloneqq p - \frac{\tdnt }{\tant } \Lap\varphi_1 ,
		\end{align*}
		this is equivalent to
		\begin{align*}
			-\Div \bigprt{ \tant  \Grad q } 
			= \Div \biggprt{ \tant  \Grad\frac{\tdnt}{\tant} \Lap\varphi_1 - \sum_{j=1,2} \Bigprt{ \barrho{j}^{-1} \Div \bigprt{2\tnu{j}\D\vel_j + \tlambda{j} \Div\vel_j\Id} + \barrho{j}^{-1}\f_j } }
			\eqqcolon G 
		\end{align*}
		in $\calD'(\bbT)$. Since $\tant>0$ in $\bbT$ and $\tant\in\rmH^2(\bbT)\hookrightarrow\rmW^{1,6}(\bbT)$, Proposition~\ref{prelim:prop:very_weak_sol} yields that this equation has a unique very weak solution $q\in \rmL^2_{(0)}(\bbT)$ satisfying 
		\begin{align*}
			\norm{q}_{\rmL^2\ofbbT}
			\leq C \norm{G}_{\rmH^{-2}\ofbbT}.
		\end{align*}
		The corresponding very weak formulation reads
		\begin{align*}
			&- \int_\bbT q \Div \bigprt{ \tant  \Grad \xi } \dx
			= \ang{G}{\xi}_{\rmH^{-2}\ofbbT, \rmH^2\ofbbT} \\
			&= \int_\bbT \biggprt{ \bigprt{ \Grad \tant  \cdot \Grad\varphi_1 } \Grad\frac{\tdnt}{\tant} 
			+ \tant  \Grad^2 \frac{\tdnt }{\tant }  \Grad\varphi_1 
			+ \sum_{j=1,2} \barrho{j}^{-1}\f_j } \cdot \Grad\xi \dx \\
			&\quad + \int_\bbT \biggprt{ \tant  \Bigprt{ \Grad\varphi_1 \otimes \Grad\frac{\tdnt}{\tant} }
			- \sum_{j=1,2} \barrho{j}^{-1} \bigprt{2\tnu{j}\D\vel_j + \tlambda{j} \Div\vel_j\Id} } : \Grad^2\xi \dx
		\end{align*}
		for all $\xi\in\rmH^2(\bbT)$, and the individual terms of $\norm{G}_{\rmH^{-2}\ofbbT}$ are estimated as follows. For the first line, we obtain
		\begin{align*}
			\Bigabs{ \int_\bbT \bigprt{ \Grad \tant  \cdot \Grad\varphi_1 } \Grad\frac{\tdnt}{\tant} \cdot \Grad\xi \dx } 
			&\leq \norm{ \Grad \tant  }_{\rmL^6\ofbbT} \norm{ \Grad\varphi_1 }_{\rmL^6\ofbbT} \Bignorm{ \Grad\frac{\tdnt}{\tant} }_{\rmL^6\ofbbT} \norm{ \Grad\xi }_{\rmL^2\ofbbT} \\
			&\leq C \norm{ \tant  }_{\rmH^2\ofbbT} \norm{ \varphi_1 }_{\rmH^2\ofbbT} \Bignorm{ \frac{\tdnt }{\tant } }_{\rmH^2\ofbbT} \norm{ \xi }_{\rmH^2\ofbbT} 
			\leq \hat C \norm{ \varphi_1 }_{\rmH^2\ofbbT} \norm{ \xi }_{\rmH^2\ofbbT},
		\end{align*}
		with a constant $\hat C=\hat C(\norm{\tphi1}_{\rmH^2\ofbbT})$, as well as
		\begin{align*}
			\Bigabs{ \int_\bbT \Bigprt{ \tant  \Grad^2 \frac{\tdnt }{\tant }  \Grad\varphi_1 } \cdot \Grad\xi \dx } 
			&\leq \norm{ \Grad \tant  }_{\rmL^6\ofbbT} \Bignorm{ \Grad^2 \frac{\tdnt }{\tant }  }_{\rmL^2\ofbbT} \norm{ \Grad\varphi_1 }_{\rmL^6\ofbbT} \norm{ \Grad\xi }_{\rmL^6\ofbbT} \\
			&\leq C \norm{ \tant  }_{\rmH^2\ofbbT}  \Bignorm{ \frac{\tdnt }{\tant } }_{\rmH^2\ofbbT} \norm{ \varphi_1 }_{\rmH^2\ofbbT} \norm{ \xi }_{\rmH^2\ofbbT} 
			\leq \hat C \norm{ \varphi_1 }_{\rmH^2\ofbbT} \norm{ \xi }_{\rmH^2\ofbbT} 
		\end{align*}
		and, for $j=1,2$, 
		\begin{align*}
			\Bigabs{ \int_\bbT \barrho{j}^{-1}\f_j \cdot \Grad\xi \dx }
			\leq  C \norm{ \f_j }_{\rmL^2\ofbbT} \norm{ \Grad\xi }_{\rmL^2\ofbbT} 
			\leq C \norm{ \f_j }_{\rmL^2\ofbbT} \norm{ \xi }_{\rmH^2\ofbbT} .
		\end{align*}
		For the second line, we have the estimates
		\begin{align*}
			\Bigabs{ \int_\bbT \tant  \Bigprt{ \Grad\varphi_1 \otimes \Grad\frac{\tdnt}{\tant} } : \Grad^2\xi \dx }
			&\leq  \norm{ \tant }_{\rmL^6\ofbbT} \norm{ \Grad\varphi_1 }_{\rmL^6\ofbbT} \Bignorm{ \Grad\frac{\tdnt}{\tant} }_{\rmL^6\ofbbT} \norm{ \Grad^2\xi }_{\rmL^2\ofbbT} \\
			&\leq  C \norm{ \tant }_{\rmH^2\ofbbT} \norm{ \varphi_1 }_{\rmH^2\ofbbT} \Bignorm{ \frac{\tdnt }{\tant } }_{\rmH^2\ofbbT} \norm{ \xi }_{\rmH^2\ofbbT} 
			\leq \hat C \norm{ \varphi_1 }_{\rmH^2\ofbbT} \norm{ \xi }_{\rmH^2\ofbbT} 
		\end{align*}
		and, for $j=1,2$, 
		\begin{align*}
			& \Bigabs{ \int_\bbT \barrho{j}^{-1} \bigprt{2\tnu{j}\D\vel_j + \tlambda{j} \Div\vel_j\Id} : \Grad^2\xi \dx } \\
			&\leq  C \Bigprt{ \norm{ \tnu{j} }_{\rmL^\infty\ofbbT} \norm{ \D\vel_j }_{\rmL^2\ofbbT} + \norm{ \tlambda{j} }_{\rmL^\infty\ofbbT} \norm{ \Div\vel_j }_{\rmL^2\ofbbT} } \norm{ \Grad^2\xi }_{\rmL^2\ofbbT} \\
			&\leq  C \Bigprt{ \norm{ \tnu{j} }_{\rmH^2\ofbbT} + \norm{ \tlambda{j} }_{\rmH^2\ofbbT} } \norm{ \vel_j }_{\rmH^1\ofbbT} \norm{ \xi }_{\rmH^2\ofbbT} 
			\leq \hat C \norm{ \vel_j }_{\rmH^1\ofbbT} \norm{ \xi }_{\rmH^2\ofbbT}.
		\end{align*}
		Altogether, taking the supremum over all $\xi\in\rmH^2(\bbT)$ with $\norm{\xi}_{\rmH^2\ofbbT} \leq 1$, the modified pressure is controlled by
		\begin{align*}
			\norm{q}_{\rmL^2\ofbbT}
			\leq C \norm{G}_{\rmH^{-2}\ofbbT}
			\leq \hat C \Bigprt{ \norm{ \varphi_1 }_{\rmH^2\ofbbT} + \sum_{j=1,2} \bigprt{ \norm{ \f_j }_{\rmL^2\ofbbT} + \norm{ \vel_j }_{\rmH^1\ofbbT} } },
		\end{align*}
		from where integrating in time and invoking the energy estimate from Lemma~\ref{Lin:lemma:energy_est} lead to
		\begin{align*}
			&\norm{q}_{\rmL^2(0,T;\rmL^2\ofbbT)} 
			\leq \hat C \Bigprt{ \norm{ \varphi_1 }_{\rmL^2(0,T;\rmH^2\ofbbT)} + \sum_{j=1,2} \bigprt{ \norm{ \f_j }_{\rmL^2(0,T;\rmL^2\ofbbT)} + \norm{ \vel_j }_{\rmL^2(0,T;\rmH^1\ofbbT)} } } \\
			&\leq \hat C \Bigprt{ \norm{ \varphi_1 }_{\rmL^2(0,T;\rmH^2\ofbbT)} + \sum_{j=1,2} \norm{ \f_j }_{\rmL^2(0,T;\rmL^2\ofbbT)} + \norm{ g_1 }_{\rmL^2(0,T;\rmH^1\ofbbT)} 
				+ \sum_{j=1,2} \norm{\vel_{j,0}}_{\rmL^2\ofbbTd} + \norm{\varphi_{1,0}}_{\rmH^1\ofbbT} }.
		\end{align*}
		
		Next, similarly to the proof of Lemma~\ref{Lin:lemma:pressure}, we fix some $(\w_1,\w_2)\in\rmH^1(\bbT)^d \times \rmH^1(\bbT)^d$ and decompose $\w_j$ into $\w_j=\hat\w_j + \barrho{j}^{-1}\Grad\xi$ such that $\Div(\tphi1\hat\w_1+\tphi2\hat\w_2)=0$. With $q$ defined above, it consequently holds
		\begin{align}
			\label{Lin:eq:eq_with_q_without_div_cond}
			&\sum_{j=1,2} \int_\bbT \trho{j}\delt\vel_j \cdot \w_j \dx
			- \sum_{j=1,2} \int_\bbT \Div\bigprt{ 2\tnu{j}\D\vel_j + \tlambda{j}\Div\vel_j\Id } \cdot \w_j \dx \nonumber\\
			&\quad- \int_\bbT (  \tphi1 \w_1 -  \tphi2 \w_2) \cdot \Grad\Lap\varphi_1 \dx
			- \sum_{j=1,2} \int_\bbT q \Div( \tphi{j}\w_j) \dx \nonumber \\
			&\quad+ \sum_{j=1,2} \int_\bbT \tphi{j} \Grad \biggprt{ \frac{\tdnt }{\tant } \Lap\varphi_1 } \cdot \w_j \dx
			= \sum_{j=1,2} \int_\bbT \f_j \cdot \w_j \dx
		\end{align}
		a.e.~in $(0,T)$. The choice of $(\w_1,0)$ and $(0,\w_2)$, respectively, reveals that $(\vel_1,\vel_2,\varphi_1,q)$ solves
		\begin{subequations}
			\begin{align}
				\label{Lin:eq:eq_deltv1_with_q}
				&\int_\bbT \trho1\delt\vel_1 \cdot \w_1 \dx
				- \int_\bbT \Div\bigprt{ 2\tnu1\D\vel_1 + \tlambda1\Div\vel_1\Id } \cdot \w_1 \dx
				- \int_\bbT  \tphi1 \w_1 \cdot \Grad\Lap\varphi_1 \dx \nonumber\\
				&\quad- \int_\bbT q \Div( \tphi1\w_1) \dx
				+ \int_\bbT \tphi1 \Grad \biggprt{ \frac{\tdnt }{\tant } \Lap\varphi_1 } \cdot \w_1 \dx
				= \int_\bbT \f_1 \cdot \w_1 \dx
			\end{align}
			a.e.~in $(0,T)$ for all $\w_1\in\rmH^1(\bbT)^d$, and
			\begin{align}
				\label{Lin:eq:eq_deltv2_with_q}
				&\int_\bbT \trho2\delt\vel_2 \cdot \w_2 \dx
				- \int_\bbT \Div\bigprt{ 2\tnu2\D\vel_2 + \tlambda2\Div\vel_2\Id } \cdot \w_2 \dx
				+ \int_\bbT  \tphi2 \w_2 \cdot \Grad\Lap\varphi_1 \dx \nonumber \\
				&\quad- \int_\bbT q \Div( \tphi2\w_2) \dx
				+ \int_\bbT \tphi2 \Grad \biggprt{ \frac{\tdnt }{\tant } \Lap\varphi_1 } \cdot \w_2 \dx
				= \int_\bbT \f_2 \cdot \w_2 \dx
			\end{align}
		\end{subequations}
		a.e.~in $(0,T)$ for all $\w_2\in\rmH^1(\bbT)^d$, as well as \eqref{Lin:eq:reg_system_eq_deltphi1} and \eqref{Lin:eq:reg_system_eq_div} a.e.~in $(0,T)\times\bbT$, cf.~Remark~\ref{Lin:rmk:ptwise}.
	\end{proof}

    As a direct consequence of the above estimate for the modified pressure, we can also control $\delt\vel_j$ in dependence on $\varphi_1$.
	
	\begin{lemma}[Estimate for $\delt\vel_j$ depending on $\varphi_1$] 
		\label{Lin:lemma:est_delt_velj_dep_phi1_gen}
        Let $0<T<\infty$ and $\varepsilon>0$, and let~$\tphi1\in\rmH^2(\bbT)$ satisfy $\tphi1\in(0,1)$ in $\bbT$. Suppose that
        \begin{equation*}
            \begin{alignedat}{4}
                \f_j &\in\rmL^2(0,T;\rmL^2(\bbT)^d), &\quad j&=1,2, &\qquad
                g_1&\in\rmL^2(0,T;\rmH^1(\bbT)) \cap \rmH^1(0,T;\rmH^{-1}(\bbT)), \\
                \vel_{j,0}&\in \rmH^2(\bbT)^d, &\quad j&=1,2, &\qquad
                \varphi_{1,0}&\in\rmH^3(\bbT), 
            \end{alignedat}
        \end{equation*}
        and that the assumptions of Lemma~\ref{Lin:lemma:existence_abstr}~\ref{Lin:item:existence_abstr_iii} hold. Then the weak solution~$(\vel_1,\vel_2,\varphi_1,q)$ to the regularized system~\eqref{Lin:eq:reg_system} from Lemma~\ref{Lin:lemma:pressure_q_gen} satisfies the uniform \textit{a priori} estimate
		\begin{align}
			\label{Lin:eq:est_deltvj_dep_phi1_gen_coeffs}
			\sum_{j=1,2} \norm{\delt\vel_j}_{\rmL^2(0,T;\rmH^{-1}\ofbbT)} 
			&\leq \hat C \Bigprt{ \norm{ \varphi_1 }_{\rmL^2(0,T;\rmH^2\ofbbT)} + \sum_{j=1,2} \norm{\f_j}_{\rmL^2(0,T;\rmL^2\ofbbTd)} + \norm{g_1}_{\rmL^2(0,T;\rmH^1\ofbbT)} \nonumber\\
			&\qquad+ \sum_{j=1,2} \norm{\vel_{j,0}}_{\rmL^2\ofbbTd} + \norm{\varphi_{1,0}}_{\rmH^1\ofbbT} }
		\end{align}
		with a constant $\hat C>0$ depending on $\norm{\varphi_{1,0}}_{\rmH^2\ofbbT}$ and $T$.
	\end{lemma}
	
	\begin{proof}
		As shown in Lemma~\ref{Lin:lemma:pressure_q_gen}, $(\vel_1,\vel_2,\varphi_1,q)$ solves the weak equation~\eqref{Lin:eq:eq_deltv1_with_q}, which gives, after a suitable integration by parts,
		\begin{align*}
			&\bigabs{ \ang{\trho1\delt\vel_1}{\w_1}_{\rmH^{-1}\ofbbT,\rmH^1\ofbbT} } 
			\leq \norm{ 2\tnu1\D\vel_1 + \tlambda1\Div\vel_1\Id }_{\rmL^2\ofbbT} \norm{\Grad\w_1}_{\rmL^2\ofbbT} 
			+ \norm{\f_1}_{\rmL^2\ofbbT} \norm{\w_1}_{\rmL^2\ofbbT} \\
			&\qquad+  \Bigprt{ \norm{\Lap\varphi_1}_{\rmL^2\ofbbT} + \norm{q}_{\rmL^2\ofbbT} + \Bignorm{ \frac{\tdnt }{\tant } \Lap\varphi_1}_{\rmL^2\ofbbT} }  \norm{\Div(\tphi1\w_1)}_{\rmL^2\ofbbT}  \\
			&\leq C \Bigprt{ 
			\bigprt{ \norm{ \tnu1 }_{\rmL^\infty\ofbbT} + \norm{\tlambda1}_{\rmL^\infty\ofbbT} } \norm{\vel_1}_{\rmH^1\ofbbT} \norm{\w_1}_{\rmH^1\ofbbT} 
			+ \norm{\f_1}_{\rmL^2\ofbbT} \norm{\w_1}_{\rmL^2\ofbbT} \\
			&\qquad+ \Bigprt{ \norm{\Lap\varphi_1}_{\rmL^2\ofbbT} + \norm{q}_{\rmL^2\ofbbT} + \Bignorm{ \frac{\tdnt }{\tant } }_{\rmL^\infty\ofbbT} \norm{\Lap\varphi_1}_{\rmL^2\ofbbT} } 
			\bigprt{ \norm{\Grad\tphi1}_{\rmL^4\ofbbT} \norm{\w_1}_{\rmL^4\ofbbT} + \norm{\tphi1}_{\rmL^\infty\ofbbT} \norm{\Grad\w_1}_{\rmL^2\ofbbT} } }
		\end{align*}
		and therefore
		\begin{align*}
			\norm{ \trho1\delt\vel_1 }_{\rmH^{-1}\ofbbT} 
			\leq \hat C \bigprt{ \norm{\vel_1}_{\rmH^1\ofbbT} + \norm{\Lap\varphi_1}_{\rmL^2\ofbbT}  + \norm{q}_{\rmL^2\ofbbT} +  \norm{\f_1}_{\rmL^2\ofbbT} },
		\end{align*}
		with $\hat C= \hat C(\norm{\tphi1}_{\rmH^2\ofbbT})$.
		Eventually, integrating with respect to time and using Lemma~\ref{Lin:lemma:energy_est} to estimate the $\vel_1$-term and Lemma~\ref{Lin:lemma:pressure_q_gen} for the $q$-term proves the claimed estimate. From~\eqref{Lin:eq:eq_deltv2_with_q}, we conclude the corresponding result for $\delt\vel_2$.
	\end{proof}

    Our next goal is to derive an estimate for $\Lap\varphi_1$ in $\rmL^2(\bbT)$, again by exploiting the damped plate equation satisfied by $\varphi_1$. In contrast to the frozen-coefficient case, additional lower order terms have to be controlled.
	
	\begin{lemma}[Estimate for $\Lap\varphi_1$] 
		\label{Lin:lemma:est_Lap_phi1_gen}
        Let $0<T<\infty$ and $\varepsilon>0$, and let $\tphi1\in\rmH^2(\bbT)$ satisfy~$\tphi1\in(0,1)$ in $\bbT$. Suppose that
        \begin{equation*}
            \begin{alignedat}{4}
                \f_j &\in\rmL^2(0,T;\rmL^2(\bbT)^d), &\quad j&=1,2, &\qquad
                g_1&\in\rmL^2(0,T;\rmH^1(\bbT)) \cap \rmH^1(0,T;\rmH^{-1}(\bbT)), \\
                \vel_{j,0}&\in \rmH^2(\bbT)^d, &\quad j&=1,2, &\qquad
                \varphi_{1,0}&\in\rmH^3(\bbT), 
            \end{alignedat}
        \end{equation*}
        and that the assumptions of Lemma~\ref{Lin:lemma:existence_abstr}~\ref{Lin:item:existence_abstr_iii} hold. Then the weak solution~$(\vel_1,\vel_2,\varphi_1)$ to the regularized system~\eqref{Lin:eq:reg_system} in the sense of Definition~\ref{Lin:def:abstract_eq_weak} satisfies the uniform \textit{a priori} estimate
		\begin{align}
			\label{Lin:eq:a_priori_est_Lap_phi1_gen_coeffs}
			&\norm{\delt\varphi_1}_{\rmL^\infty(0,T;\rmH^{-1}\ofbbT)}
			+ \norm{\Grad\varphi_1}_{\rmL^\infty(0,T;\rmL^2\ofbbTd)}
			+ \norm{\Lap\varphi_1}_{\rmL^2(0,T;\rmL^2\ofbbT)}
			+ \sqrt\varepsilon \norm{\delt\varphi_1}_{\rmL^2(0,T;\rmL^2\ofbbT)} \nonumber \\
			&\leq \hat C \Bigprt{ \sum_{j=1,2} \norm{\f_j}_{\rmL^2(0,T;\rmL^2\ofbbTd)} + \norm{g_1}_{\rmL^2(0,T;\rmH^1\ofbbT)} + \norm{g_1}_{\rmH^1(0,T;\rmH^{-1}\ofbbT)} \nonumber \\
			&\qquad+ \sum_{j=1,2} \norm{\vel_{j,0}}_{\rmL^2\ofbbTd} + \norm{\varphi_{1,0}}_{\rmH^1\ofbbT} + \norm{\delt\varphi_1\vert_{t=0}}_{\rmH^{-1}\ofbbT} + \norm{\Lap\varphi_1\vert_{t=0}}_{\rmH^{-1}\ofbbT} }
		\end{align}
		with a constant $\hat C>0$ depending on $\norm{\tphi1}_{\rmH^2\ofbbT}$ and $T$.
	\end{lemma}
	
	\begin{proof}
		The idea is to redo the proof of Lemma~\ref{Lin:lemma:est_grad_Lap_phi1} at a level one order lower than before, once again making use of the structure of a damped plate equation. The corresponding test function for the weak formulation \eqref{Lin:eq:weak_rewritten} is thus chosen as~$\bigprt{\tphi1^{-1} \Grad\Phi, -\tphi2^{-1} \Grad\Phi, 0}$, where later~$\Phi = \tant^{-1} (-\Lap)^{-1}(\delt\varphi_1 - \Lap\varphi_1)$. To ensure that $\mean{\delt\varphi_1}=0$, one may reduce to the case~$\mean{g_1}=0$, using \eqref{Lin:eq:reg_system_eq_deltphi1}, as in the proof of Proposition~\ref{Lin:prop:ex_const_coeff}. As before, the derivation of a damped plate equation relies crucially on the identities
		\begin{align*}
			\Div \vel_1 &= \tphi1^{-1} \prt{-\delt\varphi_1 + \varepsilon\Lap\varphi_1 + g_1 - \Grad\tphi1\cdot\vel_1}, \\
			\Div \vel_2 &= -\tphi2^{-1} \prt{-\delt\varphi_1 + \varepsilon\Lap\varphi_1 + g_1 + \Grad\tphi2\cdot\vel_2},
		\end{align*}
		valid a.e.~in $(0,T)\times\bbT$, which are inferred by the equations~\eqref{Lin:eq:reg_system_eq_deltphi1} and \eqref{Lin:eq:reg_system_eq_div}. In addition, we use the abbreviations
		\begin{equation*}
			\arraycolsep=2pt
			\begin{array}{rlrl}
				\ta{j} &= \barrho{j}\tphi{j}^{-1}, &
				\tb{j} &= (2\tnu{j}+\tlambda{j}) \tphi{j}^{-2},
				\quad j=1,2, \\[1ex]
				\tant&= \ta1 + \ta2, \quad &
				\tbnt &= \tb1 + \tb2.
			\end{array}
		\end{equation*}
		
		For now, let the test function $\Phi\in\rmC^\infty(\bbT)$ be smooth. Analyzing the individual parts of the equation tested as described above, one by one, we first obtain
		\begin{align*}
			&\int_\bbT \trho1 \delt \vel_1 \cdot \tphi1^{-1} \Grad\Phi \dx
			- \int_\bbT \trho2 \delt \vel_2 \cdot \tphi2^{-1} \Grad\Phi \dx \\
			&= - \int_\bbT \barrho1 \delt \Div\vel_1 \Phi \dx
			+ \int_\bbT \barrho2 \delt \Div\vel_2 \Phi \dx \\
			&= \bigang{ \tant \delt^2\varphi_1 
				- \varepsilon \tant \delt\Lap\varphi_1 
				- \tant\delt g_1}{ \Phi }_{\rmH^{-1}\ofbbT,\rmH^1\ofbbT} + L_1,
		\end{align*}
		where $L_1$ is a lower order term given by
		\begin{align*}
			L_1 = \int_\bbT \prt{\ta1\Grad\tphi1 \cdot \delt\vel_1 - \ta2\Grad\tphi2 \cdot \delt\vel_2} \Phi \dx.
		\end{align*}
		To study the second part, for $j=1,2$, we start by computing
		\begin{align*}
			&-\int_\bbT \tnu{j}\Lap\vel_j \cdot \tphi{j}^{-1} \Grad\Phi \dx 
			= \int_\bbT \tnu{j}\tphi{j}^{-1}\Grad\vel_j : \Grad^2\Phi \dx 
			- \int_\bbT \bigprt{ \Grad(\tnu{j}\tphi{j}^{-1}) \cdot \nabla } \vel_j \cdot \Grad\Phi \dx \\
			&= - \int_\bbT \tnu{j}\tphi{j}^{-1} \Grad\Div\vel_j \cdot \Grad\Phi \dx
			- 2\int_\bbT \bigprt{ \Grad(\tnu{j}\tphi{j}^{-1}) \cdot \nabla } \vel_j \cdot \Grad\Phi \dx
		\end{align*}
		and find
		\begin{align*}
			&- \int_\bbT \Div\bigprt{ 2\tnu{j}\D\vel_j + \tlambda{j}\Div\vel_j\Id } \cdot \tphi{j}^{-1} \Grad\Phi \dx \\
			&= - \int_\bbT \Bigprt{ \bigprt{\tnu{j}\Lap\vel_j + (\tnu{j}+ \tlambda{j})\Grad\Div\vel_j} + \bigprt{2\D\vel_j \Grad\tnu{j} + \Div\vel_j\Grad\tlambda{j}} }\cdot \tphi{j}^{-1} \Grad\Phi \dx \\
			&= -\int_\bbT (2\tnu{j}+\tlambda{j}) \tphi{j}^{-1} \Grad\Div\vel_j \cdot \Grad\Phi\dx
			- 2\int_\bbT \bigprt{ \Grad(\tnu{j}\tphi{j}^{-1}) \cdot \nabla } \vel_j \cdot \Grad\Phi \dx \\
			&\quad- \int_\bbT \tphi{j}^{-1} \bigprt{2\D\vel_j \Grad\tnu{j} + \Div\vel_j\Grad\tlambda{j}} \cdot \Grad\Phi \dx  \\
			&= \bigang{ (2\tnu{j}+\tlambda{j}) \tphi{j}^{-1} \Lap\Div\vel_j }{ \Phi }_{\rmH^{-1}\ofbbT,\rmH^1\ofbbT} 
			+ L_2^j,
		\end{align*}
		where $L_2^j = \sum_{k=1}^3 L_{2,k}^j$ with
		\begin{align*}
			L_{2,1}^j &= \int_\bbT \Grad\bigprt{(2\tnu{j}+\tlambda{j}) \tphi{j}^{-1} } \cdot \Grad\Div\vel_j  \Phi \dx, \\
			L_{2,2}^j &= - 2\int_\bbT \bigprt{ \Grad(\tnu{j}\tphi{j}^{-1}) \cdot \nabla } \vel_j \cdot \Grad\Phi \dx, \\
			L_{2,3}^j &= - \int_\bbT \tphi{j}^{-1} \bigprt{2\D\vel_j \Grad\tnu{j} + \Div\vel_j\Grad\tlambda{j}} \cdot \Grad\Phi \dx.
		\end{align*}
		Thus, we get for the second part 
		\begin{align*}
			&-\int_\bbT \Div\bigprt{2 \tnu1\D\vel_1 + \tlambda1\Div\vel_1\Id } \cdot \tphi1^{-1} \Grad\Phi \dx \\
			&\quad+ \int_\bbT \Div\bigprt{ 2\tnu2\D\vel_2 + \tlambda2\Div\vel_2\Id } \cdot \tphi2^{-1} \Grad\Phi \dx \\
			&= \bigang{ (2\tnu1+\tlambda1) \tphi1^{-1} \Lap\Div\vel_1 }{ \Phi }_{\rmH^{-1}\ofbbT,\rmH^1\ofbbT} + L_2^1 
			 - \bigang{ (2\tnu2+\tlambda2) \tphi2^{-1} \Lap\Div\vel_2 }{ \Phi }_{\rmH^{-1}\ofbbT,\rmH^1\ofbbT} - L_2^2 \\
			&= \bigang{ -\tbnt \delt\Lap\varphi_1 + \varepsilon \tbnt \Lap^2\varphi_1 +\tbnt \Lap g_1 }{ \Phi }_{\rmH^{-1}\ofbbT,\rmH^1\ofbbT} 
			+ L_2 + L_3,
		\end{align*}
		with $L_2= L_2^1 - L_2^2$ and $L_3= \sum_{j=1,2} (L_{3,1}^j + L_{3,2}^j) + L_{3,3}$, where 
        \begin{align*}
            L_{3,1}^j &= \int_\bbT 2(2\tnu{j}+\tlambda{j}) \tphi{j}^{-1} \Grad(\tphi{j}^{-1}) \cdot \Grad \prt{-\delt\varphi_1 + \varepsilon\Lap\varphi_1 + g_1 \mp \Grad\tphi{j}\cdot\vel_{j} } \Phi \dx, \\
            L_{3,2}^j &= \int_\bbT (2\tnu{j}+\tlambda{j}) \tphi{j}^{-1} \Lap(\tphi{j}^{-1}) \prt{-\delt\varphi_1 + \varepsilon\Lap\varphi_1 + g_1 \mp \Grad\tphi{j}\cdot\vel_{j} } \Phi \dx, \\
            L_{3,3} &= - \int_\bbT \bigprt{ \tb1 \Lap(\Grad\tphi1 \cdot \vel_1) - \tb2 \Lap(\Grad\tphi2 \cdot \vel_2) } \Phi \dx.
        \end{align*}
		Moreover, the third part simply reads
		\begin{align*}
			&- \int_\bbT \bigprt{  \tphi1 \tphi1^{-1} \Grad\Phi +  \tphi2 \tphi2^{-1} \Grad\Phi } \cdot \Grad\Lap\varphi_1 \dx
			= \bigang{ 2 \Lap^2\varphi_1 }{\Phi }_{\rmH^{-1}\ofbbT,\rmH^1\ofbbT}
		\end{align*}
		and we write the right-hand side of the tested equation~\eqref{Lin:eq:weak_rewritten} as
		\begin{align*}
			\int_\bbT \f_1 \cdot \tphi1^{-1} \Grad\Phi \dx
			- \int_\bbT \f_2 \cdot \tphi2^{-1} \Grad\Phi \dx
			= \bigang{ -\tphi1^{-1} \Div\f_1 + \tphi2^{-1} \Div\f_2 }{\Phi }_{\rmH^{-1}\ofbbT,\rmH^1\ofbbT}
			- L_4,
		\end{align*}
		where
		\begin{align*}
			L_4
			= \bigang{ -\Grad(\tphi1^{-1}) \cdot \f_1 + \Grad(\tphi2^{-1}) \cdot \f_2 }{\Phi }_{\rmH^{-1}\ofbbT,\rmH^1\ofbbT}.
		\end{align*}
		Altogether, this leads to the damped plate equation 
		\begin{align*}
			&\bigang{ \tant \delt^2\varphi_1 
				- (\tbnt +\varepsilon \tant) \delt\Lap\varphi_1 
				+ (2 +\varepsilon \tbnt) \Lap^2\varphi_1}{ \Phi }_{\rmH^{-1}\ofbbT,\rmH^1\ofbbT}  \\
			&= \bigang{ \tant \delt g_1 - \tbnt \Lap g_1 - \tphi1^{-1}\Div\f_1 + \tphi2^{-1}\Div\f_2 }{\Phi}_{\rmH^{-1}\ofbbT,\rmH^1\ofbbT}
			- \sum_{i=1}^4 L_i
		\end{align*}
		with some additional lower order terms on the right-hand side.
		Therein, we move the expression~$-\ang{ \tbnt \delt\Lap\varphi_1}{\Phi}_{\rmH^{-1}\ofbbT,\rmH^1\ofbbT}$ to the right-hand side as well, followed by the addition of a term~$-\ang{ \tant\delt\Lap\varphi_1}{\Phi}_{\rmH^{-1}\ofbbT,\rmH^1\ofbbT}$ on both sides. As a result, we end up with the equivalent equation 
		\begin{align}
			\label{Lin:eq:est_Lap_phi1_tested_eq}
			&\bigang{ \tant \delt(\delt\varphi_1 - \Lap\varphi_1) - \varepsilon \tant \delt\Lap\varphi_1 + (2+\varepsilon b)\Lap^2\varphi_1}{\Phi}_{\rmH^{-1}\ofbbT,\rmH^1\ofbbT} \nonumber \\
			&= \bigang{ (\tbnt-\tant) \delt\Lap\varphi_1 + \tant \delt g_1 - \tbnt\Lap g_1 - \tphi1^{-1}\Div\f_1 + \tphi2^{-1}\Div\f_2 }{\Phi}_{\rmH^{-1}\ofbbT,\rmH^1\ofbbT}
			- \sum_{i=1}^4 L_i \nonumber \\
            &\eqqcolon \sum_{i=1}^4 K_i - \sum_{i=1}^4 L_i,
		\end{align}
		where we denote the terms in the dual pairing on the right-hand side by $K_i$, $i\in\{1,\dots,4\}$, combining two the terms that include $\f_1$ and $\f_2$ into $K_4$. 
        
        In light of the density of $\rmC^\infty(\bbT)$ in $\rmH^1(\bbT)$, it is now possible to insert into \eqref{Lin:eq:est_Lap_phi1_tested_eq} our determined test function~$\Phi = \tant^{-1}(-\Lap)^{-1}(\delt\varphi_1 - \Lap\varphi_1)$, starting with the left-hand side terms
		\begin{align*}
			J_1
			&\coloneqq \bigang{ \tant \delt (\delt\varphi_1-\Lap\varphi_1)}{ \tant^{-1}(-\Lap)^{-1}(\delt\varphi_1 - \Lap\varphi_1) }_{\rmH^{-1}\ofbbT,\rmH^1\ofbbT} 
			= \frac12 \ddt \norm{ \delt\varphi_1-\Lap\varphi_1 }_{\rmH^{-1}\ofbbT}^2
		\end{align*}
		and
		\begin{align*}
			J_2
			&\coloneqq - \bigang{ \varepsilon \tant \delt\Lap\varphi_1 }{ \tant^{-1}(-\Lap)^{-1}(\delt\varphi_1 - \Lap\varphi_1) }_{\rmH^{-1}\ofbbT,\rmH^1\ofbbT} \\
			&= \varepsilon \int_\bbT \prt{\delt\varphi_1}^2 \dx
			+ \varepsilon \int_\bbT \delt\Grad\varphi_1 \cdot \Grad\varphi_1 \dx 
			= \varepsilon \norm{ \delt\varphi_1 }_{\rmL^2\ofbbT}^2 
			+ \varepsilon \frac12 \ddt \norm{ \Grad\varphi_1 }_{\rmL^2\ofbbT}^2. 
		\end{align*}
		Continuing with the third and the fourth summand, we extract the highest order terms to obtain
		\begin{align*}
			J_3 
			&\coloneqq \bigang{ 2 \Lap^2\varphi_1 }{ \tant^{-1}(-\Lap)^{-1}(\delt\varphi_1 - \Lap\varphi_1) }_{\rmH^{-1}\ofbbT\!,\rmH^1\ofbbT\!} 
			= \ddt \bignorm{ \tant^{-\frac12} \Grad\varphi_1}_{\rmL^2\ofbbT\!}^2
			+  2\bignorm{ \tant^{-\frac12} \Lap\varphi_1}_{\rmL^2\ofbbT\!}^2
			+ \sum_{k=1}^5 \! J_{3,k}
		\end{align*}
		and
		\begin{align*}
			J_4 
			&\coloneqq \varepsilon \bigang{ \tbnt \Lap^2\varphi_1 }{ \tant^{-1}(-\Lap)^{-1}(\delt\varphi_1 - \Lap\varphi_1) }_{\rmH^{-1}\ofbbT,\rmH^1\ofbbT} \\
			&= \varepsilon \frac12 \ddt \bignorm{ \bigprt{\tbnt\tant^{-1}}^{\frac12} \Grad\varphi_1 }_{\rmL^2\ofbbT}^2 + \varepsilon \bignorm{ \bigprt{\tbnt\tant^{-1}}^{\frac12} \Lap\varphi_1 }_{\rmL^2\ofbbT}^2 + \sum_{k=1}^5 J_{4,k},
		\end{align*}
        where the respective lower order terms are of similar structure, namely
        \begin{equation*}
            \arraycolsep=0.4pt
    		\begin{array}{rlrl}
    			J_{3,1} &= -2 \bigscp{ \Grad(\tant^{-1}) \cdot \Grad\varphi_1 }{ \delt\varphi_1 }_{\rmL^2\ofbbT}, 
                &J_{4,1} &= -\varepsilon \bigscp{ \Grad\bigprt{\tbnt\tant^{-1}} \cdot \Grad\varphi_1 }{ \delt\varphi_1 }_{\rmL^2\ofbbT}, \\[1ex]
    			J_{3,2} &= 4 \bigang{ \Grad(\tant^{-1}) \Lap\varphi_1 }{ (-\Lap)^{-1}\Grad\delt\varphi_1 }_{\rmH^{-1}\ofbbT,\rmH^1\ofbbT},\; 
                &J_{4,2} &= 2\varepsilon \bigang{ \Grad\bigprt{\tbnt\tant^{-1}} \Lap\varphi_1 }{ (-\Lap)^{-1}\Grad\delt\varphi_1 }_{\rmH^{-1}\ofbbT,\rmH^1\ofbbT}, \\[1ex]
    			J_{3,3} &= 2 \bigang{ \Lap(\tant^{-1}) \Lap\varphi_1 }{ (-\Lap)^{-1}\delt\varphi_1 }_{\rmH^{-1}\ofbbT,\rmH^1\ofbbT}, 
                &J_{4,3} &= \varepsilon \bigang{ \Lap\bigprt{\tbnt\tant^{-1}} \Lap\varphi_1 }{ (-\Lap)^{-1}\delt\varphi_1 }_{\rmH^{-1}\ofbbT,\rmH^1\ofbbT}, \\[1ex]
    			J_{3,4} &= 4 \bigscp{ \Grad(\tant^{-1}) \Lap\varphi_1 }{ \Grad\varphi_1 }_{\rmL^2\ofbbT}, 
                &J_{4,4} &= 2\varepsilon \bigscp{ \Grad\bigprt{\tbnt\tant^{-1}} \Lap\varphi_1 }{ \Grad\varphi_1 }_{\rmL^2\ofbbT}, \\[1ex]
    			J_{3,5} &= 2 \bigscp{ \Lap(\tant^{-1}) \Lap\varphi_1 }{ \varphi_1 }_{\rmL^2\ofbbT},
                &J_{4,5} &= \varepsilon \bigscp{ \Lap\bigprt{\tbnt\tant^{-1}} \Lap\varphi_1 }{ \varphi_1 }_{\rmL^2\ofbbT}.
    		\end{array}
        \end{equation*}
		
		After having collected the useful terms, it remains to estimate the right-hand side of~\eqref{Lin:eq:est_Lap_phi1_tested_eq} as well as the lower order terms $J_{3,k}$ and $J_{4,k}$, $k\in\{1,\dots5\}$, which are moved to the right-hand side later on. 
		
		To this end, we point out that $\tphi1\in\rmH^2(\bbT)$ and therefore, due to Lemma~\ref{prelim:lemma:comp_Sobolev}, all functions depending on~$\tphi1$ (such as $\tphi2$, $\tnu{j}, \tlambda{j}$, $j=1,2$,) belong to $\rmH^2(\bbT)$ as well, while also being uniformly positive. Moreover, we recall the Banach algebra property of~$\rmH^2(\bbT)$. In the following, we denote constants depending on $\norm{\tphi1}_{\rmH^2\ofbbT}$ by $\hat C=\hat C(\norm{\tphi1}_{\rmH^2\ofbbT})$.
		
		Now we start by estimating the terms $J_{3,k}$, $k\in\{1,\dots,5\}$, beginning with
		\begin{align*}
			\abs{J_{3,1}}
			&\leq C \bignorm{\Grad \bigprt{\tant^{-1}}}_{\rmL^4\ofbbT} \norm{\Grad\varphi_1}_{\rmL^4\ofbbT} \norm{\delt\varphi_1}_{\rmL^2\ofbbT} 
			\leq \hat{C} \bigprt{ \delta \norm{\varphi_1}_{\rmH^2\ofbbT}^2 + C_\delta \norm{\delt\varphi_1}_{\rmL^2\ofbbT}^2 },
		\end{align*}
		where Young's inequality allows to absorb terms with a factor $\delta$ in the end.
		Next, we have
		\begin{align*}
			\abs{J_{3,2}}
            &\leq C \bignorm{ \Grad \bigprt{\tant^{-1}} }_{\rmL^4\ofbbT} \norm{\Lap\varphi_1}_{\rmL^2\ofbbT} \bignorm{ (-\Lap)^{-1} \Grad\delt\varphi_1}_{\rmL^4\ofbbT}  \\
            &\leq C \norm{ \tant^{-1} }_{\rmH^2\ofbbT} \norm{\varphi_1}_{\rmH^2\ofbbT} \norm{ \delt\varphi_1}_{\rmL^2\ofbbT} 
			\leq \hat{C} \bigprt{ \delta \norm{\varphi_1}_{\rmH^2\ofbbT}^2 + C_\delta \norm{\delt\varphi_1}_{\rmL^2\ofbbT}^2 }
		\end{align*}
		and, repeating the second step above,
        \begin{align*}
			\abs{J_{3,3}}
            &\leq C \bignorm{ \Lap \bigprt{\tant^{-1}} }_{\rmL^2\ofbbT} \norm{\Lap\varphi_1}_{\rmL^2\ofbbT} \bignorm{ (-\Lap)^{-1} \delt\varphi_1}_{\rmH^2\ofbbT}
			\leq \hat{C} \bigprt{ \delta \norm{\varphi_1}_{\rmH^2\ofbbT}^2 + C_\delta \norm{\delt\varphi_1}_{\rmL^2\ofbbT}^2 }.
		\end{align*}
		In light of an interpolation and Young's inequality, we observe
        \begin{align*}
			\abs{J_{3,4}}
			&\leq C \bignorm{\Grad \bigprt{\tant^{-1}}}_{\rmL^6\ofbbT} \norm{\Lap\varphi_1}_{\rmL^2\ofbbT} \norm{\Grad\varphi_1}_{\rmL^3\ofbbT} \\
            &\leq C \norm{\tant^{-1}}_{\rmH^2\ofbbT} \norm{\varphi_1}_{\rmH^2\ofbbT} \norm{\varphi_1}_{\rmH^2\ofbbT}^{\frac12} \norm{\varphi_1}_{\rmH^1\ofbbT}^{\frac12} 
			\leq \hat{C} \bigprt{ \delta \norm{\varphi_1}_{\rmH^2\ofbbT}^2 + C_\delta \norm{\varphi_1}_{\rmH^1\ofbbT}^2 }.
		\end{align*}
		Similarly, since Agmon's and Young's inequalities lead to the same expression as in the second line, we obtain
        \begin{align*}
			\abs{J_{3,5}}
			&\leq C \bignorm{\Lap \bigprt{\tant^{-1}}}_{\rmL^2\ofbbT} \norm{\Lap\varphi_1}_{\rmL^2\ofbbT} \norm{\varphi_1}_{\rmL^\infty\ofbbT}
			\leq \hat{C} \bigprt{ \delta \norm{\varphi_1}_{\rmH^2\ofbbT}^2 + C_\delta \norm{\varphi_1}_{\rmH^1\ofbbT}^2 }.
		\end{align*}
		For each $k\in\{1,\dots,5\}$, $J_{4,k}$ is estimated like $J_{3,k}$, where $\varepsilon$ is moved to the respective term with a factor $\delta$.

        Now we proceed with the estimation of the right-hand side terms of equation~\eqref{Lin:eq:est_Lap_phi1_tested_eq}, starting with $K_i$, $i\in\{1,\dots,4\}$. 
        First, after an integration by parts in $K_1$, we find
        \begin{align*}
			\abs{K_1} 
			&\leq C \norm{\delt\varphi_1}_{\rmL^2\ofbbT} \bignorm{ (\tbnt-\tant) \tant^{-1} (-\Lap)^{-1} (\delt\varphi_1-\Lap\varphi_1)}_{\rmH^2\ofbbT} \\
            &\leq C \norm{\delt\varphi_1}_{\rmL^2\ofbbT} \bignorm{ (\tbnt-\tant) \tant^{-1}}_{\rmH^2\ofbbT} \norm{ \delt\varphi_1-\Lap\varphi_1}_{\rmL^2\ofbbT} 
			\leq \hat{C}_2 \bigprt{ C_\delta \norm{\delt\varphi_1}_{\rmL^2\ofbbT}^2 + \delta \norm{\Lap\varphi_1}_{\rmL^2\ofbbT} ^2 },
		\end{align*}
		and for $K_2$, we simply observe
		\begin{align*}
			\abs{K_2} 
			\leq C \norm{ \delt g_1 }_{\rmH^{-1}\ofbbT} \norm{\delt\varphi_1-\Lap\varphi_1}_{\rmH^{-1}\ofbbT} 
			\leq C \bigprt{ \norm{\delt g_1}_{\rmH^{-1}\ofbbT}^2 + \norm{\delt\varphi_1-\Lap\varphi_1}_{\rmH^{-1}\ofbbT} ^2 }.
		\end{align*}
		Analogously to $K_1$, we get
        \begin{align*}
			\abs{K_3} 
			&\leq C \norm{g_1}_{\rmL^2\ofbbT} \bignorm{ \tbnt \tant^{-1} (-\Lap)^{-1} (\delt\varphi_1-\Lap\varphi_1)}_{\rmH^2\ofbbT}
			\leq \hat{C} \bigprt{ C_\delta \norm{g_1}_{\rmL^2\ofbbT}^2 + \norm{\delt\varphi_1}_{\rmL^2\ofbbT}^2 + \delta \norm{\Lap\varphi_1}_{\rmL^2\ofbbT}^2 } .
		\end{align*}
		For both summands $K_4^j$ of $K_4$, $j=1,2$, a similar estimation reveals
        \begin{align*}
			\abs{K_4^j} 
			&\leq C \norm{\f_j}_{\rmL^2\ofbbT} \bignorm{ \tphi{j}^{-1} \tant^{-1} (-\Lap)^{-1} (\delt\varphi_1-\Lap\varphi_1)}_{\rmH^1\ofbbT} \\
			&\leq C \norm{\f_j}_{\rmL^2\ofbbT} \norm{ \tphi{j}^{-1} \tant^{-1} }_{\rmH^2\ofbbT} \norm{ \delt\varphi_1-\Lap\varphi_1}_{\rmH^{-1}\ofbbT}
			\leq \hat{C} \bigprt{ \norm{\f_j}_{\rmL^2\ofbbT}^2 + \norm{\delt\varphi_1-\Lap\varphi_1}_{\rmH^{-1}\ofbbT} ^2 }.
		\end{align*}
	
		Next, we continue with the remaining lower order terms $L_i$, $i\in\{1,\dots,4\}$, on the right-hand side of \eqref{Lin:eq:est_Lap_phi1_tested_eq}. To this end, we study the summands of $L_i$, denoted by $L_i=L_i^1-L_i^2$. We begin by splitting $L_1^j$ into the following two terms, namely
        \begin{align*}
			L_1^j 
            &= \int_\bbT \ta{j}\Grad\tphi{j} \cdot \delt\vel_j \tant^{-1} (-\Lap)^{-1} \delt\varphi_1 \dx
            + \int_\bbT \ta{j}\Grad\tphi{j} \cdot \delt\vel_j \tant^{-1} \varphi_1 \dx 
            \eqqcolon L_{1,1}^j + L_{1,2}^j.
		\end{align*}
        The first part is estimated by
        \begin{align*}
			\abs{L_{1,1}^j} 
            &\leq C \norm{\delt\vel_j}_{\rmH^{-1}\ofbbT} \bignorm{ \ta{j} \Grad\tphi{j} \tant^{-1} (-\Lap)^{-1} \delt\varphi_1 }_{\rmH^1\ofbbT} \\
		    &\leq C \norm{\delt\vel_j}_{\rmH^{-1}\ofbbT} \norm{ \ta{j} \Grad\tphi{j} \tant^{-1} }_{\rmH^1\ofbbT} \bignorm{ (-\Lap)^{-1} \delt\varphi_1 }_{\rmH^2\ofbbT} 
			\leq \hat{C} \bigprt{ \delta \norm{\delt\vel_j}_{\rmH^{-1}\ofbbT}^2 + C_\delta \norm{\delt\varphi_1}_{\rmL^2\ofbbT}^2 }.
		\end{align*}
        For the second part, Agmon's inequality and a suitable interpolation, together with Young's inequality, reveal
        \begin{align*}
			\abs{L_{1,2}^j} 
            &\leq C \norm{\delt\vel_j}_{\rmH^{-1}\ofbbT} \norm{ \ta{j} \Grad\tphi{j} \tant^{-1}\varphi_1 }_{\rmH^1\ofbbT} \\
            &\leq \hat C \norm{\delt\vel_j}_{\rmH^{-1}\ofbbT} 
            \Bigprt{ \norm{ \Grad\tphi{j} \varphi_1 }_{\rmL^2\ofbbT} + \norm{ \Grad^2\tphi{j} \varphi_1 }_{\rmL^2\ofbbT} + \norm{ \Grad\tphi{j} \cdot \Grad\varphi_1 }_{\rmL^2\ofbbT} } \\
            &\leq \hat C \norm{\delt\vel_j}_{\rmH^{-1}\ofbbT} 
            \Bigprt{ \norm{ \Grad\tphi{j} }_{\rmL^4\ofbbT} \norm{ \varphi_1 }_{\rmL^4\ofbbT} 
            + \norm{ \Grad^2\tphi{j} }_{\rmL^2\ofbbT} \norm{ \varphi_1 }_{\rmL^\infty\ofbbT} 
            + \norm{ \Grad\tphi{j} }_{\rmL^6\ofbbT} \norm{ \Grad\varphi_1 }_{\rmL^3\ofbbT} } \\
            &\leq \hat C \norm{\delt\vel_j}_{\rmH^{-1}\ofbbT} 
            \Bigprt{ \norm{ \tphi{j} }_{\rmH^2\ofbbT} \norm{ \varphi_1 }_{\rmH^1\ofbbT} 
            + \norm{ \tphi{j} }_{\rmH^2\ofbbT} \norm{ \varphi_1 }_{\rmH^2\ofbbT}^{\frac12} \norm{ \varphi_1 }_{\rmH^1\ofbbT}^{\frac12} }\\
            &\leq \hat{C} \bigprt{ \delta \norm{\delt\vel_j}_{\rmH^{-1}\ofbbT}^2 + C_\delta \norm{\varphi_1}_{\rmH^1\ofbbT}^2 
            + \delta_1\norm{\delt\vel_j}_{\rmH^{-1}\ofbbT}^2 + \delta_2 \norm{ \varphi_1 }_{\rmH^2\ofbbT}^2 + C_{\delta_1,\delta_2} \norm{ \varphi_1 }_{\rmH^1\ofbbT}^2 }.
        \end{align*}
		Then we consider the summands of $L_2^j$. After an integration by parts in $L_{2,1}^j$, we obtain
        \begin{align*}
			\abs{L_{2,1}^j}
            &\leq C \norm{\Div\vel_j}_{\rmL^2\ofbbT} \bignorm{ \Grad\bigprt{(2\tnu{j}+\tlambda{j}) \tphi{j}^{-1} } \tant^{-1} (-\Lap)^{-1} (\delt\varphi_1-\Lap\varphi_1) }_{\rmH^1\ofbbT} \\
            &\leq C \norm{\vel_j}_{\rmH^1\ofbbT} \bignorm{ \Grad\bigprt{(2\tnu{j}+\tlambda{j}) \tphi{j}^{-1} } \tant^{-1} }_{\rmH^1\ofbbT} \bignorm{ (-\Lap)^{-1} (\delt\varphi_1-\Lap\varphi_1) }_{\rmH^2\ofbbT}\\
            &\leq \hat C \norm{\vel_j}_{\rmH^1\ofbbT} \norm{ \delt\varphi_1-\Lap\varphi_1 }_{\rmL^2\ofbbT}
			\leq \hat{C} \bigprt{ C_\delta \norm{\vel_j}_{\rmH^1\ofbbT}^2 + \norm{\delt\varphi_1}_{\rmL^2\ofbbT}^2 + \delta \norm{\Lap\varphi_1}_{\rmL^2\ofbbT}^2 }.
		\end{align*}
		Furthermore, we estimate
        \begin{align*}
			\abs{L_{2,2}^j} 
			&\leq C \bignorm{ \Grad(\tnu{j}\tphi{j}^{-1}) }_{\rmL^6\ofbbT} \norm{ \nabla \vel_j  }_{\rmL^2\ofbbT} \bignorm{\Grad \bigprt{ \tant^{-1} (-\Lap)^{-1} (\delt\varphi_1-\Lap\varphi_1) }}_{\rmL^3\ofbbT} \\
            &\leq C \norm{ \tnu{j}\tphi{j}^{-1} }_{\rmH^2\ofbbT} \norm{ \vel_j }_{\rmH^1\ofbbT} \norm{ \tant^{-1} }_{\rmH^2\ofbbT} \bignorm{ (-\Lap)^{-1} (\delt\varphi_1-\Lap\varphi_1) }_{\rmH^2\ofbbT} \\
            &\leq \hat C \norm{\vel_j}_{\rmH^1\ofbbT} \norm{ \delt\varphi_1-\Lap\varphi_1 }_{\rmL^2\ofbbT}
			\leq \hat{C} \bigprt{ C_\delta \norm{\vel_j}_{\rmH^1\ofbbT}^2 + \norm{\delt\varphi_1}_{\rmL^2\ofbbT}^2 + \delta \norm{\Lap\varphi_1}_{\rmL^2\ofbbT}^2 },
		\end{align*}
		and $L_{2,3}^j$ in a similar way. Studying the terms of $L_3$, we find, after an integration by parts,
        \begin{align*}
            \abs{L_{3,1}^j} 
            &\leq C \norm{ -\delt\varphi_1 + \varepsilon\Lap\varphi_1 + g_1 \mp \Grad\tphi{j}\cdot\vel_{j} }_{\rmL^2\ofbbT} 
            \bignorm{  (2\tnu{j}+\tlambda{j}) \tphi{j}^{-1} \Grad(\tphi{j}^{-1}) \Phi }_{\rmH^1\ofbbT} \\
            &\leq \hat C \norm{ -\delt\varphi_1 + \varepsilon\Lap\varphi_1 + g_1 \mp \Grad\tphi{j}\cdot\vel_{j} }_{\rmL^2\ofbbT} 
            \bignorm{ \Grad(\tphi{j}^{-1}) }_{\rmH^1\ofbbT} 
            \norm{ (-\Lap)^{-1} (\delt\varphi_1-\Lap\varphi_1) }_{\rmH^2\ofbbT}\\
            &\leq \hat C \norm{ -\delt\varphi_1 + \varepsilon\Lap\varphi_1 + g_1 \mp \Grad\tphi{j}\cdot\vel_{j} }_{\rmL^2\ofbbT} 
            \norm{ \delt\varphi_1-\Lap\varphi_1 }_{\rmL^2\ofbbT}\\
            &\leq \hat{C} \Bigprt{ C_\delta \bigprt{\norm{\delt\varphi_1}_{\rmL^2\ofbbT}^2 + \varepsilon\norm{\Lap\varphi_1}_{\rmL^2\ofbbT}^2 + \norm{g_1}_{\rmL^2\ofbbT}^2 + \norm{\vel_j}_{\rmH^1\ofbbT}^2
            }
            + \norm{\delt\varphi_1}_{\rmL^2\ofbbT}^2 + \delta \norm{\Lap\varphi_1}_{\rmL^2\ofbbT}^2 }.
        \end{align*}
        Next, $L_{3,2}^j$ is estimated by
        \begin{align*}
            \abs{L_{3,2}^j} 
            &\leq C \norm{ -\delt\varphi_1 + \varepsilon\Lap\varphi_1 + g_1 \mp \Grad\tphi{j}\cdot\vel_{j} }_{\rmL^2\ofbbT} 
            \norm{ (2\tnu{j}+\tlambda{j}) \tphi{j}^{-1} \tant^{-1} }_{\rmH^2\ofbbT} 
            \bignorm{ \Lap(\tphi{j}^{-1}) }_{\rmL^2\ofbbT} \\
            &\quad\;\, \cdot \norm{ (-\Lap)^{-1} (\delt\varphi_1-\Lap\varphi_1) }_{\rmH^2\ofbbT},
        \end{align*}
        which can be further controlled as in the last two lines of the estimation of $L_{3,1}^j$.
        Again, integrating by parts suitably, we have
        \begin{align*}
			\abs{L_{3,3}^j} 
            &\leq C \norm{\Grad\tphi{j}\cdot\vel_j}_{\rmL^2\ofbbT} \bignorm{ \Lap \bigprt{ \tb{j} \tant^{-1} (-\Lap)^{-1} (\delt\varphi_1-\Lap\varphi_1) } }_{\rmL^2\ofbbT} \\
            &\leq C \norm{\Grad\tphi{j}}_{\rmL^4\ofbbT} \norm{\vel_j}_{\rmL^4\ofbbT} \norm{ \tb{j} \tant^{-1} }_{\rmH^2\ofbbT} \bignorm{ (-\Lap)^{-1} (\delt\varphi_1-\Lap\varphi_1) }_{\rmH^2\ofbbT} \\
            &\leq \hat C \norm{\vel_j}_{\rmH^1\ofbbT} \norm{ \delt\varphi_1-\Lap\varphi_1 }_{\rmL^2\ofbbT}
			\leq \hat{C} \bigprt{ C_\delta \norm{\vel_j}_{\rmH^1\ofbbT}^2 + \norm{\delt\varphi_1}_{\rmL^2\ofbbT}^2 + \delta \norm{\Lap\varphi_1}_{\rmL^2\ofbbT}^2 }.
		\end{align*}
		Finally, we control
        \begin{align*}
			\abs{L_4^j} 
			&\leq C \bignorm{ \Grad(\tphi{j}^{-1}) }_{\rmL^2\ofbbT} \norm{ \f_j }_{\rmL^2\ofbbT} \bignorm{ \tant^{-1} (-\Lap)^{-1} (\delt\varphi_1-\Lap\varphi_1) }_{\rmL^\infty\ofbbT} \\
            &\leq \hat C \norm{\f_j}_{\rmL^2\ofbbT} \norm{ \delt\varphi_1-\Lap\varphi_1 }_{\rmL^2\ofbbT}
			\leq \hat{C} \bigprt{ C_\delta \norm{\f_j}_{\rmL^2\ofbbT}^2 + \norm{\delt\varphi_1}_{\rmL^2\ofbbT}^2 + \delta \norm{\Lap\varphi_1}_{\rmL^2\ofbbT}^2 }.
		\end{align*}
		
		Now we are in a position to put together all the above estimates for the terms of equation~\eqref{Lin:eq:est_Lap_phi1_tested_eq}. We further make use of the fact that the norms $\norm{\cdot}_{\rmL^2\ofbbT}$, $\norm{\tant^{-\frac12}\cdot}_{\rmL^2\ofbbT}$, $\norm{\tbnt ^{-\frac12}\cdot}_{\rmL^2\ofbbT}$, and $\norm{(\tbnt\tant^{-1})^{\frac12}\cdot}_{\rmL^2\ofbbT}$ are equivalent. After choosing $\delta$ and $\delta_2$ sufficiently small, all terms including a factor $\delta$ and $\delta_2$ are absorbed. Note that the term $\sum_{j=1,2} \delta_1\norm{ \delt\vel_j}_{\rmH^{-1}\ofbbT}^2$ from $L_1$ is still present on the right-hand side. Altogether, we end up with
		\begin{align*}
			& \frac12 \ddt \norm{ \delt\varphi_1-\Lap\varphi_1 }_{\rmH^{-1}\ofbbT}^2
			+ \ddt \norm{ \Grad\varphi_1}_{\rmL^2\ofbbT}^2
			+ 2 \norm{ \Lap\varphi_1}_{\rmL^2\ofbbT}^2 
			+ \varepsilon \frac12 \ddt \norm{ \Grad\varphi_1 }_{\rmL^2\ofbbT}^2 
			+ \varepsilon \norm{ \delt\varphi_1 }_{\rmL^2\ofbbT}^2 
			+ \varepsilon \norm{ \Lap\varphi_1 }_{\rmL^2\ofbbT}^2 \\
			&\leq \hat{C} \Bigprt{ \norm{ \delt\varphi_1 -\Lap\varphi_1}_{\rmH^{-1}\ofbbT}^2 + \norm{ \varphi_1 }_{\rmH^1\ofbbT}^2 + \norm{ \delt\varphi_1 }_{\rmL^2\ofbbT}^2 
            + \varepsilon\norm{\Lap\varphi_1}_{\rmL^2\ofbbT}^2 \\
			&\qquad+ \sum_{j=1,2} \norm{ \vel_j }_{\rmH^1\ofbbT}^2 + \sum_{j=1,2} \delta_1 \norm{ \delt\vel_j }_{\rmH^{-1}\ofbbT}^2 
			+ \sum_{j=1,2} \norm{ \f_j }_{\rmL^2\ofbbT}^2 + \norm{ g_1 }_{\rmH^1\ofbbT}^2 + \norm{ \delt g_1}_{\rmL^2\ofbbT}^2 }.
		\end{align*}
		Then we integrate with respect to time and apply Gronwall's lemma to the first term on the right-hand side. Eventually, the energy estimate from Lemma~\ref{Lin:lemma:energy_est}, along with the estimate for $\delt\vel_j$ from Lemma~\ref{Lin:lemma:est_delt_velj_dep_phi1_gen}, gives
		\begin{align*}
			&\norm{\delt\varphi_1-\Lap\varphi_1}_{\rmL^\infty(0,T;\rmH^{-1}\ofbbT)}
			+ \norm{\Grad\varphi_1}_{\rmL^\infty(0,T;\rmL^2\ofbbT)} 
			+ \norm{\Lap\varphi_1}_{\rmL^2(0,T;\rmL^2\ofbbT)}\\
			&\quad+ \sqrt\varepsilon \norm{\Grad\varphi_1}_{\rmL^\infty(0,T;\rmL^2\ofbbT)}
			+ \sqrt\varepsilon \norm{\delt\varphi_1}_{\rmL^2(0,T;\rmL^2\ofbbT)}
			+ \sqrt\varepsilon \norm{\Lap\varphi_1}_{\rmL^2(0,T;\rmL^2\ofbbT)} \\
			&\leq \hat C \Bigprt{ \delta_1\norm{\varphi_1}_{\rmL^2(0,T;\rmH^2\ofbbT)} + \sum_{j=1,2} \norm{\f_j}_{\rmL^2(0,T;\rmL^2\ofbbTd)} + \norm{g_1}_{\rmL^2(0,T;\rmH^1\ofbbT)} + \norm{g_1}_{\rmH^1(0,T;\rmH^{-1}\ofbbT)} \\
			&\qquad+ \sum_{j=1,2} \norm{\vel_{j,0}}_{\rmL^2\ofbbTd} + \norm{\varphi_{1,0}}_{\rmH^1\ofbbT} + \norm{\delt\varphi_1\vert_{t=0}}_{\rmH^{-1}\ofbbT} + \norm{\Lap\varphi_1\vert_{t=0}}_{\rmH^{-1}\ofbbT} }.
		\end{align*}
		Finally, we also absorb the first term on the right-hand side, which leads to the desired estimate and finishes the proof.
	\end{proof}

	As a direct consequence, we obtain the following estimate for $\delt\vel_j$, independent of $\varphi_1$.
	
	\begin{lemma}[Estimate for $\delt\vel_j$] 
		\label{Lin:lemma:est_delt_velj_gen}
		Let $0<T<\infty$ and $\varepsilon>0$, and let $\tphi1\in\rmH^2(\bbT)$ satisfy~$\tphi1\in(0,1)$ in $\bbT$. Suppose that
        \begin{equation*}
            \begin{alignedat}{4}
                \f_j &\in\rmL^2(0,T;\rmL^2(\bbT)^d), &\quad j&=1,2, &\qquad
                g_1&\in\rmL^2(0,T;\rmH^1(\bbT)) \cap \rmH^1(0,T;\rmH^{-1}(\bbT)), \\
                \vel_{j,0}&\in \rmH^2(\bbT)^d, &\quad j&=1,2, &\qquad
                \varphi_{1,0}&\in\rmH^3(\bbT), 
            \end{alignedat}
        \end{equation*}
        and that the assumptions of Lemma~\ref{Lin:lemma:existence_abstr}~\ref{Lin:item:existence_abstr_iii} hold. Then the weak solution~$(\vel_1,\vel_2,\varphi_1)$ to the regularized system~\eqref{Lin:eq:reg_system} in the sense of Definition~\ref{Lin:def:abstract_eq_weak} satisfies the uniform \textit{a priori} estimate
		\begin{align}
			\label{Lin:eq:est_deltvj_gen_coeffs}
			&\sum_{j=1,2} \norm{\delt\vel_j}_{\rmL^2(0,T;\rmH^{-1}\ofbbT)} 
			\leq \hat C \Bigprt{ \sum_{j=1,2} \norm{\f_j}_{\rmL^2(0,T;\rmL^2\ofbbTd)} + \norm{g_1}_{\rmL^2(0,T;\rmH^1\ofbbT)} + \norm{g_1}_{\rmH^1(0,T;\rmH^{-1}\ofbbT)} \nonumber \\
			&\qquad\qquad+ \sum_{j=1,2} \norm{\vel_{j,0}}_{\rmL^2\ofbbTd} + \norm{\varphi_{1,0}}_{\rmH^1\ofbbT} + \norm{\delt\varphi_1\vert_{t=0}}_{\rmH^{-1}\ofbbT} + \norm{\Lap\varphi_1\vert_{t=0}}_{\rmH^{-1}\ofbbT}}
		\end{align}
		with a constant $\hat C>0$ depending on $\norm{\tphi1}_{\rmH^2\ofbbT}$ and $T$.
	\end{lemma}
	
	\begin{proof}
		Since $\Lap\varphi_1$ is now controlled in $\rmL^2(0,T;\rmL^2(\bbT))$ due to Lemma~\ref{Lin:lemma:est_Lap_phi1_gen}, the claimed estimate follows directly from Lemma~\ref{Lin:lemma:est_delt_velj_dep_phi1_gen}.
	\end{proof}

    Now, the original pressure can also be put under control.
	
	\begin{lemma}[Estimate for $p$] 
		\label{Lin:lemma:pressure_p_gen}
        Let $0<T<\infty$ and $\varepsilon>0$, and let $\tphi1\in\rmH^2(\bbT)$ satisfy~$\tphi1\in(0,1)$ in $\bbT$. Suppose that
        \begin{equation*}
            \begin{alignedat}{4}
                \f_j &\in\rmL^2(0,T;\rmL^2(\bbT)^d), &\quad j&=1,2, &\qquad
                g_1&\in\rmL^2(0,T;\rmH^1(\bbT)) \cap \rmH^1(0,T;\rmH^{-1}(\bbT)), \\
                \vel_{j,0}&\in \rmH^2(\bbT)^d, &\quad j&=1,2, &\qquad
                \varphi_{1,0}&\in\rmH^3(\bbT), 
            \end{alignedat}
        \end{equation*}
        and that the assumptions of Lemma~\ref{Lin:lemma:existence_abstr}~\ref{Lin:item:existence_abstr_iii} hold. Then, denoting by $(\vel_1,\vel_2,\varphi_1)$ the weak solution to the regularized system~\eqref{Lin:eq:reg_system} in the sense of Definition~\ref{Lin:def:abstract_eq_weak}, there exists a unique~$p\in\rmL^2(0,T;\rmL^2_{(0)}(\bbT))$ such that $(\vel_1,\vel_2,\varphi_1,p)$ solves \eqref{Lin:eq:reg_system_eq_deltv1}--\eqref{Lin:eq:reg_system_eq_deltv2} in the sense of distributions a.e.~in $(0,T)$ as well as \eqref{Lin:eq:reg_system_eq_deltphi1}--\eqref{Lin:eq:reg_system_eq_div} a.e.~in $(0,T)\times\bbT$. Moreover, $p$ satisfies the \textit{a priori} estimate
		\begin{align}
			\label{Lin:eq:est_p_gen_coeffs}
			\norm{p}_{\rmL^2(0,T;\rmL^2\ofbbT)} 
			&\leq \hat C \Bigprt{ \sum_{j=1,2} \norm{\f_j}_{\rmL^2(0,T;\rmL^2\ofbbTd)} + \norm{g_1}_{\rmL^2(0,T;\rmH^1\ofbbT)} + \norm{g_1}_{\rmH^1(0,T;\rmH^{-1}\ofbbT)} \nonumber \\
				&\qquad+ \sum_{j=1,2} \norm{\vel_{j,0}}_{\rmL^2\ofbbTd} + \norm{\varphi_{1,0}}_{\rmH^1\ofbbT} + \norm{\delt\varphi_1\vert_{t=0}}_{\rmH^{-1}\ofbbT} + \norm{\Lap\varphi_1\vert_{t=0}}_{\rmH^{-1}\ofbbT} }
		\end{align}
		with a constant $\hat C>0$ depending on $\norm{\tphi1}_{\rmH^2\ofbbT}$ and $T$.
	\end{lemma}
	
	\begin{proof}
		According to Lemma~\ref{Lin:lemma:pressure_q_gen}, the modified pressure defined as $q\coloneqq p - \frac{\tdnt}{\tant} \Lap\varphi_1$ is controlled in~$\rmL^2(0,T;\rmL^2(\bbT))$. In view of the estimate for $\Lap\varphi_1$ in the same space, see Lemma~\ref{Lin:lemma:est_Lap_phi1_gen}, the existence of $p\in\rmL^2(0,T;\rmL^2_{(0)}(\bbT))$ follows immediately, along with the claimed estimate.
	\end{proof}

	\subsection{Weak Solvability of the Linear System for General Coefficients}
    \label{subsec:lin:solution_gen_coeffs}
    The regularity estimates from Subsection~\ref{subsec:lin:est_gen_coeffs} allow us to pass to the limit in the regularized system~\eqref{Lin:eq:reg_system}. This shows a unique weak solvability of the linear system~\eqref{Lin:eq:lin_system} in the case of variable coefficients.
	
	\begin{proposition}
		\label{Lin:prop:ex_gen_coeff}
        Let $0<T<\infty$ and let $\tphi1\in\rmH^2(\bbT)$ satisfy $\tphi1\in(0,1)$ in $\bbT$. Suppose that
        \begin{equation*}
            \begin{alignedat}{4}
                \f_j &\in\rmL^2(0,T;\rmL^2(\bbT)^d), &\quad j&=1,2, &\quad
                g_j &\in\rmL^2(0,T;\rmH^1(\bbT)) \cap \rmH^1(0,T;\rmH^{-1}(\bbT)), &\quad j&=1,2,\\
                \vel_{j,0}&\in \rmH^1(\bbT)^d, &\quad j&=1,2, &\quad
                \varphi_{1,0}&\in\rmH^2(\bbT), 
            \end{alignedat}
        \end{equation*}
        such that $\Div(\tphi1\vel_{1,0} + \tphi2\vel_{2,0})=g_2\vert_{t=0}$ holds. Then there exists a unique weak solution~$(\vel_1, \vel_2, \varphi_1, p)$ to the linear system \eqref{Lin:eq:lin_system} in the sense of distributions a.e.~in~$(0,T)$, with
		\begin{align*}
			\vel_j &\in \rmL^\infty \prt{0,T;\rmL^2(\bbT)^d} \cap \rmL^2 \prt{ 0,T;\rmH^1(\bbT)^d} \cap \rmH^1 \prt{ 0,T;\rmH^{-1}(\bbT)^d}, \\
			\varphi_1 &\in \rmL^\infty \prt{ 0,T;\rmH^1(\bbT)} \cap \rmL^2 \prt{0,T;\rmH^2(\bbT)} \cap \rmH^1 \prt{ 0,T;\rmL^2(\bbT)}, \\
			p &\in \rmL^2 \prt{0,T;\rmL^2_{(0)}(\bbT)},
		\end{align*}
		and the solution depends continuously on the data in the respective norms.
	\end{proposition}
	
	\begin{proof}
		As in the proof of Proposition~\ref{Lin:prop:ex_const_coeff}, before passing to the limit, we first reduce to~$\mean{\varphi_1}=0$ and $\mean{\varphi_{1,0}}= 0$, as well as to right-hand sides satisfying $g_2=0$, $\mean{g_1}=0$, and
		\begin{align*}
			\Div( \barrho1^{-1}\f_1 + \barrho2^{-1}\f_2 ) =0 \quad\text{in } \mathcal D'(\bbT) \quad\text{for a.e.~}t\in[0,T] \text{ including } t=0.
		\end{align*}
		Moreover, we approximate the initial data as before. Accordingly, for each $\varepsilon>0$, we deduce from Lemma~\ref{Lin:lemma:existence_abstr} the existence of a strong solution $(\vel_{1,\varepsilon}, \vel_{2,\varepsilon}, \varphi_{1,\varepsilon}, p_{\varepsilon})$ to the regularized linear system~\eqref{Lin:eq:reg_system} with right-hand side~$(\f_{1,\varepsilon}, \f_{2,\varepsilon},g_1,0)$, approximated as in Proposition~\ref{Lin:prop:ex_const_coeff}.
		
		Then the uniform estimates from the Lemmas~\ref{Lin:lemma:energy_est}, \ref{Lin:lemma:est_Lap_phi1_gen}, \ref{Lin:lemma:est_delt_velj_gen}, and \ref{Lin:lemma:pressure_p_gen} enable us to pass to non-relabeled subsequences converging weakly (or weakly-*, respectively) as $\varepsilon\to0$ as follows:
		\begin{alignat*}{3}
			\vel_{j,\varepsilon} &\overset{(\ast)\;\,}{\rightharpoonup} \vel_j &&\quad\text{in } \rmL^\infty \prt{ 0,T;\rmL^2(\bbT)^d} \cap \rmL^2 \prt{0,T;\rmH^1(\bbT)^d}, &\qquad j=1,2, \\
			\delt\vel_{j,\varepsilon} &\rightharpoonup \delt\vel_j &&\quad\text{in } \rmL^2 \prt{0,T;\rmH^{-1}(\bbT)^d}, &\qquad j=1,2, \\
			\varphi_{1,\varepsilon} &\overset{(\ast)\;\,}{\rightharpoonup} \varphi_1 &&\quad\text{in } \rmL^\infty \prt{ 0,T;\rmH^1(\bbT)} \cap \rmL^2 \prt{0,T;\rmH^2(\bbT)}, & \\
			\delt\varphi_{1,\varepsilon} &\overset{(\ast)\;\,}{\rightharpoonup} \delt\varphi_1 &&\quad\text{in } \rmL^\infty \prt{ 0,T;\rmH^{-1}(\bbT)} \cap \rmL^2 \prt{0,T;\rmL^2(\bbT)}, & \\
			p_\varepsilon &\rightharpoonup p &&\quad\text{in } \rmL^2 \prt{0,T;\rmL^2_{(0)}(\bbT)}. &
		\end{alignat*}
		In particular, each term in \eqref{Lin:eq:reg_system_eq_deltv1}--\eqref{Lin:eq:reg_system_eq_deltv2} converges weakly in $\rmL^2(0,T;\rmH^{-1}(\bbT)^d)$, while each term in \eqref{Lin:eq:reg_system_eq_deltphi1}--\eqref{Lin:eq:reg_system_eq_div} converges weakly in $\rmL^2(0,T;\rmL^2(\bbT))$. Since this is sufficient to pass to the limit in the system, $(\vel_1,\vel_2,\varphi_1,p)$ is a weak solution to \eqref{Lin:eq:lin_system_eq_deltv1}--\eqref{Lin:eq:lin_system_eq_div} with the claimed regularity properties. Moreover, the uniform estimates from the Lemmas cited above ensure that the solution depends continuously on the data in the respective norms.
		
		Finally, uniqueness follows as in the proof of Proposition~\ref{Lin:prop:ex_const_coeff}.
	\end{proof}

	\subsection{Higher Regularity of the Solution for General Coefficients}
    \label{subsec:lin:localization}
    At this stage, we have established the existence of a unique higher-regularity solution in the case of constant coefficients, as well as a unique lower-regularity solution for general coefficients. We now show that the latter is, in fact, of higher regularity as well. This is achieved by combining perturbation and localization techniques.

	For this subsection, we recall the definition of $Z_T$, see \eqref{mr:eq:spaces_Z_T}, and modify the space $Y_T$, defined in \eqref{mr:eq:spaces_Y_T}, by including the initial data in the space. To this end, we first set
	\begin{alignat*}{2}
		\tilde\varUpsilon_T^j &\coloneqq \rmL^2 \prt{0,T;\rmL^2(\bbT)^d}, &\quad j=1,2, \\
		\tilde\varUpsilon_T^j &\coloneqq \rmL^2 \prt{0,T;\rmH^1(\bbT)} \cap \rmH^1 \prt{ 0,T;\rmH^{-1}(\bbT)}, &\quad j=3,4, \\
		\tilde\varUpsilon_T^j &\coloneqq \rmH^1(\bbT)^d, &\quad j=5,6, \\
		\tilde\varUpsilon_T^7 &\coloneqq \rmH^2(\bbT),
	\end{alignat*}
    equipped with the usual norms, and the associated product space $\tilde\varUpsilon_T\coloneqq \prod_{j=1}^7 \tilde\varUpsilon_T^j$. Then, given $\tphi1\in\rmH^2(\bbT)$, we define the subspace
    \begin{align*}
        \varUpsilon_T \coloneqq \{ (\f_1,\f_2,g_1,g_2,\vel_{1,0},\vel_{2,0},\varphi_{1,0}) \in \tilde\varUpsilon_T \colon &\Div(\tphi1\vel_{1,0} + (1\!-\!\tphi1) \vel_{2,0}) = g_2\vert_{t=0} \text{ in } \rmL^2(\bbT) \},
    \end{align*}
    whose components will be denoted by $\varUpsilon_T = \prod_{j=1}^7 \varUpsilon_T^j$, correspondingly.
    Next, we introduce the linear operator $\mathcal{A}_T(\tphi1) \colon Z_T\to \varUpsilon_T$ that corresponds to the left-hand side of~\eqref{Lin:eq:lin_system} by
    \begin{align*}
		\mathcal{A}_T(\tphi1) (\z) \coloneqq
		\begin{pmatrix}
			\rho_1(\tphi1)\delt\vel_1 - \Div\Str_1(\tphi1,\D\vel_1) + \tphi1\Grad p - \tphi1 \Grad\Lap\varphi_1 \\[1ex]
			\rho_2(\tphi1)\delt\vel_2 - \Div\Str_2(\tphi1,\D\vel_2) + (1-\tphi1)\Grad p + (1-\tphi1)\Grad\Lap\varphi_1 \\[1ex]
			\delt\varphi_1 + \Div(\tphi1\vel_1) \\[1ex]
			\Div( \tphi1\vel_1 + (1-\tphi1)\vel_2) \\
            \vel_1\vert_{t=0} \\
            \vel_2\vert_{t=0} \\
            \varphi_1\vert_{t=0}
		\end{pmatrix},
	\end{align*}
    for every $\z=(\vel_1,\vel_2,\varphi_1,p)\in Z_T$, where $\Str_j$ is defined in \eqref{int:eq:S_j}. Note that, in contrast to~$\mathcal{L}_T$, this notation additionally highlights the dependence of the operator $\mathcal{A}_T(\tphi1)$ on the given linearization parameter~$\tphi1\in\rmH^2(\bbT)$. The components of $\mathcal{A}_T(\tphi1)$ are indicated by~$\mathcal{A}_T^j(\tphi1)$, $j\in\{1,\dots,7\}$.
    
	With these definitions at hand, abbreviating $\z=(\vel_1,\vel_2,\varphi_1,p)$ and $\F=(\f_1, \f_2, g_1, g_2)$, we are able to equivalently write the linear system~\eqref{Lin:eq:lin_system_eq_deltv1}--\eqref{Lin:eq:lin_system_eq_div} in an abstract form
	\begin{align}
		\label{Lin:eq:abstr_eq}
		\mathcal{A}_T(\tphi1) (\z) &= (\F,\vel_{1,0}, \vel_{2,0}, \varphi_{1,0})
        \eqqcolon \mathfrak F.
	\end{align}
    
    In Proposition~\ref{Lin:prop:ex_const_coeff}, we have proven invertibility of $\mathcal{A}_T(\barphi1)$ with respect to the higher regularity $Z_T$ for any constant $\barphi1$. Our next goal is to verify that for variable coefficients~$\tphi1$ that are sufficiently close to some fixed constant $\barphi1$, the operator $\mathcal{A}_T(\tphi1)$ remains invertible.

	\begin{proposition}[Perturbation argument]
		\label{Lin:prop:perturbation}
		Let $0<T_0<\infty$ be given. Let $\barphi1\in(0,1)$ be a fixed constant such that the operator~$\mathcal{A}_T(\barphi1) \colon Z_T\to \varUpsilon_T$ is invertible for all~$0<T\leq T_0$. Then, for any~$\tphi1\in\rmH^2(\bbT)$ with $\tphi1\in(0,1)$ in $\bbT$ satisfying the condition
		\begin{align}
			\label{Lin:eq:tphi_close_to_constant}
			\norm{\tphi1 - \barphi1}_{\rmL^\infty\ofbbT}
			< \delta
		\end{align}
		for some sufficiently small $\delta=\delta(\barphi1, T_0)>0$, the operator~$\mathcal{A}_T(\tphi1) \colon Z_T\to \varUpsilon_T$ is invertible for all $0<T\leq T_0$ as well.
	\end{proposition}
	
	\begin{proof}
		In a first step, our goal is to prove the claimed invertibility for all $0<T\leq T_1$, where~$T_1=T_1(\norm{\tphi1}_{\rmH^2\ofbbT}) \leq T_0$ is sufficiently small. To this end, let $0<T\leq T_1$. 
		Proposition~\ref{Lin:prop:ex_const_coeff} yields that for the fixed constant $\barphi1$, the operator $\mathcal{A}_T(\barphi1) \colon Z_T\to \varUpsilon_T$ is invertible. Concerning the inverse operators, we denote $M\coloneqq \sup_{0<T\leq T_0} \norm{\mathcal A(\barphi1)^{-1}}_{\mathcal L(\varUpsilon_T,Z_T)}$. At the end of this proof, we show
		\begin{align}
			\label{Lin:eq:A_Lipschitz}
			\norm{\mathcal{A}_T(\tphi1) - \mathcal{A}_T(\barphi1)}_{\mathcal L(Z_T,\varUpsilon_T)}
			\leq L \norm{\tphi1 - \barphi1}_{\rmL^\infty\ofbbT} 
			+ CT^{\frac18} \norm{\tphi1}_{\rmH^2\ofbbT} 
		\end{align}
		for all $\tphi1\in \rmH^2(\bbT)$. Once this estimate is verified, we choose $T_1=T_1(\norm{\tphi1}_{\rmH^2\ofbbT})$ sufficiently small, such that it holds $CT_1^{\frac18} \norm{\tphi1}_{\rmH^2\ofbbT} < (4M)^{-1}$. Then we have, for any~$\tphi1\in \rmH^2(\bbT)$ satisfying the closeness condition \eqref{Lin:eq:tphi_close_to_constant} with $\delta\coloneqq(4LM)^{-1}$, and for any~$0<T\leq T_1$, that
		\begin{align*}
			\norm{\mathcal{A}_T(\tphi1) - \mathcal{A}_T(\barphi1)}_{\mathcal L(Z_T,\varUpsilon_T)}
			\leq (2M)^{-1}.
		\end{align*}
		Since the set of invertible operators is open (cf.~Neumann series), it follows that the operator~$\mathcal{A}_T(\tphi1) \colon Z_T\to \varUpsilon_T$ is also invertible for any $\tphi1\in \rmH^2(\bbT)$ satisfying \eqref{Lin:eq:tphi_close_to_constant}, and for any $0<T\leq T_1$.
	
		In a second step, we consider arbitrary $T_1<T\leq T_0$ and write $T=nT_1 + T'$, where~$n\in\N_0$ and $0\leq T'< T_1$.
		Given $(\F,\vel_{1,0},\vel_{2,0},\varphi_{1,0})\in \varUpsilon_T$, we derive from the first step that for any $T\in(0,T_1]$, there exists a unique solution $\z^1=(\vel_1^1,\vel_2^1,\varphi_1^1,p^1)\in Z_{T}$ to the equation~$\mathcal{A}_T(\z^1)= (\F,\vel_{1,0},\vel_{2,0},\varphi_{1,0})$. Then we define $(\vel_1^1\vert_{t=T_1} ,\vel_2^1\vert_{t=T_1},\varphi_1^1\vert_{t=T_1})$ as new initial data and find a unique solution to the associated problem on the interval $(T_1,T_2]$, where~$T_2=2T_1$. Iterating this argument until the final interval $(nT_1, nT_1 + T']$, we conclude the invertibility of $\mathcal{A}_T(\tphi1) \colon Z_T\to \varUpsilon_T$ for all $0<T\leq T_0$ and all $\tphi1\in \rmH^2(\bbT)$ satisfying~\eqref{Lin:eq:tphi_close_to_constant} as claimed.

		\medskip
		\textit{Ad \eqref{Lin:eq:A_Lipschitz}.} It remains to prove the estimate~\eqref{Lin:eq:A_Lipschitz}. To this end, we study each component of the difference~$\norm{\mathcal{A}_T(\tphi1)(\z) - \mathcal{A}_T(\barphi1)(\z)}_{\varUpsilon_T}$ separately, where we point out that it suffices to consider the first four components. In the following, we repeatedly make use of the fact that $\vel_1\in Z_T^1$ implies $\Grad\vel_1\in \rmL^{\frac83}\prt{0,T;\rmL^4(\bbT)^{d\times d}}$, which allows us to estimate~$\norm{ \Grad\vel_1 }_{\rmL^2\prt{0,T;\rmL^4(\bbT)^d}} \leq T^{\frac18} \norm{ \Grad\vel_1 }_{\rmL^{\frac83}\prt{0,T;\rmL^4(\bbT)^d}} \leq C T^{\frac18} \norm{\vel_1}_{Z_T^1}$. For the first component, we obtain
		\begin{align*}
			&\bignorm{\mathcal{A}_T^1(\tphi1)(\z) - \mathcal{A}_T^1(\barphi1)(\z)}_{\varUpsilon_T^1} \\
			&\leq C \norm{\tphi1 - \barphi1}_{\rmL^\infty\ofbbT} 
			\Bigprt{ \norm{\delt\vel_1}_{\rmL^2 \prt{0,T;\rmL^2\ofbbTd}}
			+ \norm{\Lap\vel_1}_{\rmL^2 \prt{0,T;\rmL^2\ofbbTd}} 
			+ \norm{\Grad\Div\vel_1}_{\rmL^2 \prt{0,T;\rmL^2\ofbbTd}} \\
			&\qquad\qquad\qquad\qquad\quad
			+ \norm{\Grad p}_{\rmL^2 \prt{0,T;\rmL^2\ofbbT}}
			+ \norm{\Grad\Lap\varphi_1}_{\rmL^2 \prt{0,T;\rmL^2\ofbbT}} }
			+ C \norm{ \Grad\vel_1 }_{\rmL^2\prt{0,T;\rmL^4\ofbbTd}} \bignorm{ \Grad\tphi1 }_{\rmL^4\ofbbTd} \\
			&\leq C\norm{\tphi1 - \barphi1}_{\rmL^\infty\ofbbT} \norm{\z}_{Z_T} + CT^{\frac18} \norm{ \tphi1 }_{\rmH^2\ofbbTd} \norm{\z}_{Z_T},
		\end{align*}
		due to the Lipschitz continuity of $\rho_1$, $\nu_1$ and $\lambda_1$ following from Assumption~\eqref{ass:coeffs}. The second component is estimated correspondingly. 
		
		Proceeding with the third component, it is necessary to take into account the definition of~$\varUpsilon_T^3= \rmL^2 \prt{0,T;\rmH^1(\bbT)} \cap \rmH^1\prt{0,T;\rmH^{-1}(\bbT)}$. Regarding the first norm, we compute 
		\begin{align*}
			&\bignorm{\mathcal{A}_T^3(\tphi1)(\z) - \mathcal{A}_T^3(\barphi1)(\z)}_{\rmL^2\prt{0,T;\rmH^1\ofbbT}} 
			= \norm{\Div \bigprt{(\tphi1-\barphi1)\vel_1} }_{\rmL^2\prt{0,T;\rmH^1\ofbbTd}} \\
			&\leq \norm{\prt{\tphi1-\barphi1} \Div\vel_1 }_{\rmL^2\prt{0,T;\rmH^1\ofbbTd}} 
			+ \norm{\Grad\tphi1\cdot\vel_1 }_{\rmL^2\prt{0,T;\rmH^1\ofbbTd}}
			\eqqcolon I + II.
		\end{align*}
		Invoking the above estimate for $\Grad\vel_1$, we find
		\begin{align*}
			I
			&\leq \norm{\prt{\tphi1-\barphi1} \Div\vel_1 }_{\rmL^2\prt{0,T;\rmL^2\ofbbTd}} 
			+ \norm{\prt{\tphi1-\barphi1} \Grad\Div\vel_1 }_{\rmL^2\prt{0,T;\rmL^2\ofbbTd}} \\
			&\quad+ \norm{ \Div\vel_1 \Grad\tphi1 }_{\rmL^2\prt{0,T;\rmL^2\ofbbTd}} \\
			&\leq \norm{\tphi1 - \barphi1}_{\rmL^\infty\ofbbT} \! \bigprt{ \norm{\Div\vel_1}_{\rmL^2\prt{0,T;\rmL^2\ofbbTd}} + \norm{\Grad\Div\vel_1}_{\rmL^2\prt{0,T;\rmL^2\ofbbTd}} }
			+ \norm{ \Div\vel_1 }_{\rmL^2\prt{0,T;\rmL^4\ofbbTd}} \norm{ \Grad\tphi1 }_{\rmL^4\ofbbTd} \\
			&\leq C \norm{\tphi1 - \barphi1}_{\rmL^\infty\ofbbT} \norm{\z}_{Z_T} + C T^{\frac18} \norm{ \tphi1 }_{\rmH^2\ofbbTd} \norm{\z}_{Z_T}.
		\end{align*}
		By interpolating suitably and using Agmon's inequality, we obtain
		\begin{align*}
			II
			&\leq \norm{\Grad\tphi1\cdot\vel_1 }_{\rmL^2\prt{0,T;\rmL^2\ofbbTd}} 
			+ \norm{\Grad\tphi1\Grad\vel_1 }_{\rmL^2\prt{0,T;\rmL^2\ofbbTd}} 
			+ \norm{ \Grad^2\tphi1\vel_1 }_{\rmL^2\prt{0,T;\rmL^2\ofbbTd}} \\
			&\leq \norm{ \Grad\tphi1 }_{\rmL^4\ofbbTd} \norm{ \vel_1 }_{\rmL^2\prt{0,T;\rmL^4\ofbbTd}}  
			+ \norm{ \Grad\tphi1 }_{\rmL^4\ofbbTd} \norm{ \Grad\vel_1 }_{\rmL^2\prt{0,T;\rmL^4\ofbbTd}} 
			+ \norm{ \Grad^2\tphi1 }_{\rmL^2\ofbbTd} \norm{ \vel_1 }_{\rmL^2\prt{0,T;\rmL^\infty\ofbbTd}}  \\
			&\leq CT^{\frac12} \norm{ \tphi1 }_{\rmH^2\ofbbTd} \norm{ \vel_1 }_{\rmL^\infty\prt{0,T;\rmH^1\ofbbTd}} 
			+ CT^{\frac18} \norm{ \tphi1 }_{\rmH^2\ofbbTd} \norm{ \Grad\vel_1 }_{\rmL^{\frac83}\prt{0,T;\rmL^4\ofbbTd}} \\
			&\quad+ CT^{\frac14} \norm{ \tphi1 }_{\rmH^2\ofbbTd} \norm{ \vel_1 }_{\rmL^2\prt{0,T;\rmH^2\ofbbTd}}^{\frac12} \norm{ \vel_1 }_{\rmL^\infty\prt{0,T;\rmH^1\ofbbTd}}^{\frac12}.
		\end{align*}
		In the other norm, we have
		\begin{align*}
			&\bignorm{\mathcal{A}_T^3(\tphi1)(\z) - \mathcal{A}_T^3(\barphi1)(\z)}_{\rmH^1\prt{0,T;\rmH^{-1}\ofbbT}} 
			= \norm{\Div \bigprt{(\tphi1-\barphi1)\vel_1} }_{\rmH^1\prt{0,T;\rmH^{-1}\ofbbTd}} \\
			&\leq \norm{(\tphi1-\barphi1)\vel_1 }_{\rmL^2\prt{0,T;\rmL^2\ofbbTd}} 
			+ \norm{(\tphi1-\barphi1)\delt\vel_1 }_{\rmL^2\prt{0,T;\rmL^2\ofbbTd}} \\
			&\leq C \norm{\tphi1 - \barphi1}_{\rmL^\infty\ofbbT} 
			\bigprt{ \norm{\vel_1}_{\rmL^2 \prt{0,T;\rmL^2\ofbbTd}} + \norm{\delt\vel_1}_{\rmL^2 \prt{0,T;\rmL^2\ofbbTd}} } \\
			&\leq C \norm{\tphi1 - \barphi1}_{\rmL^\infty\ofbbT} \norm{\z}_{Z_T}.
		\end{align*}
		We point out that the fourth component may be treated similarly and that the respective expression for the remaining components vanishes. Eventually, taking the supremum over all $\norm{\z}_{Z_T}\leq1$ leads to the claimed estimate \eqref{Lin:eq:A_Lipschitz} for the operator norm.
	\end{proof}

    We conclude by proving that the weak solution to the linear system~\eqref{Lin:eq:lin_system} constructed in Proposition~\ref{Lin:prop:ex_gen_coeff} actually enjoys strong regularity. The argument relies on a localization procedure, to which we apply the perturbation result from Proposition~\ref{Lin:prop:perturbation}.
	
	\begin{proposition}[Localization argument]
		\label{Lin:prop:localization}
        Let $0<T<\infty$ and let $\tphi1\in\rmH^2(\bbT)$ satisfy~$\tphi1\in(0,1)$ in $\bbT$. Suppose that
        \begin{equation*}
            \begin{alignedat}{4}
                \f_j &\in\rmL^2(0,T;\rmL^2(\bbT)^d), &\quad j&=1,2, &\quad
                g_j &\in\rmL^2(0,T;\rmH^1(\bbT)) \cap \rmH^1(0,T;\rmH^{-1}(\bbT)), &\quad j&=1,2,\\
                \vel_{j,0}&\in \rmH^1(\bbT)^d, &\quad j&=1,2, &\quad
                \varphi_{1,0}&\in\rmH^2(\bbT), 
            \end{alignedat}
        \end{equation*}
        such that $\Div(\tphi1\vel_{1,0} + \tphi2\vel_{2,0})=g_2\vert_{t=0}$ holds. Then the weak solution $(\vel_1,\vel_2,\varphi_1,p)$ to the linear system~\eqref{Lin:eq:lin_system} from Proposition~\ref{Lin:prop:ex_gen_coeff} has higher regularity
		\begin{align*}
			\vel_j &\in \rmL^2 \prt{0,T;\rmH^2(\bbT)^d} \cap \rmH^1 \prt{ 0,T;\rmL^2(\bbT)^d}, \\
			\varphi_1 &\in \rmL^2 \prt{0,T;\rmH^3(\bbT)} \cap \rmH^1 \prt{ 0,T;\rmH^1(\bbT)} \cap \rmH^2 \prt{ 0,T;\rmH^{-1}(\bbT)}, \\
			p &\in \rmL^2 \prt{0,T;\rmH^1_{(0)}(\bbT)}
		\end{align*}
		and is thus a strong solution in $Z_T$.
	\end{proposition}
	
	\begin{proof}
		To show that $(\vel_1,\vel_2,\varphi_1,p)$ has higher regularity, we aim to apply Proposition~\ref{Lin:prop:perturbation} in a localization argument and adapt the notation of its proof.
		For each $x_0\in\bbT$, there exists a radius $r_{x_0}>0$ such that
		\begin{align*}
			\norm{\tphi1 - \tphi1(x_0)}_{\rmL^\infty(B_{r_{x_0}}(x_0))} \leq \frac{\delta_{x_0}}{2}
		\end{align*}
		with $\delta_{x_0} \coloneqq \bigprt{4L \norm{\mathcal{A}_T(\tphi1(x_0))^{-1}}_{\mathcal L(\varUpsilon_T,Z_T)} }^{-1}$, cf.~\eqref{Lin:eq:A_Lipschitz}. 
		Since~$\bbT$ is compact, there exists a finite open covering $\bbT = \bigcup_{k=1}^N B_{r_k/2}(x_{0,k})$, where $r_k\coloneqq r_{x_{0,k}}$. Next, we fix a partition of unity $\{\eta_k\}_{k=1}^N$ of class $\rmC^\infty$ subordinate to this covering.
		Furthermore, let $\psi_k \in \rmC_0^\infty(B_{r_k}(x_{0,k}))$ be a cutoff function satisfying $0\leq\psi_k\leq1$ in~$B_{r_k}(x_{0,k})$ and~$\psi_k=1$ in $\overline{B_{r_k/2}(x_{0,k})}$. Defining
		\begin{align*}
			\hat\varphi_{1,0}^k(x)
			\coloneqq \psi_k(x)\tphi1(x) + \bigprt{1-\psi_k(x)}\tphi1(x_{0,k}),
		\end{align*}
		we observe that 
		\begin{alignat*}{2}
			\hat\varphi_{1,0}^k(x) &= \tphi1(x)
			&&\qquad \text{for all } x\in \overline{B_{r_k/2}(x_{0,k})}, \\
			\hat\varphi_{1,0}^k(x) &= \tphi1(x_{0,k})
			&&\qquad \text{for all } x\in \bbT\setminus B_{r_k}(x_{0,k}),
		\end{alignat*}
		as well as
		\begin{align}
			\label{Lin:eq:hatphik_close_to_const}
			&\bignorm{\hat\varphi_{1,0}^k - \tphi1(x_{0,k})}_{\rmL^\infty(B_{r_k}(x_{0,k}))} 
			= \bignorm{\psi_k \bigprt{ \tphi1 - \tphi1(x_{0,k})} }_{\rmL^\infty(B_{r_k}(x_{0,k}))} \nonumber \\[-0.2mm]
			&\leq \norm{\psi_k}_{\rmL^\infty(B_{r_k}(x_{0,k}))}  \norm{ \tphi1 - \tphi1(x_{0,k}) }_{\rmL^\infty(B_{r_k}(x_{0,k}))}  
			\leq \frac{\delta_k}{2}
			\leq \frac{\delta}{2},
		\end{align}
		where $\delta_k\coloneqq\delta_{x_{0,k}}$ and $\delta\coloneqq\max_{k\in\{1,\dots,N\}}\delta_k$. In the following, localized functions are denoted by $\zeta^k\coloneqq\eta_k\zeta$ for any occurring function $\zeta$ and, accordingly, we write $\z^k= (\vel_1^k,\vel_2^k,\varphi_1^k,p^k)$ as well as~$\mathfrak F^k=(\f_1^k,\f_2^k, g_1^k, g_2^k, (\vel_1\vert_{t=0})^k, (\vel_2\vert_{t=0})^k, (\varphi_1\vert_{t=0})^k) $. With the commutator defined by~$[\mathcal B,\eta] \coloneqq \mathcal B(\eta\cdot) - \eta\mathcal B(\cdot)$, multiplying the abstract equation~\eqref{Lin:eq:abstr_eq} by $\eta_k$ gives
		\begin{align}
			\label{Lin:eq:localized_abstr_eq}
			&\mathcal{A}_T\prt{\hat\varphi_{1,0}^k} (\z^k) 
			= \eta_k \mathcal{A}_T\prt{\hat\varphi_{1,0}^k} (\z)
			+ \big[\mathcal{A}_T\prt{\hat\varphi_{1,0}^k}, \eta_k \big] (\z) \nonumber \\
			&= \mathfrak F^k
			+ \eta_k \bigprt{ \mathcal{A}_T\prt{\hat\varphi_{1,0}^k} - \mathcal{A}_T(\tphi1) } (\z) 
			+ \big[\mathcal{A}_T\prt{\hat\varphi_{1,0}^k}, \eta_k \big] (\z) 
			\eqqcolon \mathfrak F^k + R_1^k + R_2^k
			\coloneqq \mathfrak G^k.
		\end{align}
		Here, we have the remainder terms $R_1^k$  and $R_2^k$, where the first is a perturbation term and the second a commutator term. 
		
		In order to apply Proposition~\ref{Lin:prop:perturbation} to show that \eqref{Lin:eq:localized_abstr_eq} admits a unique solution $\z^k\in Z_T$, we need to verify that $\mathfrak G^k$ satisfies the assumptions of that proposition. To this end, we check that $\mathfrak G^k$ has the correct regularity, i.e., that $\mathfrak G^k\in\varUpsilon_T$, and only depends on the original data~$\mathfrak F\coloneqq (\f_1,\f_2, g_1, g_2, \vel_1\vert_{t=0}, \vel_2\vert_{t=0}, \varphi_1\vert_{t=0})$. We study the parts of $\mathfrak G^k$ separately and observe that the perturbation term~$R_1^k$ vanishes in view of the properties of $\hat\varphi_{1,0}^k$ and~$\eta_k$. On the other hand, the commutator term~$R_2^k$ is of lower order, which allows us to prove~$R_2^k\in \varUpsilon_T$ as follows. In order to estimate the first component of $R_2^k$, we study the terms
		\begin{align*}
			&\bignorm{ \big[\mathcal{A}_T^1\prt{\hat\varphi_{1,0}^k}, \eta_k \big] (\z) }_{\varUpsilon_T^1} 
			= \bignorm{ \mathcal{A}_T^1\prt{\hat\varphi_{1,0}^k} (\eta_k\z) 
				- \eta_k \mathcal{A}_T^1\prt{\hat\varphi_{1,0}^k} (\z)}_{\varUpsilon_T^1} \\
			&\leq\bignorm{\rho_1\prt{\hat\varphi_{1,0}^k} \delt(\eta_k\vel_1) - \eta_k \rho_1\prt{\hat\varphi_{1,0}^k} \delt\vel_1 }_{\rmL^2 \prt{0,T;\rmL^2\ofbbTd}} \\
			&\quad+ \bignorm{\nu_1\prt{\hat\varphi_{1,0}^k} \Lap(\eta_k\vel_1) - \eta_k \nu_1\prt{\hat\varphi_{1,0}^k} \Lap\vel_1 }_{\rmL^2 \prt{0,T;\rmL^2\ofbbTd}}\\
			&\quad+ \bignorm{ \bigprt{ \nu_1\prt{\hat\varphi_{1,0}^k} + \lambda_1\prt{\hat\varphi_{1,0}^k} } \Grad\Div(\eta_k\vel_1)
			- \eta_k \bigprt{ \nu_1\prt{\hat\varphi_{1,0}^k} + \lambda_1\prt{\hat\varphi_{1,0}^k} } \Grad\Div\vel_1 }_{\rmL^2 \prt{0,T;\rmL^2\ofbbTd}} \\
			&\quad+ \bignorm{ \hat\varphi_{1,0}^k \Grad (\eta_k p) - \eta_k \hat\varphi_{1,0}^k \Grad p }_{\rmL^2 \prt{0,T;\rmL^2\ofbbTd}} 
			+ \bignorm{ \hat\varphi_{1,0}^k \Grad\Lap (\eta_k \varphi_1) - \eta_k  \hat\varphi_{1,0}^k \Grad\Lap\varphi_1
			}_{\rmL^2 \prt{0,T;\rmL^2\ofbbTd}} \\
			&\eqqcolon I + II + III + IV + V.
		\end{align*}
		We observe that, since $\eta_k$ is independent of $t$, the term $I$ vanishes. Continuing with $II$, we find
		\begin{align*}
			II
			\leq \bignorm{ \nu_1\prt{\hat\varphi_{1,0}^k}}_{\rmL^\infty\ofbbT} \norm{2\Grad\vel_1\Grad\eta_k + \Lap\eta_k\vel_1}_{\rmL^2 \prt{0,T;\rmL^2\ofbbTd}} 
			\leq C \norm{ \hat\varphi_{1,0}^k}_{\rmW^{1,4}\ofbbT} \norm{\vel_1}_{\rmL^2 \prt{0,T;\rmH^1\ofbbTd}},
		\end{align*}
		and the analogous estimate for $III$. Similarly, we obtain
		\begin{align*}
			IV
			\leq \norm{\hat\varphi_{1,0}^k}_{\rmL^\infty\ofbbT} \norm{p\Grad\eta_k}_{\rmL^2 \prt{0,T;\rmL^2\ofbbTd}} 
			\leq C \norm{ \hat\varphi_{1,0}^k}_{\rmW^{1,4}\ofbbT} \norm{p}_{\rmL^2 \prt{0,T;\rmL^2\ofbbTd}}
		\end{align*}
		as well as
		\begin{align*}
			V
			&\leq \norm{\hat\varphi_{1,0}^k}_{\rmL^\infty\ofbbT} \bignorm{\Lap\varphi_1\Grad\eta_k + 2\Grad^2\varphi_1\Grad\eta_k + \Lap\eta_k\Grad\varphi_1 + 2\Grad^2\eta_k\Grad\varphi_1 + \varphi_1\Grad\Lap\eta_k}_{\rmL^2 \prt{0,T;\rmL^2\ofbbTd}} \\
			&\leq C \norm{ \hat\varphi_{1,0}^k}_{\rmW^{1,4}\ofbbT} \norm{\varphi_1}_{\rmL^2 \prt{0,T;\rmH^2\ofbbT}}.
		\end{align*}
		An estimate for the second component of $R_2^k$ follows accordingly. 
		
		The third component of $R_2^k$ yields the parts
		\begin{align*}
			&\bignorm{ \big[\mathcal{A}_T^3\prt{\hat\varphi_{1,0}^k}, \eta_k \big] (\z) }_{\varUpsilon_T^3} 
			= \bignorm{ \mathcal{A}_T^3\prt{\hat\varphi_{1,0}^k} (\eta_k\z) 
				- \eta_k \mathcal{A}_T^3\prt{\hat\varphi_{1,0}^k} (\z)}_{\varUpsilon_T^3} \\
			&\leq \norm{ \delt(\eta_k\vel_1) - \eta_k \delt\vel_1 }_{\varUpsilon_T^3} 
			+ \norm{ \hat\varphi_{1,0}^k \Div(\eta_k \vel_1) - \eta_k \hat\varphi_{1,0}^k \Div\vel_1 }_{\varUpsilon_T^3} 
			= \norm{ \hat\varphi_{1,0}^k \Grad\eta_k \cdot \vel_1 }_{\varUpsilon_T^3}.
		\end{align*} 
		Recalling the definition of $\varUpsilon_T^3$, we proceed with estimating
		\begin{align*}
			&\norm{ \hat\varphi_{1,0}^k \Grad\eta_k \cdot \vel_1 }_{\rmL^2\prt{0,T;\rmH^1\ofbbT}} 
			\leq \norm{ \hat\varphi_{1,0}^k \Grad\eta_k \cdot \vel_1 }_{\rmL^2\prt{0,T;\rmL^2\ofbbT}} 
			+ \norm{ \hat\varphi_{1,0}^k \Grad\vel_1 \Grad\eta_k }_{\rmL^2\prt{0,T;\rmL^2\ofbbT}} \\
			&\qquad\qquad\qquad\qquad\qquad\quad+ \norm{ \hat\varphi_{1,0}^k \Grad^2\eta_k \vel_1 }_{\rmL^2\prt{0,T;\rmL^2\ofbbT}} 
			+ \norm{ \Grad\eta_k \cdot \vel_1 \Grad\hat\varphi_{1,0}^k }_{\rmL^2\prt{0,T;\rmL^2\ofbbT}} \\
			&\leq C \norm{\hat\varphi_{1,0}^k}_{\rmL^\infty\ofbbT} \norm{ \vel_1 }_{\rmL^2\prt{0,T;\rmH^1\ofbbT}} 
			+ C \norm{\Grad\hat\varphi_{1,0}^k}_{\rmL^4\ofbbT} \norm{ \vel_1 }_{\rmL^2\prt{0,T;\rmL^4\ofbbT}} 
			\leq C \norm{\hat\varphi_{1,0}^k}_{\rmW^{1,4}\ofbbT} \norm{ \vel_1 }_{\rmL^2\prt{0,T;\rmH^1\ofbbT}} 
		\end{align*}
		and 
		\begin{align*}
			\norm{ \hat\varphi_{1,0}^k \Grad\eta_k \cdot \vel_1 }_{\rmH^1\prt{0,T;\rmH^{-1}\ofbbT}} 
			&\leq \norm{ \hat\varphi_{1,0}^k \Grad\eta_k \cdot \vel_1 }_{\rmL^2\prt{0,T;\rmH^{-1}\ofbbT}} 
			+ \norm{ \hat\varphi_{1,0}^k \Grad\eta_k \cdot \delt\vel_1 }_{\rmL^2\prt{0,T;\rmH^{-1}\ofbbT}} \\
			&\leq C \norm{ \hat\varphi_{1,0}^k \Grad\eta_k \cdot \vel_1 }_{\rmL^2\prt{0,T;\rmL^2\ofbbT}} 
			+ C \norm{\hat\varphi_{1,0}^k}_{\rmW^{1,4}\ofbbT} \norm{ \delt\vel_1 }_{\rmL^2\prt{0,T;\rmH^{-1}\ofbbT}} \\
			&\leq C \norm{\hat\varphi_{1,0}^k}_{\rmW^{1,4}\ofbbT} 
			\bigprt{ \norm{ \vel_1 }_{\rmL^2\prt{0,T;\rmL^2\ofbbT}} + \norm{ \delt\vel_1 }_{\rmL^2\prt{0,T;\rmH^{-1}\ofbbT}} }.
		\end{align*}
		Furthermore, the fourth component of $R_2^k$ has the same structure as the third. 
		
		Altogether, we have established that $R_2^k$ is estimated by
		\begin{align*}
			&\bignorm{ \big[\mathcal{A}_T\prt{\hat\varphi_{1,0}^k}, \eta_k \big] (\z) }_{\varUpsilon_T} \\
			&\leq C \norm{ \hat\varphi_{1,0}^k}_{\rmW^{1,4}\ofbbT} 
			\Bigprt{ \sum_{j=1,2} \bigprt{ \norm{\vel_j}_{\rmL^2 \prt{0,T;\rmH^1\ofbbTd}} + \norm{ \delt\vel_j }_{\rmL^2\prt{0,T;\rmH^{-1}\ofbbT}} }
			+ \norm{\varphi_1}_{\rmL^2\prt{0,T;\rmH^2\ofbbT}} + \norm{p}_{\rmL^2 \prt{0,T;\rmL^2\ofbbT}} },
		\end{align*}
		which shows that $R_2^k\in \varUpsilon_T$.
		
		In summary, since we naturally have $\norm{\mathfrak F^k}_{\varUpsilon_T} \leq \norm{\mathfrak F}_{\varUpsilon_T}$, the right-hand side $\mathfrak G^k$ of \eqref{Lin:eq:localized_abstr_eq} is controlled by
		\begin{align*}
			\norm{ \mathfrak G^k }_{\varUpsilon_T}
			&\leq C \norm{ \hat\varphi_{1,0}^k}_{\rmW^{1,4}\ofbbT} 
			\Bigprt{ \sum_{j=1,2} \bigprt{ \norm{\vel_j}_{\rmL^2 \prt{0,T;\rmH^1\ofbbTd}} + \norm{ \delt\vel_j }_{\rmL^2\prt{0,T;\rmH^{-1}\ofbbT}} } \\
				&\qquad\qquad\qquad\qquad+ \norm{\varphi_1}_{\rmL^2\prt{0,T;\rmH^2\ofbbT}} + \norm{p}_{\rmL^2 \prt{0,T;\rmL^2\ofbbT}} }
			+ \norm{\mathfrak F^k}_{\varUpsilon_T} 
			\leq \hat C \norm{\mathfrak F^k }_{\varUpsilon_T},
		\end{align*}
		with $\hat C=\hat C(\norm{\tphi1}_{\rmH^2\ofbbT}, T)$, where we used the dependence of the solution in the respective norms on the data in $\varUpsilon_T$, cf.~Proposition~\ref{Lin:prop:ex_gen_coeff}.
		
		Eventually, we are in a position to apply Proposition~\ref{Lin:prop:perturbation} to the localized equation~\eqref{Lin:eq:localized_abstr_eq}. In light of the closeness condition~\eqref{Lin:eq:hatphik_close_to_const}, we obtain that $\mathcal{A}_T(\hat\varphi_{1,0}^k) \colon Z_T \to \varUpsilon_T$ is invertible. Thus, there exists a unique solution $\mathcal{A}_T(\hat\varphi_{1,0}^k)^{-1} (\mathfrak G^k) = \z^k \in Z_T$.
		With $\{\eta_k\}_{k=1}^N$ being a partition of unity, we conclude
		\begin{align*}
			\norm{\z}_{Z_T} 
			= \Bignorm{\sum_{k=1}^N \eta_k \z}_{Z_T} 
			\leq \tilde C \sum_{k=1}^N \norm{\z^k}_{Z_T},
		\end{align*}
		with a constant $\tilde C>0$ that depends on $\{\eta_k\}_{k=1}^N$. This verifies the higher regularity of~$(\vel_1,\vel_2,\varphi_1,p)=\z\in Z_T$ as claimed .
	\end{proof}

	\subsection{Invertibility of $\mathcal{L}_T$ and Uniform Bound for $\mathcal{L}_T^{-1}$}

    In the final part of this section, we complete the proof of Proposition~\ref{mr:prop:L_invertible}. In fact, the invertibility of $\mathcal{L}_T$ is already established. It remains to verify that the inverse of $\mathcal{L}_T$ is bounded uniformly with respect to the time~$T$. 

    Beforehand, for the sake of completeness, we restate Proposition~\ref{Lin:prop:localization} using the notation of Section~\ref{sec:mr}.
	
	\begin{lemma}[Invertibility of $\mathcal{L}_T$]
		\label{Lin:lemma:L_invertible}
        Let $\vel_{j,0}\in\rmH^1(\bbT)^d$, $j=1,2$, and $\varphi_{1,0}\in\rmH^2(\bbT)$ with~$\varphi_{1,0}\in(0,1)$ in $\bbT$ be given such that~$\Div(\varphi_{1,0}\vel_{1,0} + (1-\varphi_{1,0}) \vel_{2,0}) = 0$. Then, for every~$0<T<\infty$, the linear operator $\mathcal{L}_T \colon X_T\to Y_T$ is invertible. 
	\end{lemma}

	\begin{proof}
		In Proposition~\ref{Lin:prop:localization}, choosing $\tphi1=\varphi_{1,0},$ we have already established the claimed invertibility.
	\end{proof}

	Now, given $T_0>0$, we prove that $\mathcal{L}_T^{-1}$ is bounded independently of $0<T<T_0$, for which we proceed as in \cite[Lemma~9.3]{AbelsHaselboeck2026}.
	
	\begin{lemma}
        \label{Lin:lemma:inverse_of_L_unif_bdd}
		Let $\vel_{j,0}\in\rmH^1(\bbT)^d$, $j=1,2$, and $\varphi_{1,0}\in\rmH^2(\bbT)$ with~$\varphi_{1,0}\in(0,1)$ in $\bbT$ be given such that~$\Div(\varphi_{1,0}\vel_{1,0} + (1-\varphi_{1,0}) \vel_{2,0}) = 0$. Then, for any fixed $0<T_0<\infty$, there exists a constant $C_{\mathcal{L}_{T_0}^{-1}}>0$ independent of~$T$ such that
		\begin{align*}
			\bignorm{ \mathcal{L}_T^{-1}(\F) }_{Z_T}
			\leq C_{\mathcal{L}_{T_0}^{-1}} \Bigprt{ \norm{\F}_{Y_T} + \sum_{j=1,2}\norm{\vel_{j,0}}_{\rmH^1\ofbbTd} + \norm{\varphi_{1,0}}_{\rmH^2\ofbbT} }
		\end{align*}
		and 
		\begin{align*}
			\bignorm{ \mathcal{L}_T^{-1}(\F) - \mathcal{L}_T^{-1}(\G) }_{Z_T}
			\leq C_{\mathcal{L}_{T_0}^{-1}} \norm{ \F-\G }_{Y_T}
		\end{align*}
		for all $0<T\leq T_0$ and all $\F,\G\in Y_T$.
	\end{lemma}
	
	\begin{proof}
		We start by preparing the proof of the second estimate. 
		For any $0<T<\infty$, the difference~$\mathcal{L}_T^{-1}(\F) - \mathcal{L}_T^{-1}(\G) \eqqcolon \z \eqqcolon (\vel_1,\vel_2,\varphi_1,  p)$ solves the problem with homogeneous initial data, i.e.,
		\begin{alignat*}{2}
			\mathcal{L}_T(\z) &= \F-\G &&\quad\text{in }(0,T)\times\bbT, \\
			(\vel_1,\vel_2,\varphi_1)\vert_{t=0} &= (0,0,0) &&\quad\text{in }\bbT.
		\end{alignat*}
		Due to Lemma~\ref{Lin:lemma:L_invertible}, this problem has a unique solution, and, consequently, there exists some~$C(T)>0$ (which coincides with the operator norm of $\mathcal{L}_T^{-1}$) such that
		\begin{align}
			\label{Lin:eq:inverse_of_L_Lipschitz}
			\norm{ \mathcal{L}_T^{-1}(\F) - \mathcal{L}_T^{-1}(\G) }_{Z_T}
			= \norm{ \mathcal{L}_T^{-1}(\F-\G) }_{Z_T}
			\leq C(T) \norm{ \F-\G }_{Y_T}.
		\end{align}
		
		Now we prove the claimed estimates independent of $T$, beginning with the first one. To~this end, let some arbitrary~$\F=(\f_1,\f_2,g_1,g_2)\in Y_T$ be given, and we define the extension~$\tilde\F=(\tilde\f_1, \tilde\f_2,\tilde g_1, \tilde g_2)\in Y_{T_0}$ by
		\begin{align*}
			\tilde\f_j(t) \coloneqq
			\begin{cases}
				f(t), &t\in[0,T], \\
				0, &t\in(T,T_0),
			\end{cases}
			\quad j=1,2,
			\qquad
			\tilde g_1 \coloneqq E^3(g_1), 
			\qquad \tilde g_2 \coloneqq E^4(g_2),
		\end{align*}
		where $E^i\colon Y_T^i\to Y_{T_0}^i$, $i=3,4$, is the extension operator from Lemma \ref{mr:lemma:extension_op}.
		Since, by Lemma~\ref{Lin:lemma:L_invertible}, the operator $\mathcal{L}_T\colon X_T\to Y_T$ is invertible for every $0<T<\infty$, we find unique solutions $\z\in X_T$ to $\mathcal{L}_T(\z)=\F$ on $(0,T)$ and $\tilde\z\in X_{T_0}$ to $\mathcal{L}_{T_0}(\tilde\z)=\tilde\F$ on $(0,T_0)$. From~$\z$ and $\tilde\z$ solving the same equation on $(0,T)$, we conclude by uniqueness that $\tilde\z\vert_{(0,T)} = \z$. Therefore, it holds (with a constant $C>0$ independent of $T$, cf.~Lemma \ref{mr:lemma:extension_op})
		\begin{align*}
			\bignorm{ \mathcal{L}_T^{-1}(\F) }_{Z_T}
			&= \norm{\z}_{Z_T}
			\leq C \norm{\tilde\z}_{Z_{T_0}}
			= C \bignorm{ \mathcal{L}_{T_0}^{-1}(\F) }_{Z_{T_0}} \\
			&\leq C_{\mathcal{L}_{T_0}^{-1}} \Bigprt{ \norm{\F}_{Y_T} + \sum_{j=1,2}\norm{\vel_{j,0}}_{\rmH^1\ofbbTd} + \norm{\varphi_{1,0}}_{\rmH^2\ofbbT} }
		\end{align*}
        This proves the first estimate, and, invoking \eqref{Lin:eq:inverse_of_L_Lipschitz}, the second one follows analogously.
	\end{proof}

    Now we are finally in a position to finish the proof of Proposition~\ref{mr:prop:L_invertible}, which states the invertibility of $\mathcal{L}_T$, along with the uniform bound of its inverse.

    \begin{proof}[Proof of Proposition~\ref{mr:prop:L_invertible}]
        The combination of Lemma~\ref{Lin:lemma:L_invertible} and Lemma~\ref{Lin:lemma:inverse_of_L_unif_bdd} completes the claim.
    \end{proof}

	\section{Local Lipschitz Continuity of $\mathcal{F}_T$} \label{sec:Lip} 
    
    This section is devoted to proving the Lipschitz continuity of the nonlinear operator~$\mathcal{F}_T$  stated in Proposition~\ref{mr:prop:F_Lipschitz}. We begin by collecting several embedding properties of the components of the solution space $Z_T$, which will play a key role in the verification of the required Lipschitz estimates. Throughout, we repeatedly make use of the interpolation results from Subsection~\ref{subsec:prelim_results}.
    
	\medskip
	\noindent\textbf{Embeddings for $Z_T^j$, $j=1,2$.}
	From the definition of $Z_T^j$, see \eqref{mr:eq:spaces_Z_T}, and the interpolation embedding~\eqref{prelim:interpol:BUC}, we deduce
	\begin{align}
		\label{Lip:emb:XT1_1}
		Z_T^j 
		\hookrightarrow \BUC\Bigprt{ [0,T]; \bigprt{ \rmL^2(\bbT)^d, \rmH^2(\bbT)^d}_{\frac12, 2} } 
		= \BUC\bigprt{ [0,T]; \rmH^1(\bbT)^d }.
	\end{align}
	In particular, this implies for $\vel\in Z_T^j$, that
	\begin{align}
		\nabla\vel 
		&\in \rmL^\infty\prt{0,T;\rmL^2(\bbT)^{d\times d}} \cap \rmL^2\prt{0,T;\rmL^6(\bbT)^{d\times d}}
		\hookrightarrow \rmL^4 \prt{0,T;\rmL^3(\bbT)^{d\times d}}.
		\label{Lip:emb:nablavel_2}
	\end{align}
	Moreover, standard Sobolev embedding gives
	\begin{align}
		\label{Lip:emb:XT1_2}
		Z_T^j 
		\hookrightarrow \rmH^1\prt{0,T;\rmL^2(\bbT)^d}
		\hookrightarrow \rmC^{\frac12}\bigprt{ [0,T]; \rmL^2(\bbT)^d }.
	\end{align}
	Combining \eqref{Lip:emb:XT1_1} and \eqref{Lip:emb:XT1_2}, the interpolation result \eqref{prelim:interpol:Hölder} yields embeddings into Hölder spaces with values in interpolation spaces between $\rmL^2(\bbT)^d$ and $\rmH^1(\bbT)^d$, in particular
	\begin{align}
		Z_T^j 
		&\hookrightarrow \rmC^{\frac{\theta}{2}}\bigprt{ [0,T]; \rmH^{1-\theta}(\bbT)^d } \qquad \text{for all } \theta\in(0,1),
		\label{Lip:emb:XT1_3} \\
		Z_T^j 
		&\hookrightarrow \rmC^{\frac{3}{2q}-\frac14}\bigprt{ [0,T]; \rmL^q(\bbT)^d } \qquad \text{for all } q\in(2,6),
		\label{Lip:emb:XT1_4}
	\end{align}
	where \eqref{Lip:emb:XT1_4} is deduced from \eqref{Lip:emb:XT1_3} in view of $\rmH^{1-\theta}(\bbT) \hookrightarrow \rmL^q(\bbT)$ for $q=\frac{6}{1+2\theta}$, $q\in(2,6)$, or equivalently for $\theta=\frac{3}{q}-\frac12$, $\theta\in(0,1)$.

	\medskip
	\noindent\textbf{Embeddings for $Z_T^3$.}
	Similarly, the definition of $Z_T^3$ and \eqref{prelim:interpol:BUC} imply
	\begin{align}
		\label{Lip:emb:XT3_1}
		Z_T^3 
		\hookrightarrow \BUC\Bigprt{ [0,T]; \bigprt{ \rmH^1(\bbT), \rmH^3(\bbT)}_{\frac12, 2} } 
		= \BUC\bigprt{ [0,T]; \rmH^2(\bbT) },
	\end{align}
	and therefore, we additionally obtain
	\begin{align}
		Z_T^3
		&\hookrightarrow \BUC\bigprt{[0,T];\rmC^{\frac12}(\bbT)}.
		\label{Lip:emb:XT3_2}
	\end{align}
	Furthermore, we derive
	\begin{align}
		\label{Lip:emb:XT3_3}
		Z_T^3 
		\hookrightarrow \rmH^1\prt{0,T;\rmH^1(\bbT)}
		\hookrightarrow \rmC^{\frac12}\bigprt{ [0,T]; \rmH^1(\bbT) }.
	\end{align}
	In view of \eqref{Lip:emb:XT3_1} together with \eqref{Lip:emb:XT3_3}, \eqref{prelim:interpol:Hölder} provides embeddings into Hölder spaces with values in the following interpolation spaces between $\rmH^1(\bbT)$ and $\rmH^2(\bbT)$:
	\begin{align}
		Z_T^3
		&\hookrightarrow \rmC^{\frac14}\bigprt{ [0,T]; \rmL^\infty(\bbT) },
		\label{Lip:emb:XT3_4} \\
		Z_T^3
		&\hookrightarrow \rmC^{\frac{\theta}{2}}\bigprt{ [0,T]; \rmH^{2-\theta}(\bbT) } \qquad \text{for all } \theta\in(0,1),
		\label{Lip:emb:XT3_5} \\
		Z_T^3
		&\hookrightarrow \rmC^{\frac{3}{2q}-\frac14}\bigprt{ [0,T]; \rmW^{1,q}(\bbT) } \qquad \text{for all } q\in(2,6).
		\label{Lip:emb:XT3_6}
	\end{align}
	Here, we used Agmon's inequality for \eqref{Lip:emb:XT3_4}, while \eqref{Lip:emb:XT3_6} followed directly from \eqref{Lip:emb:XT3_5} in light of $\rmH^{2-\theta}(\bbT) \hookrightarrow \rmW^{1,q}(\bbT)$ for $q=\frac{6}{1+2\theta}$, $q\in(2,6)$, or equivalently for $\theta=\frac{3}{q}-\frac12$, $\theta\in(0,1)$.

	The following result will be applied to the functions $\nu_j$, $\lambda_j\in\rmCb{2}(\R)$ and $F' \in \rmC^2(\R)$, which depend nonlinearly on $\varphi_1$.
	
	\medskip
	\noindent\textbf{Composition with functions in $Z_T^3$.}
	Let $f\in\rmC^2(\R)$ and $\varphi\in Z_T^3$ with $\norm{\varphi}_{Z_T^3}\leq R$ for some $R>0$. We now state in which norms $f(\varphi)$ is bounded by some constant $C(R)>0$ independent of $T$, and in which norm $f(\varphi)$ is locally Lipschitz with a constant $L(R)>0$ independent of $T$.
	
	First, due to \eqref{Lip:emb:XT3_1}, we see that $\norm{\varphi(t)}_{\rmH^2\ofbbT} \leq C(R)$ for all $t\in[0,T]$. Lemma~\ref{prelim:lemma:comp_Sobolev} implies that $\norm{f(\varphi(t))}_{\rmH^2\ofbbT} \leq C(R)$ for all $t\in[0,T]$ and therefore,
	\begin{align}
		\label{Lip:emb:f(phi)_bdd_H2}
		\norm{f(\varphi)}_{\rmL^\infty\prt{0,T; \rmH^2\ofbbT}} \leq C(R).
	\end{align}
	Analogously, Lemma~\ref{prelim:lemma:comp_Sobolev} yields for any $q>d$ that
	\begin{align}
		\label{Lip:emb:f(phi)_bdd_W1q}
		\norm{f(\varphi)}_{\rmL^\infty\prt{0,T; \rmW^{1,q}\ofbbT}} \leq C(R),
	\end{align}
	and that for all $R>0$, there exists a constant $L(R)>0$ such that
	\begin{align}
		\label{Lip:emb:f(phi)_Lipsch}
		\norm{f(\varphi)-f(\psi)}_{\rmL^\infty\prt{0,T; \rmW^{1,q}\ofbbT}} 
		\leq L(R) \norm{\varphi-\psi}_{\rmL^\infty\prt{0,T; \rmW^{1,q}\ofbbT}}
	\end{align}
	for all $\varphi,\psi\in \rmL^\infty\prt{0,T; \rmW^{1,q}(\bbT)}$ with $\norm{\varphi}_{\rmL^\infty\prt{0,T; \rmW^{1,q}\ofbbT}}, \norm{\psi}_{\rmL^\infty\prt{0,T; \rmW^{1,q}\ofbbT}} \leq R$.

	\medskip
	Now we are in a position to carry out the proof of Proposition~\ref{mr:prop:F_Lipschitz}, that is, the local Lipschitz continuity of $\mathcal{F}_T$.
	
	\begin{proof}[Proof of Proposition~\ref{mr:prop:F_Lipschitz}]
		In order to prove local Lipschitz continuity of $\mathcal{F}_T$, defined in~\eqref{mr:eq:F_T}, we study the components $\mathcal{F}_T^1$ and $\mathcal{F}_T^3$ separately. The arguments for the components $\mathcal{F}_T^2$ and $\mathcal{F}_T^4$, respectively, are exactly the same. 
		
		Hence, for fixed $R>0$, let some arbitrary $(\vel_1,\vel_2,\varphi_1,p),(\w_1,\w_2,\psi_1,q) \in X_T$ be given that satisfy $\norm{ (\vel_1,\vel_2,\varphi_1,p)}_{X_T}, \norm{(\w_1,\w_2,\psi_1,q) }_{X_T} \leq R$. For the definition of $Y_T$, see \eqref{mr:eq:spaces_Y_T}.

		\textit{Ad $\mathcal{F}_T^1$.} The first component of the expression to be analyzed consists of the following terms, where we write $\psi_{1,0}=\varphi_{1,0}$ for the initial value of $\psi_1$:
		\begin{align*}
			&\norm{ \mathcal{F}_T^1(\vel_1,\vel_2,\varphi_1,p) - \mathcal{F}_T^1(\w_1,\w_2,\psi_1,q) }_{Y_T^1} \\
			&\leq \bignorm{ \bigprt{\rho_1(\varphi_1)-\rho_1(\varphi_{1,0})}\delt\vel_1   - \bigprt{\rho_1(\psi_1)-\rho_1(\psi_{1,0})}\delt\w_1}_{Y_T^1} \\
			&\quad + \bignorm{ \bigprt{\rho_1(\varphi_1)\vel_1\cdot\nabla} \vel_1 - \bigprt{\rho_1(\psi_1)\w_1\cdot\nabla} \w_1}_{Y_T^1} \\
			&\quad+ \bignorm{ 2\Div\bigprt{ \bigprt{\nu_1(\varphi_1)-\nu_1(\varphi_{1,0})} \D\vel_1 } - 2\Div\bigprt{ \bigprt{\nu_1(\psi_1)-\nu_1(\psi_{1,0})} \D\w_1} }_{Y_T^1} \\ 
			&\quad+	\bignorm{ \Div\bigprt{ \bigprt{\lambda_1(\varphi_1)-\lambda_1(\varphi_{1,0})} \Div\vel_1 \Id }- \Div\bigprt{ \bigprt{\lambda_1(\psi_1)-\lambda_1(\psi_{1,0})} \Div\w_1 \Id } }_{Y_T^1} \\
			&\quad+	\bignorm{ \prt{\varphi_1-\varphi_{1,0}}\Grad p - \prt{\psi_1-\psi_{1,0}}\Grad q }_{Y_T^1}
			+ \bignorm{\varphi_1 \Grad F'(\varphi_1) - \psi_1 \Grad F'(\psi_1) }_{Y_T^1} \\
			&\quad+ \bignorm{  \prt{\varphi_1-\varphi_{1,0}} \Grad\Lap\varphi_1 -  \prt{\psi_1-\psi_{1,0}} \Grad\Lap\psi_1 }_{Y_T^1}
			+ \bignorm{ \Rcal(\varphi_1)(\vel_2-\vel_1) - \Rcal(\psi_1)(\w_2-\w_1) }_{Y_T^1} \\
			&\eqqcolon I + II + III + IV + V + VI + VII + VIII.
		\end{align*}

		\textit{Ad $\mathcal{F}_T^1.I$.} Recalling $\psi_{1,0}=\varphi_{1,0}$ due to the definition of $X_T^3$, cf.~\eqref{mr:eq:spaces_X_T}, as well as the linearity of $\rho_1$ and the embedding \eqref{Lip:emb:XT3_4}, we find
			\begin{align*}
			I
			&=\bignorm{ \bigprt{\rho_1(\varphi_1)-\rho_1(\varphi_{1,0})}\delt\vel_1   - \bigprt{\rho_1(\psi_1)-\rho_1(\psi_{1,0})}\delt\w_1}_{\rmL^2\prt{0,T;\rmL^2}} \\
			&\leq \bignorm{\rho_1(\varphi_1)-\rho_1(\varphi_{1,0}) }_{\rmL^\infty\prt{0,T; \rmL^\infty}} \bignorm{ \delt\prt{\vel_1-\w_1} }_{\rmL^2\prt{0,T; \rmL^2}} \\
			&\quad+ \bignorm{\rho_1(\varphi_1)-\rho_1(\psi) }_{\rmL^\infty\prt{0,T; \rmL^\infty}} \norm{ \delt\w_1 }_{\rmL^2\prt{0,T; \rmL^2}} \\
			&\leq C T^{\frac14} \norm{ \varphi_1 }_{\rmC^{\frac14}\prt{[0,T]; \rmL^\infty}} \norm{\vel_1-\w_1}_{\rmH^1\prt{0,T; \rmL^2}} 
			+ C T^{\frac14} \norm{ \varphi_1-\psi_1 }_{\rmC^{\frac14}\prt{[0,T]; \rmL^\infty}} \norm{\w_1}_{\rmH^1\prt{0,T; \rmL^2}} \\
			&\leq C(R) T^{\frac14} \norm{\vel_1-\w_1}_{Z_T^1}
			+ C(R) T^{\frac14} \norm{\varphi_1-\psi_1}_{Z_T^3}.
		\end{align*}

		\textit{Ad $\mathcal{F}_T^1.II$.} To estimate this term, we use again $\psi_{1,0}=\varphi_{1,0}$, the linearity of $\rho_1$ together with \eqref{Lip:emb:XT3_2}, and the embeddings \eqref{Lip:emb:XT1_1} and \eqref{Lip:emb:nablavel_2}. This leads to
		\begin{align*}
			II
			&= \bignorm{ \bigprt{\rho_1(\varphi_1)\vel_1\cdot\nabla} \vel_1 - \bigprt{\rho_1(\psi_1)\w_1\cdot\nabla} \w_1}_{\rmL^2\prt{0,T;\rmL^2}} \\
			&\leq \bignorm{ \rho_1(\varphi_1)-\rho_1(\psi_1) }_{\rmL^\infty\prt{0,T; \rmL^\infty}} 
				\norm{ \vel_1 }_{\rmL^4\prt{0,T; \rmL^6}} 
				\norm{ \nabla \vel_1 }_{\rmL^4\prt{0,T; \rmL^3}} \\
			&\quad+ \bignorm{ \rho_1(\psi_1) }_{\rmL^\infty\prt{0,T; \rmL^\infty}} 
				\norm{ \w_1-\vel_1 }_{\rmL^4\prt{0,T; \rmL^6}} 
				\norm{ \nabla \vel_1 }_{\rmL^4\prt{0,T; \rmL^3}} \\
			&\quad+ \bignorm{ \rho_1(\psi_1) }_{\rmL^\infty\prt{0,T; \rmL^\infty}} 
				\norm{ \w_1 }_{\rmL^4\prt{0,T; \rmL^6}} 
				\norm{ \nabla (\vel_1-\w_1) }_{\rmL^4\prt{0,T; \rmL^3}} \\
			&\leq C\norm{ \varphi_1-\psi_1 }_{Z_T^3} 
				T^{\frac14} \norm{ \vel_1 }_{\rmL^\infty\prt{0,T; \rmH^1}} 
				\norm{ \vel_1 }_{Z_T^1} 
			+  C\norm{ \psi_1 }_{Z_T^3} 
				T^{\frac14} \norm{ \w_1-\vel_1 }_{\rmL^\infty\prt{0,T; \rmH^1}} 
				\norm{ \vel_1 }_{Z_T^1} \\
			&\quad+ C\norm{ \psi_1 }_{Z_T^3} 
				T^{\frac14} \norm{ \w_1 }_{\rmL^\infty\prt{0,T; \rmH^1}} 
				\norm{ \vel_1-\w_1 }_{Z_T^1} \\
			&\leq C(R) T^{\frac14} \norm{ \varphi_1-\psi_1 }_{Z_T^3}
			+ C(R) T^{\frac14} \norm{\vel_1-\w_1}_{Z_T^1}.
		\end{align*}

		\textit{Ad $\mathcal{F}_T^1.III$.} In light of $\psi_{1,0}=\varphi_{1,0}$, we start by splitting $III$ into the two following terms
		\begin{align*}
			III
			&= \bignorm{ 2\Div\bigprt{ \bigprt{\nu_1(\varphi_1)-\nu_1(\varphi_{1,0})} \D\vel_1 } - 2\Div\bigprt{ \bigprt{\nu_1(\psi_1)-\nu_1(\psi_{1,0})} \D\w_1} }_{Y_T^1} \\ 
			&\leq \bignorm{ 2\Div\bigprt{ \bigprt{\nu_1(\varphi_1)-\nu_1(\varphi_{1,0})} \D\prt{ \vel_1-\w_1} } }_{Y_T^1} 
			+ \bignorm{ 2\Div\bigprt{ \bigprt{\nu_1(\varphi_1)-\nu_1(\psi_1)} \D\w_1 } }_{Y_T^1} \\
			&\eqqcolon III.1 + III.2,
		\end{align*}
		where the first one is further broken down to
		\begin{align*}
			III.1
			&\leq \bignorm{ 2\D(\vel_1-\w_1) \Grad\bigprt{\nu_1(\varphi_1)-\nu_1(\varphi_{1,0})} }_{Y_T^1} 
			+ \bignorm{ \bigprt{\nu_1(\varphi_1)-\nu_1(\varphi_{1,0})} \Lap(\vel_1-\w_1) }_{Y_T^1} \\ 
			&\quad+ \bignorm{ \bigprt{\nu_1(\varphi_1)-\nu_1(\varphi_{1,0})} \Grad\Div(\vel_1-\w_1) }_{Y_T^1}  
			\eqqcolon III.1.1 + III.1.2 + III.1.3.
		\end{align*}
		Studying the first summand, the local Lipschitz continuity of $\nu_1$, cf.~\eqref{Lip:emb:f(phi)_Lipsch}, along with the embedding~\eqref{Lip:emb:XT3_6} implies
		\begin{align*}
			III.1.1
			&\leq C\bignorm{ \D(\vel_1-\w_1) }_{\rmL^2\prt{0,T; \rmL^4}} 
			\bignorm{ \Grad\bigprt{\nu_1(\varphi_1)-\nu_1(\varphi_{1,0})} }_{\rmL^\infty\prt{0,T; \rmL^4}} \\ 
			&\leq C(R)\norm{ \vel_1-\w_1 }_{\rmL^2\prt{0,T; \rmH^2}} 
			\norm{ \varphi_1-\varphi_{1,0} }_{\rmL^\infty\prt{0,T; \rmW^{1,4}}} \\ 
			&\leq C(R)\norm{ \vel_1-\w_1 }_{\rmL^2\prt{0,T; \rmH^2}} 
			T^{\frac18} \norm{ \varphi_1 }_{\rmC^{\frac18}\prt{[0,T]; \rmW^{1,4}}}  
			\leq C(R) T^{\frac18} \norm{\vel_1-\w_1}_{Z_T^1}.
		\end{align*}
		For the second term, applying Agmon's inequality, \eqref{Lip:emb:f(phi)_Lipsch}, \eqref{Lip:emb:XT3_6}, and \eqref{Lip:emb:f(phi)_bdd_H2}, we end up with
		\begin{align*}
			III.1.2
			&\leq \bignorm{ \nu_1(\varphi_1)-\nu_1(\varphi_{1,0}) }_{\rmL^\infty\prt{0,T; \rmL^\infty}} 
			\bignorm{ \Lap(\vel_1-\w_1) }_{\rmL^2\prt{0,T; \rmL^2}} \\ 
			&\leq C \bignorm{ \nu_1(\varphi_1)-\nu_1(\varphi_{1,0}) }_{\rmL^\infty\prt{0,T; \rmH^1}}^{\frac12}
			\bignorm{ \nu_1(\varphi_1)-\nu_1(\varphi_{1,0}) }_{\rmL^\infty\prt{0,T; \rmH^2}}^{\frac12}
			\norm{ \vel_1-\w_1 }_{\rmL^2\prt{0,T; \rmH^2}} \\ 
			&\leq C(R) T^{\frac{1}{16}} \norm{ \varphi_1 }_{\rmC^{\frac18}\prt{[0,T]; \rmW^{1,4}}}^{\frac12} 
			\bignorm{ \nu_1(\varphi_1)-\nu_1(\varphi_{1,0}) }_{\rmL^\infty\prt{0,T; \rmH^2}}^{\frac12}
			\norm{ \vel_1-\w_1 }_{\rmL^2\prt{0,T; \rmH^2}} \\ 
			&\leq C(R) T^{\frac{1}{16}} \norm{\vel_1-\w_1}_{Z_T^1},
		\end{align*}
		with $III.1.3$ being handled analogously. Estimating $III.2$ accordingly to $III.1$ reveals
		\begin{align*}
			III.2 
			\leq C(R) T^{\frac{1}{16}} \norm{ \varphi_1-\psi_1 }_{Z_T^3}.
		\end{align*}
	
		\textit{Ad $\mathcal{F}_T^1.IV$.} This expression may be estimated exactly like $\mathcal{F}_T^1.III$.

		\textit{Ad $\mathcal{F}_T^1.V$.} Since $\Grad p,\Grad q\in \rmL^2\prt{0,T; \rmL^2(\bbT)^d}$, we treat this term like $\mathcal{F}_T^1.I$, obtaining
		\begin{align*}
			V 
			&\leq C(R) T^{\frac14} \norm{p-q}_{Z_T^4}
			+ C(R) T^{\frac14} \norm{\varphi_1-\psi_1}_{Z_T^3}.
		\end{align*}

        \textit{Ad $\mathcal{F}_T^1.VI$.} 
        Here, we again use the fact that $\varphi_{1,0}$ and $\psi_{1,0}$ coincide. Combining \eqref{Lip:emb:XT3_2} and~\eqref{Lip:emb:f(phi)_bdd_W1q} on the one hand, and \eqref{Lip:emb:f(phi)_Lipsch} and \eqref{Lip:emb:XT3_6} on the other hand gives
        \begin{align*}
			VI
			&= \bignorm{ \varphi_1 \Grad F'(\varphi_1) - \psi_1 \Grad F'(\psi_1) }_{\rmL^2\prt{0,T; \rmL^2}} \\
			&\leq \norm{ \varphi_1-\psi_1 }_{\rmL^2\prt{0,T; \rmL^4}\!}
			\norm{ \Grad F'(\varphi_1) }_{\rmL^\infty\prt{0,T; \rmL^4}\!} 
            + \norm{ \psi_1 }_{\rmL^2\prt{0,T; \rmL^4}\!} 
			\norm{ \Grad\prt{F'(\varphi_1)\! -\! F'(\psi_1)} }_{\rmL^\infty\prt{0,T; \rmL^4}\!} \\
			&\leq C T^{\frac12} \norm{ \varphi_1-\psi_1 }_{\rmL^\infty\prt{0,T; \rmL^4}}
			\bignorm{ F'(\varphi_1) }_{\rmL^\infty\prt{0,T; \rmW^{1,4}}} \\
            &\quad+ C \norm{ \psi_1 }_{\rmL^2\prt{0,T; \rmH^3}} 
			T^{\frac18} \norm{ \varphi_1-\psi_1 }_{\rmC^{\frac18}\prt{[0,T]; \rmW^{1,4}}} \\
			&\leq C(R) T^{\frac12} \norm{\varphi_1-\psi_1}_{Z_T^3}
            + C(R) T^{\frac18} \norm{\varphi_1-\psi_1}_{Z_T^3}.
		\end{align*}

		\textit{Ad $\mathcal{F}_T^1.VII$.} Noting that $\Grad \Lap\varphi_1,\Grad\Lap\psi_1\in \rmL^2\prt{0,T; \rmL^2(\bbT)^d}$, this term can be treated analogously to $\mathcal{F}_T^1.I$, which leads to
		\begin{align*}
			VII
			&\leq C(R) T^{\frac14} \norm{\varphi_1-\psi_1}_{Z_T^3}.
		\end{align*}
	
		\textit{Ad $\mathcal{F}_T^1.VIII$.} Eventually, we study the last summand of $\mathcal{F}_T^1$, namely
		\begin{align*}
			VIII
			&= \bignorm{ \Rcal(\varphi_1)(\vel_2-\vel_1) - \Rcal(\psi_1)(\w_2-\w_1) }_{Y_T^1} \\
			&\leq \bignorm{ \Rcal(\varphi_1)(\vel_2-\w_2) }_{Y_T^1} 
			+ \bignorm{ \bigprt{\Rcal(\varphi_1) - \Rcal(\psi_1)} \w_2 }_{Y_T^1} 
			+ \bignorm{ \Rcal(\varphi_1)(\w_1-\vel_1) }_{Y_T^1} \\
			&\quad+ \bignorm{ \bigprt{\Rcal(\psi_1) - \Rcal(\varphi_1) } \w_1 }_{Y_T^1} 
			\eqqcolon VIII.1 + VIII.2 + VIII.3 + VIII.4.
		\end{align*}
		Here, we use the local Lipschitz property of $\mathcal{R}$ from Assumption~\eqref{ass:R}, together with the embeddings \eqref{Lip:emb:XT3_2} and \eqref{Lip:emb:XT1_1}, to estimate the third term
	 	\begin{align*}
	 		VIII.3 
	 		&\leq \bignorm{ \Rcal(\varphi_1) }_{\rmL^\infty\prt{0,T; \rmL^\infty}} 
	 		\norm{ \w_1-\vel_1 }_{\rmL^2\prt{0,T; \rmL^2}} \\
	 		&\leq C(R) \bignorm{ \varphi_1 }_{\rmL^\infty\prt{0,T; \rmL^\infty}} 
	 		T^{\frac12}	\norm{ \w_1-\vel_1 }_{\rmL^\infty\prt{0,T; \rmH^1}} 
	 		\leq C(R) T^{\frac12} \norm{ \w_1-\vel_1 }_{Z_T^1},
	 	\end{align*}
 		where the corresponding estimate holds true for $VIII.1$.
 		With the same tools, it follows
 		\begin{align*}
 			VIII.4 
 			&\leq \bignorm{ \Rcal(\varphi_1) - \Rcal(\psi_1) }_{\rmL^\infty\prt{0,T; \rmL^\infty}} 
 			\norm{ \w_1}_{\rmL^2\prt{0,T; \rmL^2}} \\
 			&\leq C(R) \norm{ \varphi_1 - \psi_1 }_{\rmL^\infty\prt{0,T; \rmL^\infty}} 
 			T^{\frac12}	\norm{ \w_1 }_{\rmL^\infty\prt{0,T; \rmH^1}} 
 			\leq C(R) T^{\frac12} \norm{ \varphi_1 - \psi_1 }_{Z_T^3},
 		\end{align*}
		and accordingly, one may treat $VIII.2$.
		
		Altogether, this proves the desired local Lipschitz estimate for $\mathcal{F}_T^1$.

		\medskip
		\textit{Ad $\mathcal{F}_T^3$.} By the definition of $Y_T^3 = \rmL^2 \prt{0,T;\rmH^1(\bbT)} \cap \rmH^1 \prt{ 0,T;\rmH^{-1}(\bbT)}$, see \eqref{mr:eq:spaces_Y_T}, we now need to estimate $\norm{ \mathcal{F}_T^3(\vel_1,\vel_2,\varphi_1,p) - \mathcal{F}_T^3(\w_1,\w_2,\psi_1,q) }_{Y_T^3}$ in two different norms, which we carry out separately, beginning with the first. We recall $\psi_{1,0}=\varphi_{1,0}$ and directly start by breaking the expression down into manageable parts
		\begin{align*}
			&\norm{ \mathcal{F}_T^3(\vel_1,\vel_2,\varphi_1,p) - \mathcal{F}_T^3(\w_1,\w_2,\psi_1,q) }_{\rmL^2\prt{0,T; \rmH^1}} \\
			&= \bignorm{ \Div\bigprt{(\varphi_1-\varphi_{1,0})\vel_1} - \Div\bigprt{(\psi_1-\psi_{1,0})\w_1}}_{\rmL^2\prt{0,T; \rmH^1}} \\
			&\leq \bignorm{ (\varphi_1-\varphi_{1,0}) \Div(\vel_1-\w_1) }_{\rmL^2\prt{0,T; \rmH^1}}
			+ \bignorm{ \Grad(\varphi_1-\varphi_{1,0}) \cdot (\vel_1-\w_1) }_{\rmL^2\prt{0,T; \rmH^1}} \\
			&\quad+ \bignorm{ (\varphi_1-\psi_1)\Div\w_1 }_{\rmL^2\prt{0,T; \rmH^1}} 
			+ \bignorm{ \Grad(\varphi_1-\psi_1)\cdot\w_1 }_{\rmL^2\prt{0,T; \rmH^1}} 
			\eqqcolon I_1 + II_1 + III_1 + IV_1.
		\end{align*}

		\textit{Ad $\mathcal{F}_T^3.I_1$.} The parts of the first term we study,
		\begin{align*}
			I_1 
			&\leq \bignorm{ (\varphi_1-\varphi_{1,0}) \Div(\vel_1-\w_1) }_{\rmL^2\prt{0,T; \rmL^2}} 
			+ \bignorm{ \Div(\vel_1-\w_1) \Grad(\varphi_1-\varphi_{1,0}) }_{\rmL^2\prt{0,T; \rmL^2}} \\
			&\quad+ \bignorm{ (\varphi_1-\varphi_{1,0}) \Grad\Div(\vel_1-\w_1) }_{\rmL^2\prt{0,T; \rmL^2}},
		\end{align*}
		behave like the terms in $\mathcal{F}_T^1.I$ and like $III.1.1$ and $III.1.2$, each in $\mathcal{F}_T^1$, respectively.

		\textit{Ad $\mathcal{F}_T^3.II_1$.} We continue with computing
		\begin{align*}
			II_1
			&\leq \bignorm{ \Grad(\varphi_1-\varphi_{1,0}) \cdot (\vel_1-\w_1) }_{\rmL^2\prt{0,T; \rmL^2}} 
			+ \bignorm{ \Grad^2(\varphi_1-\varphi_{1,0}) (\vel_1-\w_1) }_{\rmL^2\prt{0,T; \rmL^2}} \\
			&\quad+ \bignorm{ \trans{\Grad(\vel_1-\w_1)} \Grad(\varphi_1-\varphi_{1,0}) }_{\rmL^2\prt{0,T; \rmL^2}} 
			\eqqcolon II_1.1 + II_1.2 + II_1.3.
		\end{align*}
		Both $II_1.1$ and $II_1.3$ may be estimated analogously to $III.1.1$ in $\mathcal{F}_T^1$, while for $II_1.2$, \eqref{Lip:emb:XT1_4} and Agmon's inequality yield
		\begin{align*}
			II_1.2
			&\leq \bignorm{ \Grad^2\varphi_1 (\vel_1-\w_1) }_{\rmL^2\prt{0,T; \rmL^2}}
            + \bignorm{ \Grad^2\varphi_{1,0} (\vel_1-\w_1) }_{\rmL^2\prt{0,T; \rmL^2}}\\
			&\leq C \norm{ \Grad^2\varphi_1 }_{\rmL^2\prt{0,T; \rmL^4}}
			\norm{ \vel_1-\w_1 }_{\rmL^\infty\prt{0,T; \rmL^4}} 
            + C \norm{ \Grad^2\varphi_{1,0} }_{\rmL^2}
			\norm{ \vel_1-\w_1 }_{\rmL^2\prt{0,T; \rmL^\infty}} \\
			&\leq C \norm{ \varphi_1 }_{\rmL^2\prt{0,T; \rmH^3}}
			T^{\frac14} \norm{ \vel_1-\w_1 }_{\rmC^{\frac14}\prt{[0,T]; \rmL^4}} \\
            &\quad+ C \norm{ \varphi_{1,0} }_{\rmH^2}
			\norm{ \vel_1-\w_1 }_{\rmL^\infty\prt{0,T; \rmH^1}}^{\frac12} T^{\frac14} \norm{ \vel_1-\w_1 }_{\rmL^2\prt{0,T; \rmH^2}}^{\frac12}\\
			&\leq C(R) T^{\frac14} \norm{ \w_1-\vel_1 }_{Z_T^1}.
		\end{align*}

		\textit{Ad $\mathcal{F}_T^3.III_1$.} To estimate this term, one proceeds analogously to each step in $\mathcal{F}_T^3.I_1$.

		\textit{Ad $\mathcal{F}_T^3.IV_1$.} This summand is treated by the same means as $\mathcal{F}_T^3.II_1$.

        Now we consider the second norm of $Y_T^3$ and obtain
        {\allowdisplaybreaks
        \begin{align*}
			&\norm{ \mathcal{F}_T^3(\vel_1,\vel_2,\varphi_1,p) - \mathcal{F}_T^3(\w_1,\w_2,\psi_1,q) }_{\rmH^1\prt{0,T; \rmH^{-1}}} \\
			&= \bignorm{ \Div\bigprt{(\varphi_1-\varphi_{1,0})\vel_1} - \Div\bigprt{(\psi_1-\psi_{1,0})\w_1}}_{\rmH^1\prt{0,T; \rmH^{-1}}} \\
            &\leq \bignorm{ (\varphi_1-\varphi_{1,0})(\vel_1-\w_1) 
            + (\varphi_1-\psi_1)\w_1 }_{\rmL^2\prt{0,T; \rmL^2}} \\
            &\quad+ \bignorm{ (\varphi_1-\varphi_{1,0})\delt(\vel_1-\w_1) 
            + (\varphi_1-\psi_1)\delt\w_1 }_{\rmL^2\prt{0,T; \rmL^2}} \\
            &\quad+ \bignorm{ \delt\varphi_1(\vel_1-\w_1) 
            + \delt(\varphi_1-\psi_1)\w_1 }_{\rmL^2\prt{0,T; \rmL^2}}
            \coloneqq I_2 + II_2 + III_2,
		\end{align*}
        }
        where $I_2$ and $II_2$ can be estimated as $\mathcal{F}_T^1.I$. For $III_2$, using \eqref{Lip:emb:XT1_4} reveals
        \begin{align*}
			III_2
			&\leq C \norm{ \delt\varphi_1 }_{\rmL^2\prt{0,T; \rmL^4}}
			\norm{ \vel_1-\w_1 }_{\rmL^\infty\prt{0,T; \rmL^4}}
            + C \norm{ \delt(\varphi_1-\psi_1) }_{\rmL^2\prt{0,T; \rmL^4}}
			\norm{ \w_1 }_{\rmL^\infty\prt{0,T; \rmL^4}}\\
			&\leq C \norm{ \varphi_1 }_{\rmH^1\prt{0,T; \rmH^1}}
			T^{\frac14} \norm{ \vel_1-\w_1 }_{\rmC^{\frac14}\prt{[0,T]; \rmL^4}}
            + C \norm{ \varphi_1-\psi_1 }_{\rmH^1\prt{0,T; \rmH^1}}
			T^{\frac14} \norm{ \w_1 }_{\rmC^{\frac14}\prt{[0,T]; \rmL^4}}\\
			&\leq C(R) T^{\frac14} \norm{ \w_1-\vel_1 }_{Z_T^1}
            + C(R) T^{\frac14} \norm{ \varphi_1-\psi_1 }_{Z_T^3}.
		\end{align*}
		
		In summary, this shows the local Lipschitz continuity of $\mathcal{F}_T^3$. 
		
		Altogether, since $\mathcal{F}_T^2$ behaves like $\mathcal{F}_T^1$, and $\mathcal{F}_T^4$ like $\mathcal{F}_T^3$, we have verified the claimed local Lipschitz estimate
		\begin{align*}
			\norm{ \mathcal{F}_T(\vel_1,\vel_2,\varphi_1,p) - \mathcal{F}_T(\w_1,\w_2,\psi_1,q) }_{Y_T} 
			\leq C(R,T) \norm{ (\vel_1,\vel_2,\varphi_1,p) - (\w_1,\w_2,\psi_1,q) }_{Z_T}
		\end{align*}
	    with a constant satisfying $C(R,T)\to0$ as $T\to0$.
    \end{proof}

    \section*{Acknowledgments}
    
    The third author was funded by the Deutsche Forschungsgemeinschaft (DFG, German Research Foundation) Research Training Group~2339 ``Interfaces, Complex Structures, and Singular Limits in Continuum Mechanics -- Analysis and Numerics''. The support is gratefully acknowledged.

    \medskip
    \noindent \textbf{Declaration of interest and data availability statement.} 
    The authors report no conflict of interest. There is no associated data for the manuscript.

    \bibliographystyle{abbrv}
    \bibliography{AGG_LT_tenEikelder.bib}

\end{document}